\documentclass[
  11pt,
  letterpaper,
  onecolumn,
  oneside]{article}

\usepackage[margin=1 in,letterpaper]{geometry}
\usepackage[singlespacing]{setspace}
\usepackage{fancyhdr}
\usepackage[inline]{enumitem}
\usepackage{newclude}
\usepackage{caption}
\usepackage{subcaption}

\usepackage[utf8]{inputenc}
\usepackage[T1]{fontenc}
\usepackage{lmodern}
\usepackage{appendix}
\usepackage{comment}
\usepackage{graphicx}
\graphicspath{ {images/} }
\usepackage{hyperref}
\usepackage{titlesec}
\newcommand{\periodafter}[1]{#1.}
\titleformat{\subsection}[runin]
  {\normalfont\bfseries}{\thesubsection}{0.5em}{\periodafter}

\newenvironment*{StatementList}
{\begin{enumerate}[label={\alph*)}, noitemsep, nolistsep, align=left, leftmargin=*]}
{\end{enumerate}}

\usepackage{amssymb}
\usepackage{amsthm}
\usepackage{mathtools}
\usepackage{eufrak}
\usepackage{bm}
\usepackage{mathrsfs}
\usepackage{bbm}
\usepackage{tikz}
\usetikzlibrary{calc}
 \usetikzlibrary{arrows,3d}
 \makeatletter
 \tikzoption{canvas is xy plane at z}[]{%
   \def\tikz@plane@origin{\pgfpointxyz{0}{0}{#1}}%
   \def\tikz@plane@x{\pgfpointxyz{1}{0}{#1}}%
   \def\tikz@plane@y{\pgfpointxyz{0}{1}{#1}}%
   \tikz@canvas@is@plane
 }
 \makeatother 
\usetikzlibrary{mindmap}

\theoremstyle{definition}
\newtheorem{definition}{Definition}[section]

\theoremstyle{plain}
\newtheorem{theorem}[definition]{Theorem}
\newtheorem{proposition}[definition]{Proposition}
\newtheorem{lemma}[definition]{Lemma}
\newtheorem{corollary}[definition]{Corollary}

\newcommand*{\dotcup}{\,\dot\cup\,}

\DeclareMathOperator{\dom}{dom}

\DeclareMathOperator{\im}{im}

\DeclareMathOperator{\Po}{Po}

\DeclareMathOperator{\Tr}{Tr}

\DeclarePairedDelimiterX{\infdivx}[2]{(}{)}{%
  #1\;\delimsize\|\;#2%
}
\newcommand{\DKL}{D_{\textrm{KL}}\infdivx}
\newcommand{\DP}{D_{\chi^2}\infdivx}

\begin{document}

\title{Satisfiability Thresholds for Regular Occupation Problems\footnote{An extended abstract of a preliminary version of this paper has appeared at the proceedings of ICALP '19~\cite{panagiotou2019}. The results here apply to all 2-in-$k$ occupation problems, as opposed to only $k = 4$ in~\cite{panagiotou2019}.}}
\author{Konstantinos Panagiotou\footnote{The research leading to these results has received funding from the European Research Council, ERC Grant Agreement 772606–PTRCSP.} \\LMU München\\kpanagio@math.lmu.de \and Matija Pasch \\LMU München\\pasch@math.lmu.de}
\maketitle
\thispagestyle{empty}
%
%
%
%
%
%
%
\begin{abstract}
In the last two decades the study of random instances of constraint satisfaction problems (CSPs) has flourished across several disciplines, including computer science, mathematics and physics. The diversity of the developed methods, on the rigorous and non-rigorous side, has led to major advances regarding both the theoretical as well as the applied viewpoints.

Based on a ceteris paribus approach in terms of the density evolution equations known from statistical physics, we focus on a specific prominent class of
regular CSPs,
the so-called \emph{occupation problems}.
The regular $r$-in-$k$ occupation problems resemble a basis of this class.
By now, out of these CSPs only the satisfiability threshold -- the largest degree for  which the problem admits asymptotically a solution -- for the $1$-in-$k$ occupation problem has been rigorously established.
Here we determine the satisfiability threshold of the $2$-in-$k$ occupation problem for all $k$. In the proof we exploit the connection of an associated optimization problem regarding the overlap of satisfying assignements to a fixed point problem inspired by belief propagation, a message passing algorithm developed for solving such CSPs.
\end{abstract}
%
%
%
\section{Introduction}\label{s_introduction}
Inspired by the pioneering work \cite{erdos1960} of Erdős and Rényi in 1960, random discrete structures have been systematically studied in literally thousands of contributions. The initial motivation of this research was to study open problems in graph theory and combinatorics. In the following decades, however, the application of such models proved useful as a unified approach to treat a variety of problems in several fields. To mention just a few, random graphs turned out to be valuable in solving fundametal theoretical and practical problems, such as the development of error correcting codes~\cite{kudekar13}, the study of statistical inference through the stochastic block model~\cite{abbe2018}, and the establishment of lower bounds in complexity theory~\cite{galanis14,feige02}.

The results of the past years of research suggest the existence of \emph{phase transitions} in many classes of random discrete structures, i.e.~a specific value of a given model parameter at which the properties of the system in question  change dramatically.
Random constraint satisfaction problems are one specific type of such structures that tend to exhibit this remarkable property and that are of particular interest in too many areas to mention, covering complexity theory, combinatorics, statistical mechanics, artificial intelligence, biology, engineering and economics.
An instance of a CSP is defined by a set of variables that take values in -- typically finite -- domains and a set of constraints, where each constraint is satisfied for specific assignments of the subset of variables it involves. 
A major computational challenge is to determine whether such an instance is satisfiable, i.e.~to determine if there is an assignment of all variables that satisfies all constraints.

Since the 1980s non-rigorous methods have been introduced in statistical physics that are targeted at the analysis of phase transitions in random CSPs \cite{mezard1987,mezard2009,krzakala2016}. Within this line of research, a variety of exciting and unexpected phenomena were discovered, as for example the existence of multiple phase transitions with respect to the structure of the solution space; these transitions may have a significant impact on the hardness of the underlying instances. Since then these methods and the description of the conjectured regimes have been heavily supported by several findings, including the astounding empirical success of randomized algorithms like \emph{belief} and \emph{survey propagation} \cite{braunstein2005}, as well as rigorous verifications, most prominently the phase transition in $k$-SAT \cite{ding2015} (for large $k$) and the condensation phase transition in many important models~\cite{coja2018a}. However, a complete rigorous study is still a big challenge for computer science and mathematics.

Usually, the relevant model parameter of a random CSP is a certain problem specific \emph{density} as illustrated below. The main focus of research is to study the occurrence of phase transitions in the solution space structure and in particular the existence of \emph{(sharp) satisfiability thresholds}, i.e., critical values of the density such that the probability that a random CSP admits a solution tends to one as the number of variables tends to infinity for densities below the threshold, while this limiting probability tends to zero for densities above the threshold.

\subparagraph*{Random CSPs.}
Two important types of random CSPs are Erdős-Rényi (ER) type and random regular CSPs. In both cases the number $n=|V|$ of variables and the number $k$ of variables involved in each constraint is fixed. In ER type CSPs we further fix the number $m=|F|$ of constraints and thereby the \emph{density} $\alpha={m}/{n}$, i.e.~the average number of constraints that a variable is involved in. In the regular case we only consider instances where each variable is involved in the same number $d$ of constraints, which  fixes the \emph{density}~$d$ as well as the number $m={dn}/{k}$ of constraints.
In a second step we randomly choose the sets of satisfying assignments for each constraint depending on the problem. For example, in the prominent $k$-SAT problem one forbidden assignment is chosen uniformly at random from all possible assignments of the involved binary variables for each constraint independently.
Another example is the coloring of hypergraphs, where the constraints are attached to the hyperedges and the variables to the vertices, i.e.~the variables involved in a constraint correspond to the vertices incident to a hyperedge.
In this case a constraint is violated iff all involved vertices take the same color.

In our work we focus on a class of random regular CSPs in which the choice of satisfying assignments per constraint is fixed in advance, i.e.,~a class that contains the aforementioned coloring of ($d$-regular $k$-uniform) hypergraphs and occupation problems amongst others, but that does not include problems with further randomness in the constraints, like $k$-SAT and $k$-XORSAT.
The lack of randomness on the level of constraints makes this class particularly accessible for an analysis of the asymptotic solution space structure and significantly simplifies simulations based on the well-known \emph{population dynamics}. Using such simulations, non-rigorous results for this class have been mostly established for the case where the variables are binary valued, so called occupation problems, or restricted to variants of hypergraph coloring for non-binary variables. Besides the extensive studies on the coloring of simple graphs, i.e.~$k=2$, the only rigorous results derived so far consider the arguably most simple type of occupation problems where each constraint is satisfied if exactly one involved variable evaluates to true, which we refer to as $d$-regular $1$-in-$k$ occupation problem.
In our current work we strive to extend these results to general $d$-regular $r$-in-$k$ occupation problems, i.e.~problems where each constraint is satisfied if $r$ out of the $k$ involved variables evaluate to true.
\subsection{Occupation Problems}\label{s_occupation_problems}
We continue with the formal definition of the class of problems we consider.
Let $k$, $d\in\mathbb Z_{\ge 2}$ and $r\in[k-1] := \{1, \dots, k-1\}$ be fixed.
Additionally, we are given non-empty sets $V$ of variables and constraints $F$.
We will use the convention to index elements of $V$ with the letter $i$ and
elements of $F$ with the letter $a$ (and subsequent letters) in the remainder.
Then an instance $o$ of the $d$-regular $r$-in-$k$ occupation problem is specified by 
a sequence $o=(v(a))_{a\in F}$ of $m=|F|$ subsets $v(a)\subseteq V$ of size $k$ such that
each of the $n=|V|$ variables is contained in $d$ of the subsets.
In graph theory the instance $o$ has a natural interpretation as a $(d,k)$-biregular graph (or $d$-regular $k$-factor graph) with node sets $V\dotcup F$ and edges $\{i,a\}\in E$ if $i\in v(a)$.
By the handshaking lemma, such objects only exist if $dn=km$, which we assume in the following.

Given an instance $o$ as just described, we say that an assignment $x\in\{0,1\}^V$ \emph{satisfies} a constraint $a\in F$ if $\sum_{i\in v(a)}x_i=r$, otherwise $x$ \emph{violates} $a$. If $x$ satisfies all constraints $a\in F$, then $x$ is a \emph{solution} of $o$.
Notice that $d$ times the number of $1$'s in $x$ needs to match the total number $rm=rdn/k$ of $1$'s observed on the factor side, so $k$ has to divide $rn$, which we also assume in the following.
We write $z(o)$ for the number of solutions of $o$. An example of a $4$-regular $2$-in-$3$ occupation problem is shown in Figure \ref{f_related_problems_occupation}.

Further, for given $m$, $n\in\mathbb Z_{>0}$ let $\mathcal O$ denote the set of all 
instances $o$ with variables $V=[n]$ and constraints $F=[m]$.
If $\mathcal O$ is not empty, then the random $d$-regular $r$-in-$k$ occupation problem $O$ is
uniform on $\mathcal O$ and
$Z=z(O)$ the number of solutions of $O$.
%
%
%
%
\subsection{Examples and Related Problems}\label{s_related_problems}
%
%
\begin{figure}
\begin{subfigure}[t]{0.497\textwidth}
\centering
\begin{tikzpicture}
  [scale=0.8, every node/.style={transform shape},
    var1/.style={circle,fill,minimum size=4mm},
    var0/.style={circle,draw,thick,minimum size=4mm},
    fac/.style={rectangle,draw,thick,minimum size=4mm}]
  \node (l1) at (1,5) [fac] {}; \node [inner sep=0, outer sep=0] at (1.45,5) {$a_1$};
  \node (l2) at (4,5) [fac] {}; \node [inner sep=0, outer sep=0] at (4.45,5) {$a_2$};
  \node (l3) at (7,5) [fac] {}; \node [inner sep=0, outer sep=0] at (7.45,5) {$a_3$};
  \node (l4) at (2,3) [fac] {}; \node [inner sep=0, outer sep=0] at (1.92,3.34) {$a_4$};
  \node (l5) at (3,3) [fac] {}; \node [inner sep=0, outer sep=0] at (3.1,3.34) {$a_5$};
  \node (l6) at (5,3) [fac] {}; \node [inner sep=0, outer sep=0] at (4.92,3.34) {$a_6$};
  \node (l7) at (6,3) [fac] {}; \node [inner sep=0, outer sep=0] at (6.1,3.34) {$a_7$};
  \node (l8) at (1,1) [fac] {}; \node [inner sep=0, outer sep=0] at (1.45,1) {$a_8$};
  \node (l9) at (4,1) [fac] {}; \node [inner sep=0, outer sep=0] at (4.45,1) {$a_9$};
  \node (l10) at (7,1) [fac] {}; \node [inner sep=0, outer sep=0] at (7.5,1) {$a_{10}$};
  \node (a1) at (2.5,4) [var1] {}; \node [inner sep=0, outer sep=0] at (2.93,4) {$i_1$};
  \node (a2) at (5.5,4) [var1] {}; \node [inner sep=0, outer sep=0] at (5.93,4) {$i_2$};
  \node (a3) at (1,3) [var0] {}; \node [inner sep=0, outer sep=0] at (1.25,3.3) {$i_3$};
  \node (a4) at (4,3) [var0] {}; \node [inner sep=0, outer sep=0] at (4.25,3.3) {$i_4$};
  \node (a5) at (7,3) [var0] {}; \node [inner sep=0, outer sep=0] at (7.25,3.3) {$i_5$};
  \node (a6) at (2.5,2) [var1] {}; \node [inner sep=0, outer sep=0] at (2.93,2) {$i_6$};
  \node (a7) at (5.5,2) [var1] {}; \node [inner sep=0, outer sep=0] at (5.93,2) {$i_7$};
\begin{scope}
  \draw [thick] (l1) to (a1);
  \draw [thick] (l2) to (a1);
  \draw [thick] (l4) to (a1);
  \draw [thick] (l5) to (a1);
  \draw [thick] (l2) to (a2);
  \draw [thick] (l3) to (a2);
  \draw [thick] (l6) to (a2);
  \draw [thick] (l7) to (a2);
  \draw [thick] (l4) to (a6);
  \draw [thick] (l5) to (a6);
  \draw [thick] (l8) to (a6);
  \draw [thick] (l9) to (a6);
  \draw [thick] (l6) to (a7);
  \draw [thick] (l7) to (a7);
  \draw [thick] (l9) to (a7);
  \draw [thick] (l10) to (a7);
  \draw [thick] (l1) to (a3);
  \draw [thick] (l4) to (a3);
  \draw [thick] (l8) to (a3);
  \draw [thick] (l2) to (a4);
  \draw [thick] (l5) to (a4);
  \draw [thick] (l6) to (a4);
  \draw [thick] (l9) to (a4);
  \draw [thick] (l3) to (a5);
  \draw [thick] (l7) to (a5);
  \draw [thick] (l10) to (a5);
  \draw [thick] (0,3) to (a3);
  \draw [thick] (8,3) to (a5);
  \draw [thick] (0,4.33) to (l1);
  \draw [thick] (0,1.67) to (l8);
  \draw [thick] (8,4.33) to (l3);
  \draw [thick] (8,1.67) to (l10);
\end{scope}
\begin{scope}
\path (0,0) to (8,0) to (8,6) to (0,6) to (0,0);
\end{scope}
\end{tikzpicture}
\subcaption{Solution of the 2-occupation problem}\label{f_related_problems_occupation}
\end{subfigure}
\begin{subfigure}[t]{0.497\textwidth}
\centering
\begin{tikzpicture}
  [scale=0.8, every node/.style={transform shape},
    var/.style={circle,draw,thick,minimum size=4mm}]
  \node (v1) at (1,5) [var] {}; \node [inner sep=0,outer sep=0] (lv1) at (0.2,5) {$a_1$};\draw [->] (0.36,5) to (v1);
  \node (v2) at (4,5) [var] {}; \node [inner sep=0,outer sep=0] (lv2) at (3.2,5) {$a_2$};\draw [->] (3.36,5) to (v2);
  \node (v3) at (7,5) [var] {}; \node [inner sep=0,outer sep=0] (lv3) at (7.8,5) {$a_3$};\draw [->] (7.6,5) to (v3);
  \node (v4) at (2,3) [var] {}; \node [inner sep=0,outer sep=0] (lv4) at (2,3.8) {$a_4$};\draw [->] (2,3.7) to (v4);
  \node (v5) at (3,3) [var] {}; \node [inner sep=0,outer sep=0] (lv5) at (3,3.8) {$a_5$};\draw [->] (3,3.7) to (v5);
  \node (v6) at (5,3) [var] {}; \node [inner sep=0,outer sep=0] (lv6) at (5,3.8) {$a_6$};\draw [->] (5,3.7) to (v6);
  \node (v7) at (6,3) [var] {}; \node [inner sep=0,outer sep=0] (lv7) at (6,3.8) {$a_7$};\draw [->] (6,3.7) to (v7);
  \node (v8) at (1,1) [var] {}; \node [inner sep=0,outer sep=0] (lv8) at (0.2,1) {$a_8$};\draw [->] (0.36,1) to (v8);
  \node (v9) at (4,1) [var] {}; \node [inner sep=0,outer sep=0] (lv9) at (3.2,1) {$a_9$};\draw [->] (3.36,1) to (v9);
  \node (v10) at (7,1) [var] {}; \node [inner sep=0,outer sep=0] (lv10) at (7.73,1) {$a_{10}$};\draw [->] (7.46,1) to (v10);
\begin{scope}
  \draw [thick](0,5.5) to (1.5,5.5) to[out=0,in=45] (1,4.5) to (0,3.5);
  \draw [thick](0,0.5) to (1.5,0.5) to[out=0,in=315] (1,1.5) to (0,2.5);
  \draw [thick](8,5.3) to (6.5,5.3) to[out=180,in=135] (7,4.5) to (8,3.5);
  \draw [thick](8,0.7) to (6.5,0.7) to[out=180,in=225] (7,1.5) to (8,2.5);
  \draw [thick](0.9,5.3) to (4.1,5.3) to[out=0,in=68.96] (4,4) to (3.615,3) to[out=248.96,in=0] (3,2.7) to (2,2.7)
    to[out=180,in=291.04] (1.385,3) to (1,4) to[out=111.04,in=180] (0.9,5.3);
  \draw [thick](0.9,0.7) to (4.1,0.7) to[out=0,in=291.04] (4,2) to (3.615,3) to[out=111.04,in=0] (3,3.3) to (2,3.3)
    to[out=180,in=68.96] (1.385,3) to (1,2) to[out=248.96,in=180] (0.9,0.7);

  \draw [thick](7.1,5.5) to (3.9,5.5) to[out=180,in=109.65] (4,4.1) to (4.393,3) to[out=289.65,in=180] (5,2.7) to (6,2.7)
    to[out=0,in=250.35] (6.607,3) to (7,4.1) to[out=70.35,in=0]  (7.1,5.5);
  \draw [thick](7.1,0.5) to (3.9,0.5) to[out=180,in=250.35] (4,1.9) to (4.393,3) to[out=70.35,in=180] (5,3.3) to (6,3.3)
    to[out=0,in=109.65] (6.607,3) to (7,1.9) to[out=289.65,in=0] (7.1,0.5);

  \draw[dashed] (4,0) to[out=0,in=270] (5.33,1.33) to (5.33,4.67) to[out=90,in=0] (4,6) to[out=180,in=90] (2.67,4.67)
    to (2.67,1.33) to[out=270,in=180] (4,0);
  \draw[dashed] (8,5.5) to[out=135,in=0] (7,6) to[out=180,in=90] (5.67,4.67) to (5.67,1.33)
    to[out=270,in=180] (7,0) to[out=0,in=225] (8,0.5);
  \draw[dashed] (0,5.7) to[out=45,in=180] (1,6) to[out=0,in=90] (2.33,4.67) to (2.33,1.33)
    to[out=270,in=0] (1,0) to[out=180,in=315] (0,0.3);
\end{scope}
  \node (lh1) [inner sep=0,outer sep=0] at (0.5,3) {$e_{i_1}$};\draw [->] (0.5,3.15) to (1.13,3.73);
  \node (lh2) [inner sep=0,outer sep=0] at (7.5,3) {$e_{i_2}$};\draw [->] (7.3,3.15) to (6.87,3.73);
  \node (lh3) [inner sep=0,outer sep=0] at (2.5,5.8) {$e_{i_3}$};\draw [->] (2.25,5.85) to (1.63,5.85);
  \node (lh4) [inner sep=0,outer sep=0] at (5.5,5.8) {$e_{i_4}$};\draw [->] (5.25,5.85) to (4.63,5.85);
  \node (lh5) [inner sep=0,outer sep=0] at (5.5,0.15) {$e_{i_5}$};\draw [->] (5.7,0.2) to (6.32,0.2);
  \node (lh6) [inner sep=0,outer sep=0] at (1.8,1.8) {$e_{i_6}$};\draw [->] (1.58,1.93) to (1.13,2.27);
  \node (lh7) [inner sep=0,outer sep=0] at (6.2,1.8) {$e_{i_7}$};\draw [->] (6.3,1.93) to (6.87,2.27);
\begin{scope}
\path (0,0) to (8,0) to (8,6) to (0,6) to (0,0);
\end{scope}
\end{tikzpicture}
\subcaption{A $2$-factor in a hypergraph}\label{f_related_problems_2_factor}
\end{subfigure}
\captionsetup{width=0.89\textwidth, font=small}
\caption{
On the left we see a solution of the $4$-regular $2$-in-$3$ occupation problem on a $4$-regular $3$-factor graph, where the
rectangles and circles depict the constraints (factors) and variables (filled if they take the value one in the solution).
The figure on the right shows a $2$-factor in a $3$-regular $4$-uniform hypergraph, where the circles, solid and dashed shapes represent the vertices, hyperedges in the $2$-factor and the other hyperedges respectively.
}\label{f_related_problems}
\end{figure}
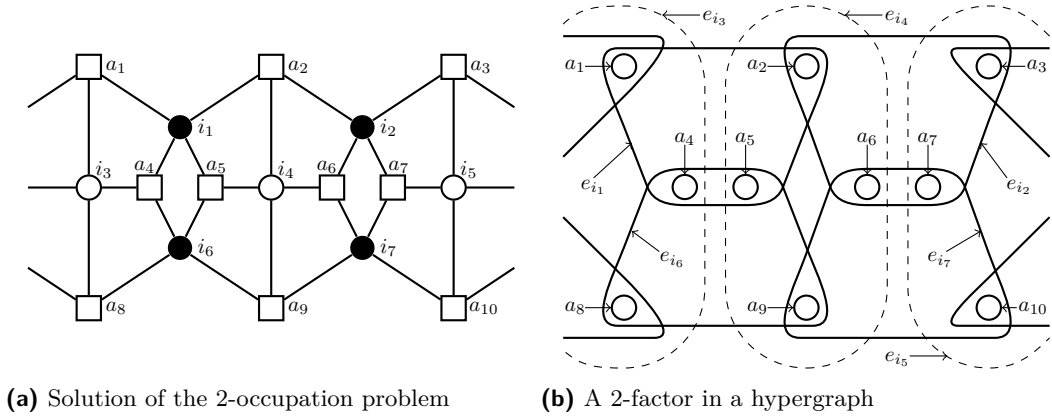
%
%
A problem that is closely related and can be reduced to the $d$-regular $r$-in-$k$ occupation problem is the $d$-regular positive $r$-in-$k$ SAT problem, a variant of $k$-SAT introduced above.
In this case, we consider a boolean formula
\begin{flalign*}
f=\bigwedge_{a\in F}c_a\textrm{, } \quad c_a=\bigvee_{i\in v(a)}i\textrm{, }a\in F\textrm{, }
\end{flalign*}
in conjunctive normal form with $m$ clauses over $n$ variables $i\in V$, such that
no literal appears negated (hence \textit{positive} $r$-in-$k$ SAT),
and where each clause $c_a$ is the disjunction of $k$ literals and
each variable appears in exactly $d$ clauses (hence $d$-regular).
The decision problem  is to determine if there exists an assignment $x$  such that exactly $r$ literals in each clause evaluate to true (hence $r$-in-$k$ SAT).
In \cite{moore2016} the satisfiability threshold for this problem was determined for $r=1$, i.e.~the case where exactly one literal in each clause evaluates to true. Our Theorem \ref{theorem_satisfiability_threshold} solves this problem for $r=2$ and $k\in\mathbb Z_{\ge 4}$.

%
%
Our second example deals with a prominent problem related to graph theory.
A $k$-regular $d$-uniform hypergraph $h$ is a pair $h=(F,E)$ with $m = |F|$ vertices and $n = |E|$ (hyper-)edges such that each edge contains $d$ vertices and the degree of each vertex is $k$.
An \emph{$r$-factor} $E'$ is a subset of the hyperedges such that each vertex $a\in F$ is incident to $r$ hyperedges $e_i\in E'$. In this case the problem is to determine if $h$ has an $r$-factor.
For example, the case $r=1$ is the well-known perfect matching problem and the threshold was determined in \cite{cooper1996}. An example of a $2$-factor in a hypergraph is shown in Figure \ref{f_related_problems_2_factor}. Theorem \ref{theorem_satisfiability_threshold} solves also this problem for $r=2$ and $k\in\mathbb Z_{\ge 4}$.

%
%
There are several other problems in complexity and graph theory that are closely related to the examples above.
The satisfiability threshold in Theorem \ref{theorem_satisfiability_threshold} also applies to a variant of the vertex cover problem (or hitting set problem from set theory perspective), where we choose a subset of the vertices (variables with value one) in a $d$-regular $k$-uniform hypergraph such that each hyperedge is incident to exactly two vertices in the subset.
Analogously, Theorem \ref{theorem_satisfiability_threshold} also establishes the threshold for a variant of the set cover problem in set theory corresponding to $2$-factors in hypergraphs, i.e.~given a family of $d$-subsets (hyperedges) and a universe (vertices) with each element contained in $k$ subsets, the problem is to find a subfamily of the subsets such that each element of the universe is contained in exactly two subsets of the subfamily.
Further, Theorem \ref{theorem_satisfiability_threshold} can e.g.~also be used to give sufficient conditions for the (asymptotic) existence of Euler families in regular uniform hypergraphs as discussed in \cite{bahmanian2017}.

%
%
\subsection{Main Results}\label{s_main_results}
The satisfiability threshold for the $d$-regular $1$-in-$k$ occupation problem has been established in \cite{moore2016,cooper1996}, which also covers the $d$-regular $2$-in-$3$ occupation problem due to color symmetry.
Our main result pins down the location of the satisfiability threshold of the random $d$-regular $2$-in-$k$ occupation problem for $k\in\mathbb Z_{\ge 4}$.
For this purpose let
\begin{flalign}\label{e_ws_ds_definition}
d^*=d^*(k)=\frac{kH(2/k)}{kH(2/k)-\ln\binom{k}{2}}\textrm{,}
\quad
k\in\mathbb Z_{\ge 4}\textrm{,}
\end{flalign}
where $H(p)=-p\ln(p)-(1-p)\ln(1-p)$ is the entropy of $p\in[0,1]$.
The following theorem establishes the location of the threshold at
$d^*$.
\begin{theorem}[$2$-in-$k$ Occupation Satisfiability Threshold]
\label{theorem_satisfiability_threshold}
Let $k\in\mathbb Z_{\ge 4}$, $d\in\mathbb Z_{\ge 2}$, and let $Z$ be the number of solutions
from Section \ref{s_occupation_problems}.
There exists a sharp satisfiability threshold at $d^*$, i.e.~for any increasing sequence $(n_i)_{i\in\mathbb Z_{>0}}\subseteq\mathcal N=\{n:dn\textrm{, }2n\in k\mathbb Z_{>0}\}$ and $m_i=dn_i/k$ we have
\begin{flalign*}
\lim_{i\rightarrow\infty}\mathbb P[Z>0]=
\left\{\begin{matrix}
1&\textrm{, }d<d^*\\
0&\textrm{, }d\ge d^*
\end{matrix}\right.\textrm{.}
\end{flalign*}
\end{theorem}
We provide a self-contained proof for Theorem~\ref{theorem_satisfiability_threshold} using the first and second moment method with small subgraph conditioning for $Z$.
In particular, a main technical contribution in proving Theorem~\ref{theorem_satisfiability_threshold} is the optimization of a certain multivariate function that appears in the computation of the second moment, which encodes the interplay between the ``similarity'' of various assignments and the change in the corresponding probability of being satisfying that they induce. A direct corollary of this optimization step at the threshold $d^*$ is the confirmation of the conjecture by the authors in \cite{panagiotou2019}. Among other things, at the core of our contribution we take a novel and rather different approach to tackle the optimization, inspired by \cite{robalewska1996} and \cite{yedidia2001} as well as other works relating the fixed points of belief propagation to the stationary points of the bethe free entropy, respectively to the computation of the annealed free entropy density; see Section~\ref{s_second_moment_optimization} for details.
Finally, we show that $d^*$ is not an integer in Lemma \ref{lem_satthresh_basics} below, so as opposed to the case $r=1$ \cite{moore2016}, for $r=2$ there is no need for a dedicated analysis at criticality.
%
%
\subsection{Related Work}\label{s_background}
The regular version of the random $1$-in-$k$ occupation problem (and related problems) has been studied in \cite{cooper1996,moore2016} using the first and second moment method with small subgraph conditioning.
The paper \cite{robalewska1996} shows that $\lim_{i\rightarrow\infty}\mathbb P[Z>0]=1$ for $d=2$ and $k\in\mathbb Z_{\ge 2}$  in the $d$-regular $2$-in-$k$ occupation problem, i.e.~the existence of $2$-factors in $k$-regular simple graphs.
A recent discussion of $2$-factors (and the related Euler families) that does not rely on the probabilistic method is presented in \cite{bahmanian2017}.
Further, randomized polynomial time algorithms for the generation and approximate counting of $2$-factors in random regular graphs have been developed in \cite{frieze1996}.

The study of Erd\H{o}s R\'{e}nyi (hyper-)graphs was initiated by the groundbreaking paper~\cite{erdos1960} in 1960 and turned into a fruitful field of research with many applications, including early results on $1$-factors in simple graphs~\cite{erdos1966}. On the contrary, results for the random $d$-regular $k$-uniform (hyper-)graph ensemble were rare before the introduction of the configuration (or pairing) model by Bollob\'{a}s \cite{bollobas1980} and the development of the small subgraph conditioning method \cite{janson1995,janson2000}.
While the proof scheme facilitated rigorous arguments to establish the existence and location of satisfiability thresholds of random regular CSPs \cite{molloy1997,bapst2017,kemkes2010,coja2016a,coja2016,ding2016,ding2016a}, the problems are treated on a case by case basis, while results on entire classes of random regular CSPs are still outstanding.

One of the main reasons responsible for the complexity of a rigorous analysis of random (regular) CSPs seems to be a conjectured structural change of the solution space for increasing densities.
This hypothesis has been put forward by physicists, verified in parts and mostly for ER ensembles, further led to new rigorous proof techniques \cite{ding2015,coja2016,coja2018} and to randomized algorithms \cite{braunstein2005,maneva2007} for NP-hard problems
that are not only of great value in practice, but can also be employed for precise numerical (though non-rigorous) estimates of satisfiability thresholds.
An excellent introduction to this \emph{replica theory} can be found in \cite{mezard2009,krzakala2016,talagrand2003}.
Specifically, numerical results indicating the satisfiability thresholds for $d$-regular $r$-in-$k$ occupation problems (more general variants, and for ER type hypergraphs) based on this conjecture were discussed in various publications \cite{castellani2003,asta2008,schmidt2016,gabrie2017,higuchi2010,zdeborova2008a, zdeborova2011}, where \emph{occupation problems} were introduced for the first time in \cite{mora2007}.

Another fundamental obstacle in the rigorous analysis is of a very technical nature and directly related to the second moment method as discussed in detail in our current work. In the case of regular $2$-in-$k$ occupation problems (amongst others) this optimization problem can be solved by exploiting a connection to the fixed points of belief propagation. This well-studied message passing algorithm is thoroughly discussed in \cite{mezard2009}.
\subsection{Open Problems}
In this work we rigorously establish the threshold for $r=2$ and $k\in\mathbb Z_{\ge 4}$ for the random regular $r$-in-$k$ occupation problem.
A rigorous proof for general $r$ (and $k$) seems to be involved, but further assumptions may significantly simplify the analysis.
For example, as an extension of the current work one may focus on $r$-in-$2r$ occupation problems, where the constraints are symmetric in the colors. As can be seen from our proof, this yields useful symmetry properties. Further, as suggested by the literature \cite{coja2018, coja2018a} such \emph{balanced} problems \cite{zdeborova2011, zdeborova2008a}
are usually more accessible to a rigorous study.
On the other hand, the optimization usually also significantly simplifies if only carried out for $k\ge k_0(r)$ for some (large) $k_0(r)$.

Apart from the generalizations discussed above, results for the general $r$-in-$k$ occupation problems are also still outstanding for Erdős-Rényi type CSPs,
the only exception being the satisfiability threshold for perfect matchings which was recently established by Kahn \cite{kahn2022}.
Further, there only exist bounds for the exact cover problem \cite{moore2008} on $3$-uniform hypergraphs, i.e.~$r=1$ and $k=3$.
\subparagraph*{Outline of the Proofs.}
In Section \ref{s_proof_techniques} we present the proof strategy on a high level.
Then, we turn to the notation and do some groundwork, in particular the analysis of $d^*(k)$, in Section \ref{m_preliminaries}.
The easy part of the main result is established in Section \ref{s_first_moment} using the first moment method.
The remainder is devoted to the proof that solutions exist below the threshold with probability tending to one,
starting with the second moment method in Section \ref{s_second_moment}.
Most of the twenty pages in this section are devoted to the solution of the optimization problem and related conjecture from \cite{panagiotou2019} using a belief propagation inspired approach.
Finally, we complete the small subgraph conditioning method in Section \ref{s_small_subgraph_conditioning}, using the proof of Lemma \ref{lem_number_of_cycles} in the Appendix \ref{a_number_of_cycles} as a blueprint.
%
%
\section{Proof Techniques}\label{s_proof_techniques}
In this section we give a high-level overview of our proof. We make heavy use of the so-called \emph{configuration model} for the generation of random instances in the form used by Moore \cite{moore2016}.

\subsection{The Configuration Model}\label{p_configuration_model}
Working with the uniform distribution on $d$-regular $k$-uniform hypergraphs directly is challenging.
Instead, we show Theorem \ref{theorem_satisfiability_threshold} for occupation problems on so-called \emph{configurations}.
A \emph{$d$-regular $k$-configuration} is a bijection $g:[n]\times[d]\rightarrow[m]\times[k]$, where the \emph{v-edges} $(i,h)\in[n]\times[d]$ represent pairs of variables $i\in[n]$ and so-called $i$-edges, i.e.~half-edge indices $h\in[d]$. The image $(a,h')=g(i,h)$ is an \emph{f-edge}, i.e.~a pair of a constraint (factor) $a\in[m]$ and an $a$-edge (or half-edge) $h'\in[k]$, indicating that the $i$-edge $h$ of the variable $i$ is wired to the $a$-edge $h'$ of $a$ and thereby suggesting that $i$ is connected to $a$ in the corresponding $d$-regular $k$-factor graph.
Notice that we can represent $g$ by an equivalent, four-partite, graph with
 (disjoint) vertex sets given by the variables $V=[n]$, constraints (factors) $F=[m]$, v-edges $H'=[n]\times[d]$ and f-edges $H=[m]\times[k]$, where each variable $i\in[n]$ connects to all its v-edges $(i,h')\in H'$, each constraint $a\in[m]$ to all its f-edges $(a,h)\in H$ and a v-edge $(i,h')$ connects to an f-edge $(a,h)$ if $g(i,h')=(a,h)$.

Let $\mathcal G$ be the set of all $d$-regular $k$-configurations on $n$ variables, and notice that $|\mathcal G|=\emptyset$ iff $dn\neq km$ and $|\mathcal G|=(dn)!=(km)!$ for $m=dn/k\in\mathbb Z$, which we assume from here on.
Further, the occupation problem on factor graphs directly translates to configurations,
i.e.~an assignment $x\in\{0,1\}^n$ is a solution of $g\in\mathcal G$ if for each constraint $a\in[m]$ there exist exactly two distinct $a$-edges $h$, $h'\in[k]$ such that $x_{i(a,h)}=x_{i(a,h')}=1$, where $i(a,h)=(g^{-1}(a,h))_1$ denotes the $h$-th neighbour of $a$. To wit, the occupation problem on a configuration corresponding to the example in Figure \ref{f_related_problems_occupation} is shown in Figure \ref{f_related_problems_2_configuration}.

Let $Z(g)$ be the number of solutions of $g \in \cal G$. As before, $Z=0$ almost surely unless $2n\in k\mathbb Z$. The following result is the main tool from proving Theorem~\ref{theorem_satisfiability_threshold}.
\begin{theorem}[Satisfiability Threshold for Configurations]
\label{satisfiability_threshold_configuration}
Theorem \ref{theorem_satisfiability_threshold} also holds for $Z$ as defined in this section.
\end{theorem}
%
%
Theorem \ref{satisfiability_threshold_configuration} translates to the models in Section \ref{s_occupation_problems} and Section \ref{s_related_problems} using standard arguments, namely symmetry, the discussion of parallel edges, contiguity, and the fact that both the variable and the factor neighborhoods are unique with probability tending to one.
Hence, in the remainder of this contribution we exclusively consider the occupation problem on configurations.
\begin{figure}
\begin{subfigure}[t]{0.497\textwidth}
\centering
\begin{tikzpicture}
  [scale=0.8, every node/.style={transform shape},
    var1/.style={circle,fill,minimum size=4mm},
    var0/.style={circle,draw,thick,minimum size=4mm},
    fac/.style={rectangle,draw,thick,minimum size=4mm},
    varedg1/.style={circle,fill,minimum size=2mm,inner sep=0},
    varedg0/.style={circle,draw,thick,minimum size=2mm,inner sep=0},
    facedg1/.style={rectangle,fill,minimum size=2mm,inner sep=0},
    facedg0/.style={rectangle,draw,thick,minimum size=2mm,inner sep=0}]
  \node (l2) at (4,5) [fac] {}; \node [inner sep=0, outer sep=0] at (4,5.35) {$a_2$};
  \node (l5) at (2.5,3) [fac] {}; \node [inner sep=0, outer sep=0] at (2.5,3.35) {$a_5$};
  \node (l6) at (5.5,3) [fac] {}; \node [inner sep=0, outer sep=0] at (5.5,3.35) {$a_6$};
  \node (l9) at (4,1) [fac] {}; \node [inner sep=0, outer sep=0] at (4,0.65) {$a_9$};
  \node (a1) at (1,4) [var1] {}; \node [inner sep=0, outer sep=0] at (0.72,4.3) {$i_1$};
  \node (a2) at (7,4) [var1] {}; \node [inner sep=0, outer sep=0] at (7.28,4.3) {$i_2$};
  \node (a4) at (4,3) [var0] {}; \node [inner sep=0, outer sep=0] at (4.3,2.7) {$i_4$};
  \node (a6) at (1,2) [var1] {}; \node [inner sep=0, outer sep=0] at (0.85,1.7) {$i_6$};
  \node (a7) at (7,2) [var1] {}; \node [inner sep=0, outer sep=0] at (7.3,1.7) {$i_7$};
  \node (fe21) at (3,4.67) [facedg1]{};\node [inner sep=0, outer sep=0] at (3,5) {$h_{a_2,1}$};
  \node (fe22) at (4,4.33) [facedg0]{};\node [inner sep=0, outer sep=0] at (4.5,4.22) {$h_{a_2,2}$};
  \node (fe23) at (5,4.67) [facedg1]{};\node [inner sep=0, outer sep=0] at (5,5) {$h_{a_2,3}$};
  \node (fe51) at (3,3) [facedg0]{};\node [inner sep=0, outer sep=0] at (3.2,3.3) {$h_{a_5,1}$};
  \node (fe52) at (2,3.33) [facedg1]{};\node [inner sep=0, outer sep=0] at (1.53,3.25) {$h_{a_5,2}$};
  \node (fe53) at (2,2.67) [facedg1]{};\node [inner sep=0, outer sep=0] at (1.5,2.75) {$h_{a_5,3}$};
  \node (fe61) at (6,3.33) [facedg1]{};\node [inner sep=0, outer sep=0] at (6.5,3.2) {$h_{a_6,1}$};
  \node (fe62) at (5,3) [facedg0]{};\node [inner sep=0, outer sep=0] at (5.1,2.7) {$h_{a_6,2}$};
  \node (fe63) at (6,2.67) [facedg1]{};\node [inner sep=0, outer sep=0] at (6.5,2.73) {$h_{a_6,3}$};
  \node (fe91) at (5,1.33) [facedg1]{};\node [inner sep=0, outer sep=0] at (5.15,1.02) {$h_{a_9,1}$};
  \node (fe92) at (4,1.67) [facedg0]{};\node [inner sep=0, outer sep=0] at (3.5,1.7) {$h_{a_9,2}$};
  \node (fe93) at (3,1.33) [facedg1]{};\node [inner sep=0, outer sep=0] at (3,1.1) {$h_{a_9,3}$};
  \node (ve11) at (2,4.33) [varedg1]{};\node [inner sep=0, outer sep=0] at (2,4.67) {$h_{i_1,1}$};
  \node (ve14) at (1.5,3.67) [varedg1]{};\node [inner sep=0, outer sep=0] at (1.98,3.7) {$h_{i_1,4}$};
  \node (ve22) at (6,4.33) [varedg1]{};\node [inner sep=0, outer sep=0] at (6,4.67) {$h_{i_2,2}$};
  \node (ve23) at (6.5,3.67) [varedg1]{};\node [inner sep=0, outer sep=0] at (6.05,3.75) {$h_{i_2,3}$};
  \node (ve41) at (4.5,3) [varedg0]{};\node [inner sep=0, outer sep=0] at (4.6,3.3) {$h_{i_4,1}$};
  \node (ve42) at (4,3.67) [varedg0]{};\node [inner sep=0, outer sep=0] at (3.53,3.7) {$h_{i_4,2}$};
  \node (ve43) at (3.5,3) [varedg0]{};\node [inner sep=0, outer sep=0] at (3.55,2.75) {$h_{i_4,3}$};
  \node (ve44) at (4,2.33) [varedg0]{};\node [inner sep=0, outer sep=0] at (4.48,2.2) {$h_{i_4,4}$};
  \node (ve61) at (1.5,2.33) [varedg1]{};\node [inner sep=0, outer sep=0] at (1.93,2.17) {$h_{i_6,1}$};
  \node (ve64) at (2,1.67) [varedg1]{};\node [inner sep=0, outer sep=0] at (2,1.45) {$h_{i_6,4}$};
  \node (ve72) at (6.5,2.33) [varedg1]{};\node [inner sep=0, outer sep=0] at (6.05,2.35) {$h_{i_7,2}$};
  \node (ve73) at (6,1.67) [varedg1]{};\node [inner sep=0, outer sep=0] at (6.15,1.36) {$h_{i_7,3}$};
\begin{scope}
  \draw [thick] (l2) to (fe21);\draw [thick] (fe21) to (ve11);\draw [thick] (ve11) to (a1);
  \draw [thick] (l5) to (fe52);\draw [thick] (fe52) to (ve14);\draw [thick] (ve14) to (a1);
  \draw [thick] (l2) to (fe23);\draw [thick] (fe23) to (ve22);\draw [thick] (ve22) to (a2);
  \draw [thick] (l6) to (fe61);\draw [thick] (fe61) to (ve23);\draw [thick] (ve23) to (a2);
  \draw [thick] (l5) to (fe53);\draw [thick] (fe53) to (ve61);\draw [thick] (ve61) to (a6);
  \draw [thick] (l6) to (fe63);\draw [thick] (fe63) to (ve72);\draw [thick] (ve72) to (a7);
  \draw [thick] (l9) to (fe93);\draw [thick] (fe93) to (ve64);\draw [thick] (ve64) to (a6);
  \draw [thick] (l9) to (a7);
  \draw [thick] (l2) to (fe22);\draw [thick] (fe22) to (ve42);\draw [thick] (ve42) to (a4);
  \draw [thick] (l5) to (fe51);\draw [thick] (fe51) to (ve43);\draw [thick] (ve43) to (a4);
  \draw [thick] (l6) to (fe62);\draw [thick] (fe62) to (ve41);\draw [thick] (ve41) to (a4);
  \draw [thick] (l9) to (fe92);\draw [thick] (fe92) to (ve44);\draw [thick] (ve44) to (a4);
  \draw [dashed] (0,4.33) to (a1);\draw [dashed] (0,3.33) to (a1);
  \draw [dashed] (0,1.67) to (a6);\draw [dashed] (0,2.67) to (a6);
  \draw [dashed] (8,4.33) to (a2);\draw [dashed] (8,3.33) to (a2);
  \draw [dashed] (8,1.67) to (a7);\draw [dashed] (8,2.67) to (a7);
\end{scope}
\begin{scope}
\path (0,0) to (8,0) to (8,6) to (0,6) to (0,0);
\end{scope}
\end{tikzpicture}
\subcaption{occupation problem on configurations}\label{f_related_problems_2_configuration}
\end{subfigure}
\begin{subfigure}[t]{0.497\textwidth}
\centering
\begin{tikzpicture}
  [scale=0.8, every node/.style={transform shape},
    var1/.style={circle,fill,thick,minimum size=4mm},
    var0/.style={circle,draw,thick,minimum size=4mm},
    fac/.style={rectangle,draw,thick,minimum size=4mm}]
  \node (a1) at (2.5,4) [var1] {};
  \node (a2) at (5.5,4) [var1] {};
  \node (a3) at (1,3) [var0] {};
  \node (a4) at (4,3) [var0] {};
  \node (a5) at (7,3) [var0] {};
  \node (a6) at (2.5,2) [var1] {};
  \node (a7) at (5.5,2) [var1] {};
\begin{scope}
  \draw (0,4.3) to (2.8,4.3) to[out=0,in=33.69] (2.8,3.7) to (1.3,2.7) to[out=213.69,in=326.31] (0.7,2.7) to (0,3.167);
  \draw (0,1.7) to (2.8,1.7) to[out=0,in=326.31] (2.8,2.3) to (1.3,3.3) to[out=146.31,in=33.69] (0.7,3.3) to (0,2.833);
  \draw (8,4.3) to (5.2,4.3) to[out=180,in=146.31] (5.2,3.7) to (6.7,2.7) to[out=326.31,in=213.69] (7.3,2.7) to (8,3.167);
  \draw (8,1.7) to (5.2,1.7) to[out=180,in=213.69] (5.2,2.3) to (6.7,3.3) to[out=33.69,in=146.31] (7.3,3.3) to (8,2.833);
  \draw (0.5,1.5) to (0.5,4.5) to[out=90,in=90] (3.2,4.5) to (3.2,1.5) to[out=270,in=270] (0.5,1.5);
  \draw (1.8,0.8) to (1.8,5.2) to[out=90,in=180] (3.15,6) to[out=0,in=90] (4.5,5.2) to (4.5,0.8) to[out=270,in=0] (3.15,0) to[out=180,in=270] (1.8,0.8);
  \draw (3.5,1.5) to (3.5,4.5) to[out=90,in=90] (6.2,4.5) to (6.2,1.5) to[out=270,in=270] (3.5,1.5);
  \draw (7.5,0.8) to (7.5,5.2) to[out=90,in=0] (6.15,6) to[out=180,in=90] (4.8,5.2) to (4.8,0.8) to[out=270,in=180] (6.15,0) to[out=0,in=270] (7.5,0.8);
  \draw (8,0.7) to[out=180,in=270] (6.5,1.5) to (6.5,4.5) to[out=90,in=180] (8,5.3);
  \draw (0,0) to[out=0,in=270] (1.5,1) to (1.5,5) to[out=90,in=0] (0,6);
  \draw (2.3,4.7) to (5.7,4.7) to[out=0,in=33.69] (5.7,3.5) to (4.5,2.7) to[out=213.69,in=326.31] (3.5,2.7)
    to (2.3,3.5) to[out=146.31,in=180] (2.3,4.7);
  \draw (2.3,1.3) to (5.7,1.3) to[out=0,in=326.31] (5.7,2.5) to (4.5,3.3) to[out=146.31,in=33.69] (3.5,3.3)
    to (2.3,2.5) to[out=213.69,in=180] (2.3,1.3);
\end{scope}
\node[inner sep=0,outer sep=0] at (2.18,4) {$i_1$};
\node[inner sep=0,outer sep=0] at (5.85,4) {$i_2$};
\node[inner sep=0,outer sep=0] at (1.37,3) {$i_3$};
\node[inner sep=0,outer sep=0] at (3.8,2.75) {$i_4$};
\node[inner sep=0,outer sep=0] at (6.65,3) {$i_5$};
\node[inner sep=0,outer sep=0] at (2.18,2) {$i_6$};
\node[inner sep=0,outer sep=0] at (5.85,2) {$i_7$};
\node[inner sep=0,outer sep=0] at (0.25,3.6) {$e_{a_1}$};\draw [->] (0.25,3.7) to (0.25,4.3);
\node[inner sep=0,outer sep=0] at (4,4.2) {$e_{a_2}$};\draw [->] (4,4.3) to (4,4.7);
\node[inner sep=0,outer sep=0] at (7,3.6) {$e_{a_3}$};\draw [->] (7,3.7) to (7,4.3);
\node[inner sep=0,outer sep=0] at (1,4.5) {$e_{a_4}$};\draw [->] (1,4.6) to (1,5.13);
\node[inner sep=0,outer sep=0] at (3.15,5.2) {$e_{a_5}$};\draw [->] (3.15,5.3) to (3.15,6);
\node[inner sep=0,outer sep=0] at (5.1,1) {$e_{a_6}$};\draw [->] (5.08,1.08) to (6.03,1.08);
\node[inner sep=0,outer sep=0] at (6.15,5.2) {$e_{a_7}$};\draw [->] (6.15,5.3) to (6.15,6);
\node[inner sep=0,outer sep=0] at (0.25,2.4) {$e_{a_8}$};\draw [->] (0.25,2.3) to (0.25,1.7);
\node[inner sep=0,outer sep=0] at (4,1.8) {$e_{a_9}$};\draw [->] (4,1.7) to (4,1.3);
\node[inner sep=0,outer sep=0] at (7,2.4) {$e_{a_{10}}$};\draw [->] (7,2.3) to (7,1.7);
\begin{scope}
\path (0,0) to (8,0) to (8,6) to (0,6) to (0,0);
\end{scope}
\end{tikzpicture}
\subcaption{$4$-regular $2$-in-$3$ vertex cover}\label{related_problems_2_vertex_cover}
\end{subfigure}
\captionsetup{width=0.89\textwidth, font=small}
\caption{
The figure on the left shows the solution on a configuration corresponding to the solution in Figure \ref{f_related_problems}. We only denoted $a$-edges (small boxes, filled if they the $a$-edge takes the value one) and $i$-edges (small circles, filled if the $i$-edge takes the value one) instead of f-edges and v-edges for brevity (e.g. $h_{a_1,1}$ instead of $(a_1,h_{a_1,1})$).
The figure on the right illustrates the corresponding $2$-in-$3$ vertex cover (given by the filled circles).}\label{f_related_problems_2}
\end{figure}
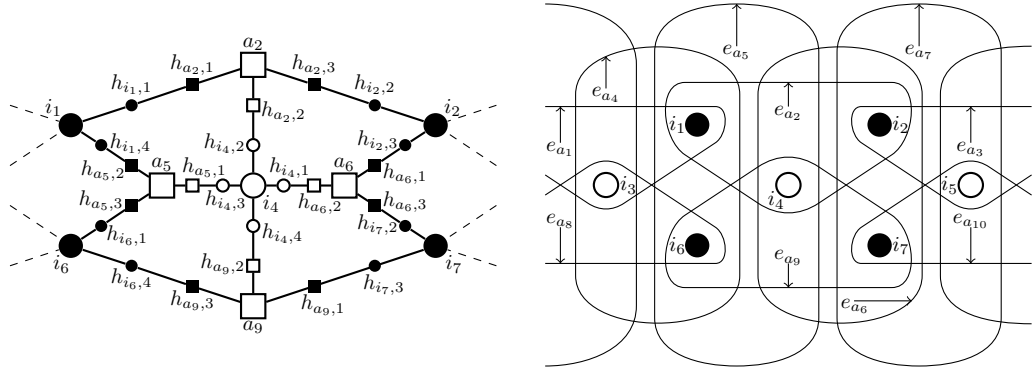
\subsection{The First Moment Method}\label{p_first_moment}
In the first step we apply the first moment method to the occupation problem on configurations, yielding the following result.
\begin{lemma}[First Moment Method]
\label{p_first_moment_results}
Let $k\in\mathbb Z_{\ge 4}$, $d\in\mathbb Z_{\ge 2}$.
For $n\in\mathcal N$ tending to infinity
\begin{flalign*}
\mathbb E[Z]&\sim
\sqrt{d}e^{n\phi_1}, \quad\textrm{where} \quad
\phi_1=\frac{d}{k}\ln\binom{k}{2}-(d-1)H\left(\frac2k\right)
\textrm{.}
\end{flalign*}
\end{lemma}
In particular, $\mathbb E[Z]\rightarrow\infty$ for $d<d^*$ and $\mathbb E[Z]\rightarrow 0$ for $d>d^*$ with $d^*$ as in \eqref{e_ws_ds_definition}.
So, Markov's inequality implies $\mathbb P[Z>0]\rightarrow 0$ for $d>d^*$. The map $\phi_1$ is known as \emph{annealed free entropy density}.
\subsection{The Second Moment Method}\label{p_second_moment}
Let $k\in\mathbb Z_{\ge 4}$ and $d\in\mathbb Z_{\ge 2}$.
We denote the set of distributions on a countable set $\mathcal S$ by $\mathcal P(\mathcal S)$ and identify $p\in\mathcal P(\mathcal S)$ with its probability mass function, meaning $\mathcal P(\mathcal S)=\{p\in[0,1]^{\mathcal S}:\sum_{x\in\mathcal S}p(s)=1\}$.
Further, let $\mathcal P_{\ell}(\mathcal S)=\{p\in\mathcal P(\mathcal S):\ell p\in\mathbb Z^{\mathcal S}\}$ be the empirical distributions over $\ell\in\mathbb Z_{>0}$ trials.

In order to apply the second moment method we will consider a (new) CSP with  $m$ factors on $n$ variables with the larger domain $\{0,1\}^2$, and where the constraint $a\in[m]$ is satisfied by an assignment $x\in(\{0,1\}^2)^n$ if $\sum_{i\in v(a)}x_{i,1}=\sum_{i\in v(a)}x_{i,2}=2$.
Here, there are qualitatively three types of satisfying assignments for the constraints, namely with $0$, $1$ or $2$ overlapping ones.
We will analyze the empirical \emph{overlap distributions}
$p\in\mathcal P(\{0,1,2\})$ of assignments satisfying all constraints, which determine the empirical distributions $p_{\mathrm{e}}\in\mathcal P_m(\{0,1\}^2)$ of the values $\{0,1\}^2$ over the $km$ edges, given by
\begin{flalign*}
p_{\mathrm{e}}(11)=\frac{1}{k}p(1)+\frac{2}{k}p(2)~~\textrm{ and }~~
p_{\mathrm{e}}(10)=p_{\mathrm{e}}(01)=\frac{1}{k}p(1)+\frac{2}{k}p(0)\textrm{.}
\end{flalign*}
So, if $p\in\mathcal P_m(\{0,1,2\})$ is indeed an achievable empirical distribution on the $m$ factors, then $p_{\mathrm{e}}$ needs to be an empirical distribution on the $n$ variables
and thus the attainable overlap distributions are 
$\mathcal P_n=\left\{p\in\mathcal P_m(\{0,1,2\}):p_{\mathrm{e}}\in\mathcal P_n(\{0,1\}^2)\right\}$.

In the first -- combinatorial -- part we establish that the second moment for $n\in\mathcal N$ can be written as a sum of contributions for fixed overlap distributions.
\begin{lemma}[Second Moment Combinatorics]
\label{p_second_moment_combinatorics}
For any $n\in\mathcal N$ we have
\begin{flalign*}
\mathbb E[Z^2]=\sum_{p \in \mathcal P_n}E(p),
\text{ where }
E(p)=\binom{m}{mp}\prod_{s\in\{0,1,2\}}\binom{k}{s,2-s,2-s,k-4+s}^{mp(s)}\binom{n}{np_{\mathrm{e}}}\binom{dn}{dnp_{\mathrm{e}}}^{-1}\textrm{.}
\end{flalign*}
Here, we use the notation $\binom{m}{mp}$, $p\in\mathcal P_m(\{0,1,2\})$, for multinomial coefficients.
\end{lemma}
To study further the second moment in Lemma \ref{p_second_moment_combinatorics} we identify the maximal contributions.
For this purpose, let $p^*\in\mathcal P(\{0,1,2\})$ be the hypergeometric distribution with
\begin{flalign*}
p^*(s)=\frac{\binom{2}{s}\binom{k-2}{2-s}}{\binom{k}{2}}  ~\textrm{ for }~ s\in\{0,1,2\},
\quad\textrm{ and }\quad
p^*_{\mathrm{e}}(11)=\frac{4}{k^2}, ~ p^*_{\mathrm{e}}(10)=\frac{2(k-2)}{k^2}\textrm{.}
\end{flalign*}
The overlap distribution $p^*$ is a natural candidate for maximizing $E(p)$.
Indeed, we obtain~$p^*$ when we consider two independent uniformly random assignments in $\{0,1\}^k$ with $2$ ones each, and $p^*_{\mathrm{e}}$ is exactly the marginal probability if we jointly consider two independent uniformly random assignments in $\{0,1\}^n$ to the variables with $2n/k$ ones each.
In the next step, we derive the limits of the log-densities $\frac{1}{n}\ln(E(p))$. Recall that the K(ullback)-L(eibler) divergence $\DKL{p}{q}$  of two distributions $p$, $q\in\mathcal P(\mathcal D)$ with finite support $\mathcal D$ is
\begin{flalign*}
\DKL{p}{q}=\sum_{x\in\mathcal D}p(x)\ln\left(\frac{p(x)}{q(x)}\right)\textrm{.}
\end{flalign*}
\begin{lemma}[Second Moment Asymptotics]
\label{p_second_moment_asymptotics}
For any fully supported $p\in\mathcal P(\{0,1,2\})$ and any sequence $(p_n)_{n\in\mathcal N}\subseteq\mathcal P_n$ with $\lim_{n\rightarrow\infty}p_n=p$ we have
$\lim_{n\rightarrow\infty}\frac{1}{n}\ln(E(p_n))=\phi_2(p)$, where
\begin{flalign*}
\phi_2(p)=2\phi_1-\frac{d}{k}\Delta_d(p)
\quad \text{and} \quad
\Delta_d(p)=\DKL{p}{p^*}-\frac{(d-1)k}{d}\DKL{p_{\mathrm{e}}}{p^*_{\mathrm{e}}}\textrm{.}
\end{flalign*}
\end{lemma}
The following proposition is the main contribution of this work.
\begin{proposition}[Second Moment Minimizers]
\label{p_second_moment_optimization}
For $k=4$ the global minimizers of $\Delta_{d^*(4)}$ are $p^*$, $p^{(0)}$ given by $p^{(0)}(0)=1$ and $p^{(2)}$ given by $p^{(2)}(2)=1$. For $k\in\mathbb Z_{\ge 5}$ the global minimizers of $\Delta_{d^*(k)}$ are $p^*$ and $p^{(2)}$.
\end{proposition}
With Proposition \ref{p_second_moment_optimization}, we easily verify that $p^*$ is the unique minimizer of $\Delta_d$ for any $d<d^*(k)$, since $\DKL{p_{\mathrm{e}}}{p^*_{\mathrm{e}}}$ is minimized by its unique root $p^*_{\mathrm{e}}$ and that $(d-1)k/d$ is increasing in $d$.
This conclusion then allows to compute the limit of the scaled second moment using Laplace's method for sums.
More than that, we confirm the conjecture by the authors in \cite{panagiotou2019} as an immediate corollary.
\begin{proposition}[Second Moment Limit]
\label{p_second_moment_results}
For any $k\in\mathbb Z_{\ge 4}$ and $d<d^*(k)$ 
\begin{flalign*}
\frac{\mathbb E[Z^2]}{\mathbb E[Z]^2}&\sim\sqrt{\frac{k-1}{k-d}},~~\textrm{ as }n\in\mathcal N\textrm{ tends to infinity.}
\end{flalign*}
\end{proposition}
Using Proposition \ref{p_second_moment_results} and the Paley-Zygmund inequality we see that
$\mathbb P[Z>0]\ge\sqrt{\frac{k-d}{k-1}}$. While this bound suggests a threshold exists, we need to show that the threshold at $d^*$ is \emph{sharp}.
\subsection{Small Subgraph Conditioning}\label{p_small_subgraph_conditioning}
We prove the theorem by applying the small subgraph conditioning method.
For this purpose let $(a)_b=\prod_{c=0}^{b-1}(a-c)$ denote the falling factorial.
\begin{theorem}[{Small Subgraph Conditioning, \cite[Theorem 2]{moore2016}}]
\label{small_subgraph_conditioning}
Let $Z_n$ and $X_{n,1},X_{n,2},\dots$ be non-negative integer-valued random variables.
Suppose that $\mathbb E[Z_n]>0$ and that
for each $\ell\in\mathbb Z_{>0}$ there are $\lambda_\ell\in\mathbb R_{>0}$, $\delta_\ell\in\mathbb R_{>-1}$ such that
for any $\bar\ell\in\mathbb Z_{>0}$
\begin{StatementList}
\item\label{small_subgraph_conditioning_cycle_dist}
 the variables $X_{n,1},\dots,X_{n,\bar\ell}$ are asymptotically independent and Poisson with
$\mathbb E[X_{n,\ell}]\sim\lambda_\ell$,
\item\label{small_subgraph_conditioning_joint_dist}
for any sequence $r_1,\dots,r_{\bar\ell}$ of non-negative integers,
\begin{flalign*}
\frac{\mathbb E\left[Z_n\prod_{\ell=1}^{\bar\ell}\left(X_{n,\ell}\right)_{r_\ell}\right]}{\mathbb E[Z_n]}\sim\prod_{\ell=1}^{\bar\ell}[\lambda_\ell(1+\delta_\ell)]^{r_\ell},
\end{flalign*}
\item\label{small_subgraph_conditioning_variance_explained}
we explain the variance, i.e.
\begin{flalign*}
\frac{\mathbb E[Z_n^2]}{\mathbb E[Z_n]^2}\sim\exp\left(\sum_{\ell \ge 1}\lambda_\ell\delta_\ell^2\right)
\quad\textrm{ and }\quad
\sum_{\ell \ge 1}\lambda_\ell\delta_\ell^2<\infty\textrm{.}
\end{flalign*}
\end{StatementList}
Then  $\lim_{n\rightarrow\infty}\mathbb P[Z_n>0]=1$.
\end{theorem}
We will apply Theorem \ref{small_subgraph_conditioning} to the number $Z$ of solutions from Section \ref{p_configuration_model} and the numbers $X_\ell$ of small cycles in the configuration $G$.
In order to understand what a cycle in a configuration is, we recall the representation of a configuration $g$ as a four-partite graph from Section \ref{p_configuration_model}.
Since we are mostly interested in the factor graph associated with a configuration we divide lengths of paths by three, e.g.~what we call a cycle of length four in the bijection, is actually a cycle of length twelve in its equivalent four-partite graph representation.
Figures \ref{f_related_problems_occupation} and \ref{f_related_problems_2_configuration} show an example of a factor graph and the corresponding configuration in its graph representation. Showing the following statement, which establishes Assumption \ref{small_subgraph_conditioning}\ref{small_subgraph_conditioning_cycle_dist}, is rather routine.
\begin{lemma}[Small Cycles]
\label{lem_number_of_cycles}
For $\ell\in\mathbb Z_{>0}$ let $X_\ell$ be the number of $2\ell$-cycles in $G$, and set
\begin{flalign*}
\lambda_\ell=\frac{[(k-1)(d-1)]^\ell}{2\ell}.
\end{flalign*}
Then $X_1,\dots,X_{\bar\ell}$ are asymptotically independent and Poisson with $\mathbb E[X_\ell]\sim\lambda_\ell$ for all $\bar\ell\in\mathbb Z_{>0}$.
\end{lemma}
We give a self-contained proof of Lemma \ref{lem_number_of_cycles} in the appendix,
which we build upon to argue that Assumption \ref{small_subgraph_conditioning}\ref{small_subgraph_conditioning_joint_dist} in Theorem \ref{small_subgraph_conditioning} holds.
With Lemma \ref{lem_number_of_cycles} in place, we consider the base case in Assumption \ref{small_subgraph_conditioning}\ref{small_subgraph_conditioning_joint_dist}, i.e.~for $\ell\in\mathbb Z_{>0}$ we let $r_\ell=1$ and $r_{\ell'}=0$ otherwise, to determine $\delta_\ell=(1-k)^{-\ell}$.
We easily verify that $\sum_{\ell\ge 1}\lambda_\ell\delta_\ell^2=\frac12\ln(\frac{k-1}{k-d})$ and thereby establish Assumption \ref{small_subgraph_conditioning}\ref{small_subgraph_conditioning_variance_explained} using Proposition \ref{p_second_moment_results}.
Finally, we follow the proof of Lemma \ref{lem_number_of_cycles} to complete the verification of Assumption \ref{small_subgraph_conditioning}\ref{small_subgraph_conditioning_joint_dist} and thereby complete the proof of Theorem \ref{theorem_satisfiability_threshold}.
%
%
%
%
%
%
\section{Preliminaries and Notation}\label{m_preliminaries}
After introducing notation in Section \ref{m_notation}, we establish a few basic facts in Section \ref{m_occupation_problem_configuration}.
%
%
\subsection{Notation}\label{m_notation}
We use the notation
$[n]=\{1,\dots,n\}$ for $n\in\mathbb Z_{>0}$,
denote the falling factorial (or $k$-factorial) with $(n)_k$ for $n,k\in\mathbb Z_{\ge 0}$, $k\le n$, and
multinomial coefficients with $\binom{n}{k}$ for $n\in\mathbb Z_{\ge 0}$ and $k\in\mathbb Z_{\ge 0}^d$, $d\in\mathbb Z_{>1}$, such that $\sum_{i\in[d]}k_i=n$.
For functions $f$, $g$ on integers with $\lim_{n\rightarrow\infty}{f(n)}/{g(n)}=1$ we write $f(n)\sim g(n)$.
We make heavy use of Stirling's formula \cite{robbins1955}, i.e.
\begin{flalign*}
\sqrt{2\pi n}\left(\frac{n}{e}\right)^ne^{\frac{1}{12n+1}}\le n!\le
\sqrt{2\pi n}\left(\frac{n}{e}\right)^ne^{\frac{1}{12n}},
\quad n\in\mathbb Z_{>0},
\end{flalign*}
and in particular $n!\sim \sqrt{2\pi n}(\frac{n}{e})^n$.
If a random variable $X$ has law $P$ we write $X\sim P$ and use $\Po(\lambda)$ to denote the Poisson distribution with parameter $\lambda$.
We use
`wlog' for `without loss of generality',
`a.s.'~for `almost surely', i.e.~$P(\mathcal E)=0$ for an event $\mathcal E$, and 
`a.a.s.'~for `asymptotically almost surely', i.e.~$\lim_{n\rightarrow\infty}P(\mathcal E_n)=1$ for a sequence of events $\mathcal E_n$.
Distributions $p\in\mathcal P(\mathcal S)$ in the convex polytope $\mathcal P(\mathcal S)$ of distributions with finite support $\mathcal S$ are identified with their probability mass functions $p\in[0,1]^{\mathcal S}$.
Further, $\mathcal P_n(\mathcal S)=\{p\in\mathcal P(\mathcal S):np\in[n]_0\}$ denotes the set of empirical distributions obtained from $n\in\mathbb Z_{\ge 1}$ trials.
Finally, let $v^{\mathrm{t}}$ denote the transpose of a vector $v$.
%
%
\subsection{Basic Observations}\label{m_occupation_problem_configuration}
We briefly establish the claims in Section \ref{s_occupation_problems} for the configuration version, and the claim that $d^*$ is not an integer.
\begin{lemma}\label{lem_satthresh_basics}
The set $\mathcal G$ is empty iff $dn\neq km$, so let $dn=km$.
Then,  we have $Z=0$ almost surely if $n_1=2n/k\not\in\mathbb Z$, for $Z$ as defined in this section.
Finally, $d^*\in(1,\infty)\setminus\mathbb Z$.
\end{lemma}
\begin{proof}
Since $g\in\mathcal G$ is a bijection $g:[n]\times[d]\rightarrow[m]\times[k]$, the set $\mathcal G$ is empty for $dn\neq km$ and $|\mathcal G|=(dn)!=(km)!$ otherwise, which proves the first assertion.
Next, we fix a solution $x\in\{0,1\}^n$ of $g\in\mathcal G$ with $n'_1$ ones.
Then two $a$-edges $h$ have to take the value one, i.e.~$x_{i(a,h)}=1$, for each $a\in F$ and hence $2m$ f-edges $(a,h)\in[m]\times[k]$ in total.
On the other hand, there are $dn'_1$ v-edges $(i,h)\in[n]\times[d]$ that take the value one.
Since $g$ is a bijection, $dn'_1=2m$, so $n_1=n'_1\in\mathbb Z$.

For the last assertion, we first focus on the denominator of $d^*$, i.e.
\begin{flalign*}
kH\left(\frac{2}{k}\right)-\ln\binom{k}{2}=-\ln\left(\binom{k}{2}\left(\frac{2}{k}\right)^2\left(\frac{k-2}{k}\right)^{k-2}\right)>0\textrm{,}
\end{flalign*}
so $d^*>0$ for $k\in\mathbb Z_{\ge 3}$.
Next, notice that $d^*$ is a solution of $f(d)=1$ with
\begin{flalign*}
f(d)=e^{(d-1)(kH(2/k)-\ln\binom{k}{2})-\ln\binom{k}{2}}=\frac{2}{k(k-1)}\left(\frac{k^{k-1}}{2(k-2)^{k-2}(k-1)}\right)^{d-1}\textrm{,}
\end{flalign*}
which directly implies that $d^*>1$
and further, since $\gcd(k,k-1)=1$, that $d^*\in(1,\infty)\setminus\mathbb Z$.
\end{proof}
\section{The First Moment Method -- Proof of Lemma~\ref{p_first_moment_results}}
\label{s_first_moment}
This short section is dedicated to the proof of Lemma \ref{p_first_moment_results}.
We write the expectation in terms of the number $|\mathcal E|$ of pairs $(g,x)\in\mathcal E$ such that $x\in\{0,1\}^n$ satisfies $g\in\mathcal G$, i.e.
\begin{flalign*}
\mathbb E[Z]
=\frac{|\mathcal E|}{|\mathcal G|}
=\frac{1}{(dn)!}\binom{n}{n_1}\binom{k}{2}^m(2m)!(dn-2m)!\textrm{,}
\end{flalign*}
with $n_1=2k^{-1}n$ and for the following reasons.
First, we choose the $n_1$ variables with value one in~$x$,
then we choose the two $a$-edges for each constraint $a\in[m]$ with value one,
wire the v-edges and f-edges with value one
and finally wire the edges with value zero.
In particular, this implies that $\mathbb E[Z]>0$ for all $n\in\mathcal N$.
Stirling's formula yields after some straightforward but tedious manipulations 
$\mathbb E[Z]\sim
\sqrt{d}e^{n\phi_1}$,
as claimed.
%
%
%
\section{The Second Moment Method}\label{s_second_moment}
In this section we consider the case $d<d^*$.
We prove Lemma \ref{p_second_moment_combinatorics}, Lemma \ref{p_second_moment_asymptotics}, Proposition \ref{p_second_moment_optimization} and Proposition \ref{p_second_moment_results}, the main contribution of this work.
\subsection{How to Square a Constraint Satisfaction Problem}\label{p_second_moment_CSP}
In order to facilitate the presentation we introduce
the \emph{squared} $d$-regular $2$-in-$k$ occupation problem.
As before, an instance of this problem is given by a bijection $g:[n]\times[d]\rightarrow[m]\times[k]$.
Now, for an assignment $x\in(\{0,1\}^2)^n$ let
$y_{g,x}=(x_{i(a,h)})_{a\in[m],h\in[k]}$ be the corresponding f-edge assignment under $g$,
where we recall from Section \ref{p_configuration_model} that $i(a,h)=(g^{-1}(a,h))_1\in[n]$ is the variable $i(a,h)$ wired to the f-edge $(a,h)$ under~$g$.
A constraint $a\in[m]$ is satisfied by a constraint assignment $x\in(\{0,1\}^2)^k$ iff $x\in\mathcal S^{(2)}$, where
\begin{flalign*}
\mathcal S^{(2)}
=\left\{x\in(\{0,1\}^2)^k:\sum_{h\in[k]}x_{h,1}=\sum_{h\in[k]}x_{h,2}=2\right\}\textrm{.}
\end{flalign*}
An f-edge assignment $x\in(\{0,1\}^2)^{[m]\times[k]}$ is satisfying if $x_a=(x_{a,h})_{h\in[k]}$ satisfies $a$ for all $a\in[m]$.
Finally, an assignment $x\in(\{0,1\}^2)^n$ is a solution of $g$ if $y_{g,x}$ is satisfying.
Notice that the pairs of solutions $x$, $x'\in\{0,1\}^n$ of the standard problem on $g$ are in one to one correspondence with the solutions $y\in(\{0,1\}^2)^n$ of the squared problem on $g$ via $y=(x_i,x'_i)_{i\in[n]}$.
So, $z^{(2)}(g)=z(g)^2$ for the number $z^{(2)}(g)$ of solutions of the squared problem, hence $Z^{(2)}=Z^2$ for $Z^{(2)}=z^{(2)}(G)$ and in particular $\mathbb E[Z^{(2)}]=\mathbb E[Z^2]$.
This equivalence allows us to entirely focus on the squared problem.
\subsection{Proof of Lemma \ref{p_second_moment_combinatorics}}\label{p_second_moment_combinatorics_proof}
As before, we can write $\mathbb E[Z^{(2)}]=\frac{1}{(dn)!}|\mathcal E|$, where $|\mathcal E|$ is the number of pairs $(g,x)\in\mathcal E$ such that $x\in(\{0,1\}^2)^n$ solves $g$.
Set
\begin{flalign*}
\mathcal Y=\left\{y\in(\{0,1\}^2)^{[m]\times[k]}:y_a\in\mathcal S^{(2)}~\textrm{for all}~a\in[m]\right\}.
\end{flalign*}
For $y\in\mathcal Y$ let the \emph{overlap distribution} $p_y\in\mathcal P_m(\{0,1,2\})$ be given by
\begin{flalign*}
p_y(s)=\frac{1}{m}|\{a\in[m]:|y_a^{-1}(1,1)|=s\}|, \quad s\in\{0,1,2\}.
\end{flalign*}
Further, let the \emph{edge distribution} $q_y\in\mathcal P_{km}(\{0,1\}^2)$ be given by
\begin{flalign*}
q_y(x)=\frac{1}{km}|\{(a,h)\in[m]\times[k]:y_{a,h}=x\}|=\frac{1}{km}|y^{-1}(x)|,\quad  x\in\{0,1\}^2\textrm{.}
\end{flalign*}
Using that $|y_a^{-1}(1,0)|=|y_a^{-1}(0,1)|=2-|y_a^{-1}(1,1)|$ and hence $|y^{-1}(0,0)|=k-4+|y^{-1}(1,1)|$
we directly get
\begin{flalign*}
q_y(1,1)&=\frac{1}{km}\sum_{a\in[m]}|y_a^{-1}(1,1)|=\frac{1}{km}\sum_{s\in\{0,1,2\}}s|\{a\in[m]:|y_a^{-1}(1,1)|=s\}|=\sum_{s\in\{0,1,2\}}\frac{s}{k}p_y(s)\textrm{,}\\
q_y(1,0)&=q_y(0,1)=\frac{1}{km}\sum_{s\in\{0,1,2\}}(2-s)|\{a\in[m]:|y_a^{-1}(1,1)|=s\}|=\sum_{s\in\{0,1,2\}}\frac{2-s}{k}p_y(s)\textrm{,}\\
q_y(0,0)&=\sum_{s\in\{0,1,2\}}\frac{k-4+s}{k}p_y(s)\textrm{.}
\end{flalign*}
Hence, let $p_{\mathrm{e}}=Wp\in\mathcal P(\{0,1\}^2)$ denote the edge distribution of any (not necessarily empirical) overlap distribution $p\in\mathcal P(\{0,1,2\})$, where $W\in[0,1]^{\{0,1\}^2\times\{0,1,2\}}$ is given by
\begin{flalign}\label{eq_Wdefinition}
W_{11,s}=\frac{s}{k}\textrm{, }
~~W_{10,s}=W_{01,s}=\frac{2-s}{k}
~\textrm{ and }~W_{00,s}=\frac{k-4+s}{k},
\quad s\in\{0,1,2\}\textrm{.}
\end{flalign}
Now, notice that for any $(g,x)\in\mathcal E$ we have $y_{g,x,a,h}=x_{i(a,h)}$ for all $a\in[m]$, $h\in[k]$, hence
$g(x^{-1}(z)\times[d])=y_{g,x}^{-1}(z)$ and by that
\begin{flalign*}
q_{y_{g,x}}(z)=\frac{|y_{g,x}^{-1}(z)|}{km}=\frac{d|x^{-1}(z)|}{km}=\frac{1}{n}|x^{-1}(z)|=q_{x}(z)\textrm{ for }z\in\{0,1\}^2\textrm{,}
\end{flalign*}
i.e.~the relative frequencies of the values in the f-edge assignment $y_{g,x}$ coincide with the relative frequencies of the values in the variable assignment $x$.
In particular, this shows that a satisfying f-edge assignment $y\in\mathcal Y$ is only \emph{attainable} if $q_y\in\mathcal P_n(\{0,1\}^2)$.
This shows that
\begin{align*}
\mathbb E[Z^2]=\sum_{p\in\mathcal P_n}|\{(g,x)\in\mathcal E:p_{y_{g,x}}=p\}|.
\end{align*}
Now, fix an attainable satisfying f-edge assignment $y\in\mathcal Y$ and an assignment $x\in(\{0,1\}^2)^n$ with $q_x=q_y$, i.e.~$|x^{-1}(z)\times[d]|=|y^{-1}(z)|$ for all $z\in\{0,1\}^2$.
Any bijection $g$ with $y=y_{g,x}$ needs to respect $g(x^{-1}(z)\times[d])=y_{g,x}^{-1}(z)$ for $z\in\{0,1\}^2$ and can hence be uniquely decomposed into its restrictions $g_z:x^{-1}(z)\times[d]\rightarrow y_{g,x}^{-1}(z)$. On the other hand, any choice of such restrictions $g_z$ gives a bijection $g$ with $y=y_{g,x}$, and so 
\begin{flalign*}
|\mathcal E_{x,y}|&=\prod_{z\in\{0,1\}^2}(dnq_x(z))!=\prod_{z\in\{0,1\}^2}(dnp_{y,\mathrm{e}}(z))!~~\textrm{, where }~
\mathcal E_{x,y}=\{(g,x)\in\mathcal E:y_{g,x}=y\}\textrm{.}
\end{flalign*}
Notice that $\mathcal E_{x,y}\cap\mathcal E_{x',y'}=\emptyset$ for any $(x,y)\neq(x',y')$,
which is obvious for $x\neq x'$, and also for $y\neq y'$, since $y_{g,x}=y\neq y'=y_{g',x}$ implies that $g\neq g'$.
But since $|\mathcal E_{x,y}|$ only depends on $p_{y}$ (actually only on $p_{y,\mathrm{e}}$) this completes the proof, because for any fixed attainable overlap distribution $p\in\mathcal P_n$,  we can now independently choose the satisfying f-edge assignment $y$ and variable assignment $x$, subject to $q_x=p_{\mathrm{e}}$ and $p_y=p$ (which implies $q_y=q_x$). So we have $\mathbb E[Z^{(2)}]=\sum_{p}E(p)$ with $p\in\mathcal P_n$ and
\begin{flalign*}
E(p)=\frac{1}{(dn)!}\binom{n}{np_{\mathrm{e}}}\binom{m}{mp}\prod_{s\in\{0,1,2\}}\binom{k}{s,2-s,2-s,k-4+s}^{mp(s)}\prod_{x\in\{0,1\}^2}(dnp_{\mathrm{e}}(z))!\textrm{,}
\end{flalign*}
where we choose a variable assignment $x$ with $q_x=p_{\mathrm{e}}$, an f-edge assignment $y$ with $p_y=p$ by first choosing one of the $\binom{m}{mp}$ options for $(|y_a^{-1}(1,1)|)_{a\in[m]}$ and then independently one of the $\binom{k}{s,2-s,2-s,k-4+s}$ satisfying constraint assignments for each of the $mp(s)$ constraints with overlap $s\in\{0,1,2\}$, and finally choosing a bijection $g$ with $(g,x)\in\mathcal E_{x,y}$.
\subsection{Empirical Overlap Distributions}\label{p_second_moment_empirical_overlap_distributions}
This section is dedicated to deriving properties of the set $\mathcal P_n$ for $n\in\mathcal N$.
In the following we will use the canonical ascending order on $\{0,1,2\}$ to denote points in $\mathbb R^{\{0,1,2\}}$ and the ascending lexicographical order on $\{0,1\}^2$ to denote points in $\mathbb R^{\{0,1\}^2}$.
Recall that $p^{(s)}\in\mathcal P(\{0,1,2\})$ given by $p^{(s)}(s)=1$ for $s\in\{0,1,2\}$ denote the corners of the convex polytope $\mathcal P(\{0,1,2\})$ and further consider the vectors in $\mathbb R^{\{0,1,2\}}$
\begin{align}\label{eq_basisdef}
	b_1=(-d,d,0)^{\mathrm{t}}, \quad b_2=(1,-2,1)^{\mathrm{t}}.
\end{align}
Finally, let
\[
	\mathcal X=\Big\{x\in\mathbb R^2:(b_1,b_2)x\ge-p^{(0)}\Big\},
	\quad
	\mathcal X_n=\mathcal X\cap(m^{-1}\mathbb Z)^2.
\]
\begin{lemma}\label{p_second_moment_empirical_overlap_distributions_characterization}
The map $\iota_n:\mathcal X_n\rightarrow\mathcal P_n$, $x\mapsto p^{(0)}+(b_1,b_2)x$ is a bijection.
\end{lemma}
\begin{proof}
With the shorthand $1_{\{0,1,2\}}=(1)_{s\in\{0,1,2\}}$, further
\begin{flalign*}
1_{\{0,1,2\}}^{\perp}&=\left\{x\in\mathbb R^{\{0,1,2\}}:1_{\{0,1,2\}}^{\mathrm{t}}x=0\right\}
=\left\{x\in\mathbb R^{\{0,1,2\}}:\sum_{s\in\{0,1,2\}}x_s=0\right\}.
\end{flalign*}
Note that $\mathcal P(\{0,1,2\})\subseteq p^{(0)}+1_{\{0,1,2\}}^{\perp} = \{p^{(0)}+x:x\in 1_{\{0,1,2\}}^{\perp}\}$. On the other hand, $(b_1,b_2)$ is a basis of $1_{\{0,1,2\}}^{\perp}$, and hence
\begin{equation}
\label{eq:iota}
	\iota:\mathbb R^2\rightarrow p^{(0)}+1_{\{0,1,2\}}^{\perp}, \quad x\mapsto p^{(0)}+(b_1,b_2)x, 
\end{equation}
is bijective.
This gives that $\iota(\mathcal X)=\mathcal P(\{0,1,2\})$ and that $\iota_n$ is the restriction of $\iota$ to $\mathcal X_n$, so $\iota_n$ is a bijection from $\mathcal X_n$ to $\mathcal P(\{0,1,2\})\cap\iota\left((m^{-1}\mathbb Z)^2\right)$.
Consequently, it remains to show that $\mathcal P_n=\mathcal P(\{0,1,2\})\cap\iota\left((m^{-1}\mathbb Z)^2\right)$, where
\begin{flalign*}
\iota\left((m^{-1}\mathbb Z)^2\right)=\left\{p^{(0)}+\frac{i_1}{m}b_1+\frac{i_2}{m}b_2:i\in\mathbb Z^2\right\}
\end{flalign*}
is a grid anchored at $p^{(0)}$ and spanned by $m^{-1}b_1$ and $m^{-1}b_2$.
Note that 
\begin{flalign*}
p^{(0)}_{\mathrm{e}}=\left(\frac{k-4}{k},\frac{2}{k},\frac{2}{k},0\right)^{\mathrm{t}}\textrm{,}
\end{flalign*}
so $np^{(0)}_{\mathrm{e}}\in\mathbb Z^{\{0,1\}^2}$ since $n\in\mathcal N$, and hence $p^{(0)}\in\mathcal P_n$ by the definition of $\mathcal P_n$.
Next, we show that $\mathcal P_n$ is on the grid, i.e.~$\mathcal P_n\subseteq\iota\left((m^{-1}\mathbb Z)^2\right)$.
For this purpose fix $p\in\mathcal P_n$ and let $x=\iota^{-1}(p)$, i.e.~$mp\in\mathbb Z^{\{0,1,2\}}$, $n(Wp)\in\mathbb Z^{\{0,1\}^2}$ and $p=p^{(0)}+x_1b_1+x_2b_2$. This directly gives $mx_2=mp(2)\in\mathbb Z$. Further, we notice that $b_2$ is in the kernel of $W$ from Equation \eqref{eq_Wdefinition}, i.e.~$Wb_2=0_{\{0,1\}^2}$, and $Wb_1=\frac{d}{k}w$ with $w=(1,-1,-1,1)^{\mathrm{t}}$.
This directly gives $p_{\mathrm{e}}(1,1)=0+\frac{d}{k}x_1+0$ and hence $mx_1=np_{\mathrm{e}}(1,1)\in\mathbb Z$, i.e.~$x\in(m^{-1}\mathbb Z)^2$ and hence $p=\iota(x)\in\iota\left((m^{-1}\mathbb Z)^2\right)$.
Conversely, for any $x\in\mathcal X_n$ and with $p=\iota(x)$ we have $p\in\mathcal P(\{0,1,2\})$ since $x\in\mathcal X$, further $mp=mp^{(0)}+(b_1,b_2)(mx)\in\mathbb Z^{\{0,1,2\}}$ since $mx\in\mathbb Z^2$ and the other terms on the right-hand side are integer valued by definition, and finally $np_{\mathrm{e}}=np^{(0)}_{\mathrm{e}}+mx_1w\in\mathbb Z^{\{0,1\}^2}$.
\end{proof}
Using Lemma \ref{p_second_moment_empirical_overlap_distributions_characterization} we have
$\mathbb E[Z^{(2)}]=\sum_{x\in\mathcal X_n}E(\iota_n(x))$, where $\mathcal X_n\subseteq\mathbb R^2$ may be considered as a normalization of the grid $\mathcal P_n\subseteq p^{(0)}+1_{\{0,1,2\}}^{\mathrm{t}}$.
In order to prepare the upcoming asymptotics of the second moment, we give a complete characterization of the convex polytope $\mathcal X$ and the image of $\mathcal X$ under $W(b_1,b_2)$, i.e.~the image $p_{\mathrm{e}}=Wp$ of $p\in\mathcal P(\{0,1,2\})$ under $W$ from Equation \eqref{eq_Wdefinition}.
Let $w=(1,-1,-1,1)^{\mathrm{t}}$ from the proof of Lemma \ref{p_second_moment_empirical_overlap_distributions_characterization}, and set
\[
	\mathcal W = \Big\{p^{(0)}_{\mathrm{e}}+yw:y\in[0,2/k]\Big\}\subseteq\mathcal P(\{0,1\}^2),
	\quad \mathcal X_p=\left\{x\in\mathcal X:x_1=\frac{k}{d}p(1,1)\right\} \text{ for } p\in\mathcal W.
\]
Moreover, recall the bijection $\iota$ from~\eqref{eq:iota} and the point and
\begin{flalign*}
x^{(0)}=\begin{pmatrix}0\\ 0\end{pmatrix}\textrm{, }
x^{(1)}=d^{-1}\begin{pmatrix}1\\ 0\end{pmatrix}\textrm{, }
x^{(2)}=d^{-1}\begin{pmatrix}2\\ d\end{pmatrix}\in\mathbb R^2~\textrm{ and }~
x^*=\iota^{-1}(p^*)\textrm{.}
\end{flalign*}
\begin{lemma}\label{p_second_moment_empirical_overlap_distributions_characterization_continuous}
The set $\mathcal X$ is a two-dimensional convex polytope with corners $x^{(0)},x^{(1)},x^{(2)}$, and $x^*$ is in the interior of $\mathcal X$. The image of $\mathcal X$ under $W(b_1,b_2)$ is the one-dimensional convex polytope $\mathcal W$ with corners $p^{(0)}_{\mathrm{e}}$ and $p^{(2)}_{\mathrm{e}}$. Further, the pre-image of $p\in\mathcal W$ under $W(b_1,b_2)$ is $\mathcal X_p$, where $\mathcal X_{p^{(s)}_{\mathrm{e}}}=\{p^{(s)}\}$ for $s\in\{0,2\}$ and the intersection of $\mathcal X_p$ with the interior of $\mathcal X$ is non-empty otherwise.
\end{lemma}
\begin{proof}
Notice that $\iota(x^{(s)})=p^{(s)}$ for $s\in\{0,1,2\}$, so since $\mathcal P(\{0,1,2\})$ is the convex hull of its corners $p^{(s)}$, $s\in\{0,1,2\}$, we have that $\mathcal X$ is the convex hull of $x^{(s)}$, $s\in\{0,1,2\}$, i.e.~a two-dimensional convex polytope with corners $x^{(s)}$, since $\iota$ is a bijection and linear up to translation. In particular this also directly yields that $x^*$ is in the interior of $\mathcal X$.
Further, this shows that for any $x\in\mathcal X$ we have $x_1\ge 0$ with equality iff $x=x^{(0)}$ and further $x_1\le\frac{2}{d}$ with equality iff $x=x^{(2)}$.
Using $Wb_2=0_{\{0,1\}^2}$ and $Wb_1=\frac{d}{k}w$ from the proof of Lemma \ref{p_second_moment_empirical_overlap_distributions_characterization} we directly get that
\begin{flalign*}
W(b_1,b_2)x=p^{(0)}_{\mathrm{e}}+\frac{d}{k}x_1w\textrm{ with }\frac{d}{k}x_1\in[0,2/k]\textrm{,}
\end{flalign*}
hence the image of $\mathcal X$ under $W(b_1,b_2)$ is a subset of $\mathcal W$. Conversely, for $y\in[0,2/k]$ and $x=\frac{k}{d}yx^{(2)}\in\mathcal X$ we have $W(b_1,b_2)x=p^{(0)}_{\mathrm{e}}+yw$, which shows that $\mathcal W$ is the image of $\mathcal X$ under $W(b_1,b_2)$. This also shows that $\mathcal X_p$ is the pre-image of $p\in\mathcal W$, since for $y\in[0,1/k]$ and $p=p^{(0)}_{\mathrm{e}}+yw$ we have $p(1,1)=y$. This in turn directly yields that $\mathcal X_{p^{(s)}_{\mathrm{e}}}=\{p^{(s)}\}$ for $s\in\{0,2\}$. To see that $\mathcal X_p$ contains interior points of $\mathcal X$ otherwise, we can consider non-trivial convex combinations of $x^*$ and $x^{(0)}$ for $\frac{k}{d}p(1,1)<x^*_1$ and non-trivial convex combinations of $x^*$ and $x^{(2)}$ for $\frac{k}{d}p(1,1)>x^*_1$, which are points in the interior of $\mathcal X$.
\end{proof}
Notice that in the two-dimensional case at hand, the proof of Lemma \ref{p_second_moment_empirical_overlap_distributions_characterization_continuous} is overly formal. The set $\mathcal X$ is simply (the convex hull of) the triangle given by $x^{(s)}$, $s\in\{0,1,2\}$, with $\mathcal X_p$ given by the vertical lines in $\mathcal X$ with $x_1=\frac{d}{k}p(1,1)$.
Further, the set $\mathcal X_n$ is a canonical discretization of $\mathcal X$ in that it is given by the points of the grid $(m^{-1}\mathbb Z)^2$ contained in the triangle $\mathcal X$.
\subsection{Proof of Lemma \ref{p_second_moment_asymptotics}}\label{p_second_moment_asymptotics_proof}
We derive Lemma \ref{p_second_moment_asymptotics} from the following stronger assertion.
\begin{lemma}\label{p_second_moment_asymptotics_uniform}
Let $\mathcal U\subseteq\mathcal P(\{0,1,2\})$ be a subset with non-empty interior and such that the closure of $\mathcal U$ is contained in the interior of $\mathcal P(\{0,1,2\})$. Then there exists a constant $c=c(\mathcal U)\in\mathbb R_{>0}$ such that for all $n\in\mathcal N$ and all $p\in\mathcal P_n\cap\mathcal U$ we have $\tilde E(p)e^{-c/n}\le E(p)\le\tilde E(p)e^{c/n}$, where
\begin{flalign*}
\tilde E(p)=\sqrt{\frac{d^3}{(2\pi)^2m^2\prod_{s}p(s)}}e^{n\phi_2(p)}\textrm{.}
\end{flalign*}
\end{lemma}
\begin{proof} 
Let $\mathcal C$ denote the closure of $\mathcal U$ and $\pi_s:\mathcal C\rightarrow[0,1]$, $p\mapsto p(s)$ the projection for $s\in\{0,1,2\}$. Since $\mathcal C$ is compact, the continuous map $\pi_s$ attains its maximum $p_+(s)$ and its minimum $p_-(s)$, which directly gives $0<p_-(s)<p_+(s)<1$ since all $p\in\mathcal C$ are fully supported and the interior of $\mathcal C$ is non-empty (that gives the second inequality).
Using Lemma \ref{p_second_moment_empirical_overlap_distributions_characterization_continuous},
the continuous map $\pi:\mathcal C\rightarrow[0,2/k]$, $p\mapsto p_{\mathrm{e}}(1,1)$, and the same reasoning as above we obtain the maximum $p_{\mathrm{e},+}(1,1)$ and minimum $p_{\mathrm{e},-}(1,1)$ of $\pi$ with $0<p_{\mathrm{e},-}(1,1)<p_{\mathrm{e},+}(1,1)<2/k$, which directly give the bounds $p_{\mathrm{e},-}(x)$, $p_{\mathrm{e},+}(x)>0$ for $x\in\{0,1\}^2$ as functions of $p_{\mathrm{e},+}(1,1)$ and $p_{\mathrm{e},-}(1,1)$.
Now, we can use these bounds with the Stirling bound to obtain a constant $c\in\mathbb R_{>0}$ such that for all $n\in\mathcal N$ and $p\in\mathcal P_n\cap\mathcal C$ we have
$E'(p)e^{-c/n}\le E(p')\le E'(p)e^{c/n}$, where
\begin{flalign*}
E'(p)&=\sqrt{\frac{2\pi md^3}{\prod_{s}(2\pi mp'(s))}}\prod_{s\in\{0,1,2\}}\binom{k}{s,2-s,2-s,k-4+s}^{mp'(s)}
e^{mH(p')-(d-1)nH(p'_{\mathrm{e}})}\\
&=\sqrt{\frac{d^3}{(2\pi)^2m^2\prod_{s}p'(s)}}e^{2m\ln\binom{k}{2}-m\DKL{p'}{p^*}-(d-1)nH(p'_{\mathrm{e}})}\textrm{.}
\end{flalign*}
To see that $E'(p)=\tilde E(p)$, we observe that $\DKL{p_{\mathrm{e}}}{p^*_{\mathrm{e}}}=2H(2/k)-H(p_{\mathrm{e}})$,
since $H(p^*_{\mathrm{e}})=2H(2/k)$, $p_{\mathrm{e}}(1,0)=p_{\mathrm{e}}(0,1)$, and $p_{\mathrm{e}}(1,1)+p_{\mathrm{e}}(1,0)=2/k$ for any $p\in\mathcal P(\{0,1,2\})$.
\end{proof}
Now, Lemma \ref{p_second_moment_asymptotics} is an immediate corollary.
To see this, fix a fully supported overlap distribution $p\in\mathcal P(\{0,1,2\})$ and a sequence $(p_n)_{n\in\mathcal N}\subseteq\mathcal P_n$ converging to $p$, e.g.~$p_n=\iota(m^{-1}\lfloor mx_1\rfloor,m^{-1}\lfloor mx_2\rfloor)$ with $\iota(x)=p$ and $n$ sufficiently large. Further, fix a neighbourhood $\mathcal U$ of $p$ as described in Lemma \ref{p_second_moment_asymptotics_uniform}, which is possible since $p$ is fully supported. Then we have $p_n\in\mathcal U$ for sufficiently large $n$, hence with the continuity of $\phi_2$ we have
\begin{flalign*}
\lim_{n\rightarrow\infty}\frac{1}{n}\ln(E(p_n))
=\lim_{n\rightarrow\infty}\frac{1}{n}\ln(\tilde E(p_n))
=\lim_{n\rightarrow\infty}\phi_2(p_n)=\phi_2(p)\textrm{.}
\end{flalign*}
\subsection{Proof of Proposition \ref{p_second_moment_results}}\label{s_second_moment_results}
We postpone the proof of Proposition \ref{p_second_moment_optimization} and continue with Laplace's method for sums using the result. We obtain that
\begin{flalign*}
\frac{\mathbb E[Z^2]}{\mathbb E[Z]^2}=\sum_pe(p)\textrm{, } \quad
e(p)=\frac{E(p)}{\mathbb E[Z]^2}
=\frac{\binom{2n/k}{p_{\mathrm{e}}(1,1)}\binom{(k-2)n/k}{np_{\mathrm{e}}(0,1)}\binom{dn}{2dn/k}\binom{m}{mp}\prod_sp^*(s)^{mp(s)}}
{\binom{2dn/k}{dnp_{\mathrm{e}}(1,1)}\binom{(k-2)dn/k}{dnp_{\mathrm{e}}(0,1)}\binom{n}{2n/k}}\textrm{,}
\end{flalign*}
where the sum is over $p\in\mathcal P_n$.
First, we use Proposition \ref{p_second_moment_optimization} to show that Laplace's method of sums is applicable.
While we have already established that $\Delta_{d^*}$ is non-negative, we still need to ensure that $p^*$ is the unique minimizer of $\Delta_d$ for $d<d^*$ and that the Hessian at $p^*$ is positive definite.
We will need the second order Taylor approximation of the KL divergence.
To be specific, let $\mu^*$ have finite non-trivial support $\mathcal S$ and let $f:\mathcal P(\mathcal S)\rightarrow\mathbb R_{\ge 0}$, $\mu\mapsto\DKL{\mu}{\mu^*}$, be the corresponding KL divergence.
Then 
\begin{flalign*}
f^{(2)}:\mathcal P(\mathcal S)\rightarrow\mathbb R_{\ge 0},
\quad
\mu\mapsto\frac{1}{2}\DP{\mu}{\mu^*}
=\frac12\sum_s\frac{(\mu(s)-\mu^*(s))^2}{\mu^*(s)}=(\mu-\mu^*)^{\mathrm{t}}D_{\mu^*}^{-1}(\mu-\mu^*)\textrm{, }
\end{flalign*}
is the second order Taylor approximation of $f$ at $\mu^*$, where $\DP{\mu}{\mu^*}$
denotes \emph{Pearson's $\chi^2$ divergence}, $D_{\mu^*}=(\delta_{i,j}\mu^*(i))_{i,j\in\mathcal S}$ the matrix with $\mu^*$ on the diagonal, and $\delta_{i,j}=1$ if $i=j$ and $0$ otherwise; this can be easily seen by considering the extension of $f$ to $\mathbb R_{\ge 0}^{\mathcal S}$.
On the other hand, we would like to consider $\Delta_d$ as a function over the suitable domain $\mathcal X$ from Section \ref{p_second_moment_empirical_overlap_distributions}, however relative to the base point $p^*$.
Hence, let $\mathcal X^*=\{x-x^*:x\in\mathcal X\}$ be the triangle $\mathcal X$ centred at $x^*$ instead of $x^{(0)}$, and $\iota^*:\mathcal X^*\rightarrow\mathcal P(\{0,1,2\})$ the bijection given by
\begin{flalign*}
\iota^*(x)=\iota(x+x^*)=p^{(0)}+(b_1,b_2)x+(b_1,b_2)x^*=\iota(x^*)+(b_1,b_2)x=p^*+(b_1,b_2)x
\end{flalign*}
for $x\in\mathcal X^*$, with $b_1,b_2$ from Equation \eqref{eq_basisdef}. Now, let $f_d:\mathcal X^*\rightarrow\mathbb R_{\ge 0}$, $x\mapsto\Delta_d(\iota^*(x))$, denote the corresponding parametrization of  $\Delta_d$.
Then, using the chain rule for multivariate calculus as indicated above for both $(b_1,b_2)$ and $W$ from Equation \eqref{eq_Wdefinition}, we derive the Hessian
\begin{flalign}\label{eq_Hessian}
H_d=(b_1,b_2)^{\mathrm{t}}\left(D_{p^*}^{-1}-\frac{(d-1)k}{d}W^{\mathrm{t}}D_{p^*_{\mathrm{e}}}^{-1}W\right)(b_1,b_2)
\end{flalign}
of $f_d$ at $0_{[2]}\in\mathbb R^2$, using the shorthand $D_{\mu^*}=(\delta_{i,j}\mu^*(i))_{i,j}$.
The properties of the KL divergence imply that $f_d(0_{[2]})=0$ and $f_d$ has a stationary point at $0_{\{0,1,2\}}$.
Now, the second order Taylor approximation $f^{(2)}_d:\mathcal X^*\rightarrow\mathbb R$, $x\mapsto\frac{1}{2}x^{\mathrm{t}}H_dx$, of $f_d$ at $0_{[2]}$ can be written as $f^{(2)}_d=\Delta^{(2)}_d\circ\iota^*$ with
\begin{flalign}\label{eq_delta2normalized}
\Delta^{(2)}_d(p)&=\frac{1}{2}\left[\DP{p}{p^*}-\frac{(d-1)k}{d}\DP{p_{\mathrm{e}}}{p^*_{\mathrm{e}}}\right]
\end{flalign}
and further for any neighbourhood $\mathcal U$ of $0_{[2]}$ such that the closure of $\mathcal U$ is contained in the interior of $\mathcal X^*$ Taylor's theorem yields a constant $c\in\mathbb R_{>0}$ such that
\begin{flalign*}
f^{(2)}_d(x)-c\|x\|_2^3\le f_d(x)\le f^{(2)}_d(x)+c\|x\|_2^3
\end{flalign*}
for all $x\in\mathcal U$.
Since $H_d$ is symmetric, let $\lambda_1$, $\lambda_2\in\mathbb R$ with $\lambda_1\le\lambda_2$ denote its eigenvalues and fix a corresponding orthonormal basis of eigenvectors $v_1$, $v_2\in\mathbb R^2$, i.e.~$H_dv_1=\lambda_1v_1$ and $H_dv_2=\lambda_2v_2$.

Formally and analogously to the KL divergence we will take the liberty to identify $\Delta_d$ and $\Delta^{(2)}_d$ with their extensions to the maximal domains $\mathcal D\subseteq\mathbb R^{\{0,1,2\}}$ and $\mathcal D^{(2)}=\mathbb R^{\{0,1,2\}}$ respectively. In particular, Lemma \ref{p_second_moment_empirical_overlap_distributions_characterization_continuous} shows that for any fully supported $p\in\mathcal P(\{0,1,2\})$ the edge distribution $p_{\mathrm{e}}$ also has full support, hence we can use the Lipschitz continuity of $W$ on $\mathbb R^{\{0,1,2\}}$ to find $\varepsilon\in\mathbb R_{>0}$ such that both $x>0$ and $Wx>0$ for any $x\in\mathcal B_\varepsilon(p)\subseteq\mathbb R_{>0}^{\{0,1,2\}}$ and thereby $\Delta_d$ is well-defined and smooth on $\mathcal B_\varepsilon(p)$.
\begin{lemma}\label{p_second_moment_laplace_conditions}
Let $k\in\mathbb Z_{\ge 4}$ and $d\in(0,d^*)$. Then the unique minimizer of $f_d$ is $0_{[2]}$ and $H_d$ is positive definite.
\end{lemma}
\begin{proof}
Using Proposition \ref{p_second_moment_optimization} we know that $H_{d^*}$ is positive semidefinite since $0_{[2]}$ is a global minimum of $f_{d^*}$. This in turn yields that $f^{(2)}_{d^*}\ge 0$ or equivalently $\Delta^{(2)}_{d^*}\ge 0$.
Now, for any $d<d^*$ the unique minimizer of $\Delta_d$ is $p^*$ since $\Delta_d(p^*)=0$, further $\Delta_d(p)>0$ for any $p\neq p^*$ with $p_{\mathrm{e}}=p^*_{\mathrm{e}}$ and
\begin{flalign*}
\Delta_d(p)=\DKL{p}{p^*}-\left(1-\frac{1}{d}\right)k\DKL{p_{\mathrm{e}}}{p^*_{\mathrm{e}}}
>\Delta_{d^*}(p)\ge 0
\end{flalign*}
for any $p$ with $p_{\mathrm{e}}\neq p^*_{\mathrm{e}}$.
But the same argumentation shows that $p^*$ is the unique minimizer of $\Delta^{(2)}_d$, since $\DP{\mu}{\mu^*}$ is also minimal with value $0$ iff $\mu=\mu^*$. This in turn shows that $f^{(2)}_d$ is uniquely minimized at $0_{[2]}$ and hence $H_d$ is positive definite.
\end{proof}
Let $\eta_{\mathrm{KL}}=\sup_{p\neq p^*}\frac{\DKL{p_{\mathrm{e}}}{p^*_{\mathrm{e}}}}{\DKL{p}{p^*}}$ denote the contraction coefficient with respect to the KL divergence.
Notice that by Proposition \ref{p_second_moment_optimization} we have $\frac{d^*}{(d^*-1)k}\ge\frac{\DKL{p_{\mathrm{e}}}{p^*_{\mathrm{e}}}}{\DKL{p}{p^*}}$ for all $p\neq p^*$ with equality for $p=p^{(2)}$, hence $\eta_{\mathrm{KL}}=\frac{d^*}{(d^*-1)k}$ (so Proposition \ref{p_second_moment_optimization} indeed confirms the conjecture by the authors in \cite{panagiotou2019}).
Further, let $\eta_{\chi^2}=\sup_{p\neq p^*}\frac{\DP{p_{\mathrm{e}}}{p^*_{\mathrm{e}}}}{\DP{p}{p^*}}$ denote the contraction coefficient with respect to the $\chi^2$ divergence.
The proof of Lemma \ref{p_second_moment_laplace_conditions} suggests that $\eta_{\chi^2}\le\eta_{\mathrm{KL}}$, a result known from literature.

In the rest of this section we discuss the straightforward (but cumbersome) application of Laplace's method for sums.
For convenience, we first show that the boundaries can be neglected and derive the asymptotics of the sum on the interior using the uniform convergence established in Lemma \ref{p_second_moment_asymptotics_uniform}.
\begin{lemma}\label{m_laplace_negligible}
Let $\mathcal U$ be a neighbourhood of $p^*$ such that its closure is contained in the interior of $\mathcal P(\{0,1,2\})$. Then 
\begin{flalign*}
\frac{\mathbb E[Z^2]}{\mathbb E[Z]^2}=\sum_pe(p)\sim\sum_{p\in\mathcal U}\sqrt{\frac{d}{(2\pi)^2m^2\prod_sp(s)}}e^{-m\Delta_d(p)}\textrm{.}
\end{flalign*}
\end{lemma}
\begin{proof} 
Let $\Delta_{\min}>0$ denote the global minimum of $\Delta$ on $\mathcal P(\{0,1,2\})\setminus\mathcal U$.
Now, we can use the well-known bounds
$\frac{1}{a+1}\exp(aH(\frac{b}{a}))\le\binom{a}{b}\le\exp(aH(\frac{b}{a}))$ for binomial coefficients and the corresponding upper bound for multinomial coefficients (using the entropy of the distribution determined by the weights $\frac{b_i}{a}$) to derive
\begin{flalign*}
\sum_{p\not\in\mathcal U}e(p)
&\le\rho(n)\sum_{p\not\in\mathcal U}e^{-m\Delta_d(p)}
\le \rho(n)e^{-m\Delta_{\min}}|\mathcal P_m(\{0,1,2\})|
\le \rho(n)m^3e^{-m\Delta_{\min}}\textrm{, where}\\
\rho(n)&=(n+1)\left(\frac{2dn}{k}+1\right)\left(\frac{(k-2)dn}{k}+1\right)\textrm{.}
\end{flalign*}
Here, we used the form of $e(p)$ introduced at the beginning of this section and further notice that the bounds used are tight for the log-densities, i.e.~the exponent is $\Delta_d(p)$ by the computations in Section \ref{p_second_moment_asymptotics_proof}.
The right hand side vanishes for $n$ tending to infinity, hence we have
\begin{flalign*}
\frac{\mathbb E[Z^2]}{\mathbb E[Z]^2}\sim\sum_{p\in\mathcal U}e(p)\textrm{.}
\end{flalign*}
Now, the result directly follows using Lemma \ref{p_second_moment_asymptotics_uniform} and Lemma \ref{p_first_moment_results}.
\end{proof}
Lemma \ref{m_laplace_negligible} shows that the overlap distributions $p$ with material contributions $e(p)$ to the second moment are concentrated around $p^*$.
Hence, instead of considering a fixed neighbourhood $\mathcal U$ of $p^*$ we consider a sequence $(\mathcal U_n)_{n\in\mathcal N}$ of decreasing neighbourhoods.
First, we choose a scaling that improves the assertion of Lemma \ref{m_laplace_negligible} and further allows to simplify the asymptotics of the right hand side, in the sense that the leading factor collapses to a constant and $f_d(x)=\Delta_d(\iota^*(x))$ 
can be replaced by its second order Taylor approximation $f^{(2)}_d(x)=\Delta^{(2)}_d(\iota^*(x))=\frac{1}{2}x^{\mathrm{t}}H_dx$ from above \eqref{eq_delta2normalized}.
For this purpose let $\mathcal U^*\subseteq\mathcal X^*$ be a sufficiently small neighbourhood of $0_{[2]}$ (in particular bounded away from the boundary of $\mathcal X^*$) and
\begin{flalign*}
\mathcal U^*_n=\left\{x\in\mathcal X^*:\|x\|_2<\frac{\ln(m)}{\sqrt{m}}\right\}
\textrm{ for }n\in\mathcal N\textrm{.}
\end{flalign*}
In the following we restrict to $n\ge n_0$ where $n_0\in\mathcal N$ is such that $\mathcal U^*_{n_0}\subseteq\mathcal U^*$.
\begin{lemma}\label{m_laplace_asymptotics_riemann}
We have
\begin{flalign*}
\frac{\mathbb E[Z^2]}{\mathbb E[Z]^2}
\sim\sqrt{\frac{d}{(2\pi)^2m^2\prod_sp^*(s)}}\sum_{x\in\mathcal U^*_n}e^{-\frac{m}{2}x^{\mathrm{t}}H_dx}\textrm{.}
\end{flalign*}
\end{lemma}
\begin{proof} 
First, notice that we can apply Lemma \ref{m_laplace_negligible} to $\iota^*(\mathcal U^*)$.
So, we need to show that the sum over $\mathcal U^*\setminus\mathcal U^*_n$ is negligible.
Then we proceed to derive the asymptotics of the sum over $\mathcal U^*_n$.
Obviously, we have $f_{\min}(n)\rightarrow 0$ for $n\rightarrow\infty$ with $f_{\min}(n)=\min_{x\not\in\mathcal U^*_n}f_d(x)>0$, since $f_d(x)=\Delta_d(\iota^*(x))$ is continuous and $f_d(0_{[2]})=0$.
The main objective of the proof is to show that $f_{\min}(n)$ converges to zero sufficiently slow.
But with $f^{(2)}_d(x)=\frac{1}{2}x^{\mathrm{t}}H_dx$ from above \eqref{eq_delta2normalized} and for any $x\in\mathbb R^2$ we have
\begin{flalign*}
f^{(2)}_d((v_1,v_2)x)&=\frac{1}{2}(\lambda_1x_1^2+\lambda_2x_2^2)
\ge\frac{\lambda_1}{2}\|x\|_2^2=\frac{\lambda_1}{2}\|(v_1,v_2)x\|_2^2
\end{flalign*}
since $(v_1,v_2)$ is an orthonormal basis, so $f^{(2)}_d(x)\ge\frac{\lambda_1}{2}\|x\|_2^2$ for all $x\in\mathbb R^2$.
Now, for any sufficiently small $\varepsilon\in(0,1)$ let $c\in\mathbb R_{>0}$ be the constant for $\mathcal B_{\varepsilon}(0_{[2]})$ from Taylor's theorem applied to $f_d$ at $0_{[2]}$,
then for any $x\in\mathcal U^*=\mathcal B_{\varepsilon}(0_{[2]})\cap\mathcal B_\delta(0_{[2]})$, with $\delta=\frac{\varepsilon\lambda_1}{2c}$, we have $f_d(x)\ge(1-\varepsilon)f^{(2)}_d(x)$ since
\begin{flalign*}
f_d(x)-(1-\varepsilon)f^{(2)}_d(x)\ge\varepsilon f^{(2)}_d(x)-c\|x\|_2^3
\ge\left(\frac{\varepsilon\lambda_1}{2}-c\delta\right)\|x\|_2^2=0\textrm{.}
\end{flalign*}
In combination we have $f_d(x)\ge\frac{(1-\varepsilon)\lambda_1}{2}\|x\|_2^2$ and hence
\begin{flalign*}
\sum_{x\not\in\mathcal U^*_n}e(\iota^*(x))&\sim\sum_{x\in\mathcal U^*\setminus\mathcal U^*_n}\sqrt{\frac{d}{(2\pi)^2m^2\prod_sp(s)}}e^{-mf_d(x)}
\le m^3e^{-\frac{(1-\varepsilon)\lambda_1}{2}\ln(m)^2}\sim 0\textrm{,}
\end{flalign*}
by using Lemma \ref{m_laplace_negligible}, bounding $x\in\mathcal U^*$ uniformly for the leading factor, the summation region with $m^3$ and an additional $m$ to compensate the constants for large $n$.
With this we have
\begin{flalign*}
\frac{\mathbb E[Z^2]}{\mathbb E[Z]^2}
\sim\sum_{x\in\mathcal U^*_n}e(\iota^*(x))
\sim\sqrt{\frac{d}{(2\pi)^2m^2\prod_sp^*(s)}}\sum_{x\in\mathcal U^*_n}e^{-mf^{(2)}_d(x)}\textrm{,}
\end{flalign*}
where the last inequality follows from the fact that the leading factor converges to the respective constant uniformly on $\mathcal U^*_n$ and by Taylor's theorem for $f_d$ on $\mathcal U^*_n$ as seen before.
\end{proof}
Lemma \ref{m_laplace_asymptotics_riemann} completes the analytical part of the proof.
For the last, measure theoretic, part we recall the bijection $\iota_n$ from Lemma \ref{p_second_moment_empirical_overlap_distributions_characterization}.
The translation of the sum on the right-hand side of Lemma \ref{m_laplace_asymptotics_riemann} into a Riemann sum and further into the integral $\int g_\infty(x)\mathrm{d}x$, where
\begin{flalign*}
g_{\infty}:\mathbb R^2\rightarrow\mathbb R_{>0},
\quad
y \mapsto \sqrt{\frac{d}{(2\pi)^2\prod_sp^*(s)}}\exp\left(-\frac{1}{2}y^{\mathrm{t}}H_dy\right)\textrm{, }
\end{flalign*}
is essentially given by the grid $\mathcal X_n\subseteq(m^{-1}\mathbb Z)^2\subseteq\mathbb R^2$. We make this rigorous in the following.
\begin{lemma}\label{m_laplace_sum_limit}
We have
\begin{flalign*}
\frac{\mathbb E[Z^2]}{\mathbb E[Z]^2}
\sim\int g_\infty(x)\mathrm{d}x\textrm{.}
\end{flalign*}
\end{lemma}
\begin{proof}
We start with the partition of $\mathbb R^2$ into the squares
\begin{flalign*}
\mathcal Q_{n,x}=\left\{x+\alpha_1\begin{pmatrix}1\\ 0\end{pmatrix}+\alpha_2\begin{pmatrix}0\\ 1\end{pmatrix}:\alpha\in\left[-\frac{1}{2m},\frac{1}{2m}\right)^2\right\}\textrm{, }x\in(m^{-1}\mathbb Z)^2\textrm{.}
\end{flalign*}
Next, we need a suitable selection of squares to cover the disc
\begin{flalign*}
x^*+\mathcal U^*_n=\left\{x^*+x:x\in\mathbb R^2,\|x\|_2<\frac{\ln(m)}{\sqrt{m}}\right\}\subseteq\mathbb R^2
\end{flalign*}
corresponding to the disc $\mathcal U^*_n$.
For this purpose let $x_{\min}$, $x_{\max}\in(m^{-1}\mathbb Z)^2$ be given by
\begin{flalign*}
x_{\min,1}&=m^{-1}\left\lfloor m\left(x^*_1-\frac{\ln(m)}{\sqrt{m}}\right)\right\rfloor\textrm{, }
x_{\min,2}=m^{-1}\left\lfloor m\left(x^*_2-\frac{\ln(m)}{\sqrt{m}}\right)\right\rfloor\textrm{, }\\
x_{\max,1}&=m^{-1}\left\lceil m\left(x^*_1+\frac{\ln(m)}{\sqrt{m}}\right)\right\rceil\textrm{, }
x_{\max,2}=m^{-1}\left\lceil m\left(x^*_2+\frac{\ln(m)}{\sqrt{m}}\right)\right\rceil\textrm{.}
\end{flalign*}
Further, let $\mathcal G_n=(m^{-1}\mathbb Z)^2\cap\left([x_{\min,1},x_{\max,1}]\times[x_{\min,2},x_{\max,2}]\right)$.
By the definition of $x_{\min}$ and $x_{\max}$ the points on the boundary are not in $x^*+\mathcal U^*_n$, which ensures that $x^*+\mathcal U^*_n\subseteq\mathcal Q_n$ with $\mathcal Q_n=\bigcup_{x\in\mathcal G_n}\mathcal Q_{n,x}$. Further, we have $\mathcal Q_-\subseteq\mathcal Q_n\subseteq\mathcal Q_+$ with
\begin{flalign*}
\mathcal Q_-=\left\{x\in\mathbb R^2:\|x-x^*\|_{\infty}\le\frac{\ln(m)}{\sqrt{m}}\right\}
\textrm{, }
\mathcal Q_+=\left\{x\in\mathbb R^2:\|x-x^*\|_{\infty}\le\frac{\ln(m)}{\sqrt{m}}+\frac{3}{2m}\right\}\textrm{,}
\end{flalign*}
which ensures that $\mathcal Q_n\subseteq\mathcal X$ for $n\in\mathcal N$ sufficiently large.
Now, we translate the notions back to $\mathcal X^*$ using the bijection $\tau:\mathbb R^2\rightarrow\mathbb R^2$, $x\mapsto x-x^*$, i.e.~let
$\mathcal G^*_n=\tau(\mathcal G_n)$,
$\mathcal Q^*_{n,x}=\tau(\mathcal Q_{n,\tau^{-1}(x)})$ for $x\in\mathcal G^*_n$, $\mathcal Q^*_n=\tau(\mathcal Q_n)$, $\mathcal Q^*_-=\tau(\mathcal Q_-)$ and $\mathcal Q^*_+=\tau(\mathcal Q_+)$.
This directly gives
\begin{flalign*}
\mathcal Q^*_{n,x}&=\left\{x+\alpha_1\begin{pmatrix}1\\ 0\end{pmatrix}+\alpha_2\begin{pmatrix}0\\ 1\end{pmatrix}:\alpha\in\left[-\frac{1}{2m},\frac{1}{2m}\right)^2\right\}\textrm{, }x\in\mathcal G^*_n\textrm{, }
\mathcal Q^*_n=\bigcup_{x\in\mathcal G^*_n}\mathcal Q^*_{n,x}\textrm{,}\\
\mathcal Q^*_-&=\left\{x\in\mathbb R^2:\|x\|_\infty\le\frac{\ln(m)}{\sqrt{m}}\right\}\textrm{, }
\mathcal Q^*_+=\left\{x\in\mathbb R^2:\|x\|_\infty\le\frac{\ln(m)}{\sqrt{m}}+\frac{3}{2m}\right\}\textrm{, }
\end{flalign*}
and $\mathcal U^*_n\subseteq\mathcal Q^*_-\subseteq\mathcal Q^*_n\subseteq\mathcal Q^*_+\subseteq\mathcal X^*$ for $n\in\mathcal N$ sufficiently large.
Further, with Lemma \ref{m_laplace_asymptotics_riemann} and the definition of $f_d^{(2)}$ we now have
\begin{flalign*}
\frac{\mathbb E[Z^2]}{\mathbb E[Z]^2}\sim\sum_{x\in\mathcal G^*_n}\sqrt{\frac{d}{(2\pi)^2m^2\prod_sp^*(s)}}\exp\left(-\frac{1}{2}mx^{\mathrm{t}}H_dx\right)
\end{flalign*}
Finally, we need to adjust the scaling to turn the sum on the right hand side into a Riemann sum.
For this purpose let $\sigma:\mathbb R^2\rightarrow\mathbb R^2$, $x\mapsto\sqrt{m}x$, further
$\mathcal G'_n=\sigma(\mathcal G^*_n)$, $\mathcal Q'_{n,x}=\sigma(\mathcal Q^*_{n,\sigma^{-1}(x)})$ for $x\in\mathcal G'_n$, 
$\mathcal Q'_n=\sigma(\mathcal Q^*_n)$, $\mathcal Q'_-=\sigma(\mathcal Q^*_-)$ and $\mathcal Q'_+=\sigma(\mathcal Q^*_+)$.
This directly gives
\begin{flalign*}
\mathcal Q'_{n,x}&=\left\{x+\alpha_1\begin{pmatrix}1\\ 0\end{pmatrix}+\alpha_2\begin{pmatrix}0\\ 1\end{pmatrix}:\alpha\in\left[-\frac{1}{2\sqrt{m}},\frac{1}{2\sqrt{m}}\right)^2\right\}\textrm{, }x\in\mathcal G'_n\textrm{, }
\mathcal Q'_n=\bigcup_{x\in\mathcal G'_n}\mathcal Q'_{n,x}\textrm{,}\\
\mathcal Q'_-&=\left\{x\in\mathbb R^2:\|x\|_\infty\le\ln(m)\right\}\textrm{, }
\mathcal Q'_+=\left\{x\in\mathbb R^2:\|x\|_\infty\le\ln(m)+\frac{3}{2\sqrt{m}}\right\}\textrm{, }
\end{flalign*}
and $\mathcal Q'_-\subseteq\mathcal Q'_n\subseteq\mathcal Q'_+$.
Using that $mx^{\mathrm{t}}H_dx=\sigma(x)^{\mathrm{t}}H_d\sigma(x)$ for all $x\in\mathcal G^*_n$ and further that the area of $\mathcal Q'_{n,x}$ is $m^{-1}$ for all $x\in\mathcal G'_n$ we have
\begin{flalign*}
\frac{\mathbb E[Z^2]}{\mathbb E[Z]^2}&\sim\sum_{x\in\mathcal G'_n}\sqrt{\frac{d}{(2\pi)^2m^2\prod_sp^*(s)}}\exp\left(-\frac{1}{2}x^{\mathrm{t}}H_dx\right)=\int g_n(y)\mathrm{d}y\textrm{,}\\
g_n(y)&=\sum_{x\in\mathcal G'_n}\mathbbm 1\{y\in\mathcal Q'_{n,x}\}\sqrt{\frac{d}{(2\pi)^2\prod_sp^*(s)}}\exp\left(-\frac{1}{2}x^{\mathrm{t}}H_dx\right)\textrm{, }y\in\mathbb R^2\textrm{.}
\end{flalign*}
In order to show that $\int g_n(y)\mathrm{d}y$ converges to $\int g_{\infty}(y)\mathrm{d}y$
we recall from Lemma \ref{p_second_moment_laplace_conditions} that $H_d$ is positive definite, which ensures that $\int g_{\infty}(y)\mathrm{d}y$ exists and is finite.
Now, using Taylor's theorem with order $0$ and the Lagrange form of the first order remainder with the fact that the absolutes of the first derivatives of $g_{\infty}$ are bounded from above yields a constant $c\in\mathbb R_{>0}$ such that for all $n\in\mathcal N$ and all $y\in\mathcal Q'_n$, with $x\in\mathcal G'_n$ such that $y\in\mathcal Q'_{n,x}$, we have
\begin{flalign*}
\|g_{\infty}(y)-g_n(y)\|_{\infty}=\|g_{\infty}(y)-g_{\infty}(x)\|_{\infty}\le\frac{c}{2\sqrt{m}}\textrm{.}
\end{flalign*}
This bound directly suggests that
\begin{flalign*}
&\left|\int \mathbbm 1\{y\in\mathcal Q'_{n,x}\}g_{\infty}(y)\mathrm{d}y-\int \mathbbm 1\{y\in\mathcal Q'_{n,x}\}g_{n}(y)\mathrm{d}y\right|\le cm^{-\frac{3}{2}}\textrm{, }\\
&\left|\int \mathbbm 1\{y\in\mathcal Q'_{n}\}g_{\infty}(y)\mathrm{d}y-\int \mathbbm 1\{y\in\mathcal Q'_{n}\}g_{n}(y)\mathrm{d}y\right|\le c\left(\frac{\ln(m)}{\sqrt{m}}+\frac{3}{2m}\right)\textrm{ and }\\
&\left|\int g_{\infty}(y)\mathrm{d}y-\int g_{n}(y)\mathrm{d}y\right|\le c\left(\frac{\ln(m)}{\sqrt{m}}+\frac{3}{2m}\right)+\int \mathbbm 1\{y\not\in\mathcal Q'_{n}\}g_{\infty}(y)\mathrm{d}y\textrm{.}
\end{flalign*}
In particular the last bound suggests that $\int g_n(y)\mathrm{d}y\rightarrow\int g_{\infty}(y)\mathrm{d}y$ since the error on the right-hand side tends to zero as $n$ tends to infinity.

\end{proof}
The only remaining part of the proof is to compute $\int g_\infty(x)\mathrm{d}x$.
\begin{lemma}\label{m_laplace_integral}
We have
$\int g_\infty(x)\mathrm{d}x=\sqrt{\frac{k-1}{k-d}}$.
\end{lemma}
\begin{proof}
The Gaussian integral gives
\begin{flalign*}
\int g_\infty(x)\mathrm{d}x
=\sqrt{\frac{d}{(2\pi)^2\prod_sp^*(s)}}\sqrt{\frac{(2\pi)^2}{\det(H_d)}}
=\sqrt{\frac{d}{\det(H_d)\prod_sp^*(s)}}\textrm{.}
\end{flalign*}
Recall that the Hessian $H_d\in\mathbb R^{2\times 2}$ from \eqref{eq_Hessian} is a $2\times 2$ matrix and all entries are given explicitly, so a straightforward calculation asserts that $\det(H_d)=(k-d)d/[(k-1)\prod_sp^*(s)]$.
\end{proof}
Finally, combining Lemma \ref{m_laplace_sum_limit} with Lemma \ref{m_laplace_integral} completes the proof of Proposition \ref{p_second_moment_results}.
%
%
%
%
\subsection{Proof of Proposition \ref{p_second_moment_optimization}}\label{s_second_moment_optimization}

We start with a characterization of the stationary points of $\Delta_d$ for any $d\in\mathbb R_{>0}$.
In order to determine these, we first determine the stationary points of the restriction of $\Delta_d$ to overlap distributions with the same fixed edge distribution.
For this purpose, recall the line $\mathcal W\subseteq\mathcal P(\{0,1\}^2)$ of attainable edge distributions and
the lines $\mathcal P_q=\iota(\mathcal X_q)=\{p\in\mathcal P(\{0,1,2\}):p_{\mathrm{e}}=q\}$ of overlap distributions with fixed edge distribution $q\in\mathcal W$ from Lemma \ref{p_second_moment_empirical_overlap_distributions_characterization_continuous}.
Further, let $\Delta_{d,q}:\mathcal P_q\rightarrow\mathbb R$ denote the restriction of $\Delta_d$ to $\mathcal P_q$.
For $x\in\mathbb R_{>0}$ let $p_x\in\mathcal P(\{0,1,2\})$ be given by $p_x(s)\propto p^*(s)x^s$, $s\in\{0,1,2\}$, further let
$p_{0}=p^{(0)}$, $p_{\infty}=p^{(2)}$, and $\mathcal P_{\min}=\{p_x:x\in[0,\infty]\}$.
Finally, let $\iota_{\mathrm{rp}}:[0,\infty]\rightarrow\mathcal P_{\min}$, $x\mapsto p_x$, denote the induced map and $\iota_{\mathrm{pe}}:\mathcal P_{\min}\rightarrow\mathcal W$, $p\mapsto p_{\mathrm{e}}$, the corresponding edge distributions.
\begin{lemma}\label{m_optimization_minimizers_local}
For all $q\in\mathcal W\setminus\{p^{(0)}_{\mathrm{e}},p^{(2)}_{\mathrm{e}}\}$ the map $\Delta_{d,q}$ has a unique stationary point $p_q\in\mathcal P_q$ that is a global minimum. The unique global minimizer of $\Delta_{d,p^{(s)}_{\mathrm{e}}}$ is $p_{p^{(s)}_{\mathrm{e}}}=p^{(s)}$ for $s\in\{0,2\}$.
Further, we have $\mathcal P_{\min}=\{p_q:q\in\mathcal W\}$ and the maps $\iota_{\mathrm{rp}}$, $\iota_{\mathrm{pe}}$ are bijections.
\end{lemma}
\begin{proof}
Recall from Lemma \ref{p_second_moment_empirical_overlap_distributions_characterization_continuous} that $\mathcal P_q$ is one-dimensional for $q\in\mathcal W\setminus\{p^{(0)}_{\mathrm{e}},p^{(2)}_{\mathrm{e}}\}$.
Further, the map $\Delta_{d,q}$ is strictly convex since the KL divergence $\DKL{p}{p^*}$ (respectively $x\ln(x)$) is and further $\DKL{p_{\mathrm{e}}}{p^*_{\mathrm{e}}}=\DKL{q}{p^*_{\mathrm{e}}}$ is constant.
Now, fix an interior point $p_\circ\in\mathcal P_q$ and let a boundary point $p_{\mathrm{b}}\in\mathcal P_q$ be given. Then $p_{\mathrm{b}}$ is not fully supported since it is on the boundary of $\mathcal P(\{0,1,2\}$ and hence the derivative of $(\DKL{\alpha p_{\circ}+(1-\alpha)p_{\mathrm{b}}}{p^*})'$ tends to $-\infty$ as $\alpha$ tends to $0$, which shows that $\Delta_{d,q}$ is not minimized on the boundary. Hence, we know that there exists exactly one stationary point $p_q\in\mathcal P_q$ and that $\Delta_{d,q}(p)$ is minimial iff $p=p_q$.
As discussed in Lemma \ref{p_second_moment_empirical_overlap_distributions_characterization_continuous} we have $\mathcal P_q=\{p^{(s)}\}$ for $q=p^{(s)}_{\mathrm{e}}$ and $s\in\{0,2\}$, so $p_q=p^{(s)}$ is obviously the unique global minimizer of $\Delta_{d,q}$ in this case and further $\Delta_{d,q}$ has no stationary points (since $\mathcal P_q$ has empty interior). This shows that the map $q\mapsto p_q$ for $q\in\mathcal W$ is a bijection.

Further, for $q$ in the interior of $\mathcal W$ the stationary point $p_q$ is fully supported and the unique root of the first derivative of $\Delta_{d,q}$ in the direction $b_2$, i.e.
\begin{flalign*}
\ln\left(\frac{p_q(0)}{p^*(0)}\right)+\ln\left(\frac{p_q(2)}{p^*(2)}\right)=2\ln\left(\frac{p_q(1)}{p^*(1)}\right)\textrm{ or equivalently }
\frac{p_q(2)/p^*(2)}{p_q(1)/p^*(1)}=\frac{p_q(1)/p^*(1)}{p_q(0)/p^*(0)}\textrm{.}
\end{flalign*}
Let $\mathcal P'_{\min}$ denote the set of all fully supported $p\in\mathcal P(\{0,1,2\})$ satisfying
$\frac{p(2)/p^*(2)}{p(1)/p^*(1)}=\frac{p(1)/p^*(1)}{p(0)/p^*(0)}$, i.e.~our set of candidates for stationary points.
Now, for $p\in\mathcal P'_{\min}$ let $q=p_{\mathrm{e}}$, then we obviously have $p\in\mathcal P_q$ and $p$ is a root of the first derivative of $\Delta_{d,q}$ in the direction $b_2$, so $p$ is the unique root and $p=p_q$. Hence, the map $\iota_{\mathrm{pe}}':\mathcal P'_{\min}\rightarrow\mathcal W$, $p\mapsto p_{\mathrm{e}}$, is a bijection (up to the corners of $\mathcal W$) with inverse $q\mapsto p_q$.
Now, let $\iota_{\mathrm{pr}}:\mathcal P'_{\min}\rightarrow\mathbb R_{>0}$, $p\mapsto x_p$, with $x_p=\frac{p(1)p^*(0)}{p^*(1)p(0)}$.
Notice that $\iota_{\mathrm{pr}}$ is surjective since for any $x\in\mathbb R_{>0}$ we have
\begin{flalign*}
\frac{p_x(2)/p^*(2)}{p_x(1)/p^*(1)}=\frac{x^2}{x}=x=\frac{p_x(1)/p^*(1)}{p_x(0)/p^*(0)}
\end{flalign*}
and hence $p_x\in\mathcal P'_{\min}$. To show that $\iota_{\mathrm{pr}}$ is injective let $p\in\mathcal P'_{\min}$ and $x=x_p$. Using the definition of $x_p$ and the defining property of $\mathcal P'_{\min}$ we get
\begin{flalign*}
p(0)&=p^*(0)\frac{p(0)}{p^*(0)}\textrm{, }
p(1)=p^*(1)x\frac{p(0)}{p^*(0)}\textrm{, }
p(2)=p^*(2)x\frac{p(1)}{p^*(1)}=p^*(2)x^2\frac{p(0)}{p^*(0)}\textrm{, so}\\
p(s)&=\frac{p(s)}{p(0)+p(1)+p(2)}=\frac{p^*(s)x^s}{\sum_sp^*(s)x^s}=p_x(s)\textrm{, }s\in\{0,1,2\}\textrm{.}
\end{flalign*}
This shows that $\mathcal P_{\min}=\mathcal P'_{\min}\cup\{p_0,p_\infty\}$, that $\iota_{\mathrm{rp}}$ is a bijection with inverse $\iota_{\mathrm{pr}}$ (canonically extended to the endpoints), and finally that $\iota_{\mathrm{pe}}=\iota'_{\mathrm{pe}}$ is a bijection as well.
\end{proof}
Lemma \ref{m_optimization_minimizers_local} has a few immediate consequences.
For one, the only minimizers of $\Delta_d$ in the direction of $b_2$ on the boundary are $p^{(0)}$ and $p^{(2)}$, while all other boundary points are maximizers in the direction of $b_2$, hence if $p$ is a global minimizer of $\Delta_d$ on the boundary, then we have $p\in\{p^{(0)},p^{(2)}\}$.
Further, all stationary points of $\Delta_d$ are either local minima or saddle points.
Finally, we have $p\in\mathcal P_{\min}$ for any stationary point $p\in\mathcal P(\{0,1,2\})$ of $\Delta_d$ since then also the derivative in the direction of $b_2$ vanishes.

For the upcoming characterization of the stationary points of $\Delta_d$ let
\begin{flalign*}
\iota_{\mathrm{rr}}:\mathbb R_{>0}\rightarrow\mathbb R_{>0},
\quad
\iota_{\mathrm{rr}}(x)=\iota^*_{\mathrm{rr}}(x)^{\frac{d-1}{d}}\textrm{, }
\iota^*_{\mathrm{rr}}(x)=\frac{p_{x,\mathrm{e}}(1,1)p_{x,\mathrm{e}}(0,0)}{p_{x,\mathrm{e}}(1,0)p_{x,\mathrm{e}}(0,1)} \textrm{.}
\end{flalign*}
Notice that $\iota_{\mathrm{rr}}(x)\in\mathbb R_{>0}$ for $x\in\mathbb R_{>0}$ since then $p_x$ is fully supported and hence $p_{x,\mathrm{e}}$ is fully supported by Lemma \ref{p_second_moment_empirical_overlap_distributions_characterization_continuous}.
Finally, let $\mathcal X_{\mathrm{st}}=\{x\in\mathbb R_{>0}:\iota_{\mathrm{rr}}(x)=x\}$ denote the fixed points of $\iota_{\mathrm{rr}}$ and $\mathcal P_{\mathrm{st}}=\{p_x:x\in\mathcal X_{\mathrm{st}}\}$ the corresponding distributions.
Notice that $p_x=p^*$ for $x=1$ and further $\iota^*(1)=1$, i.e.~$\iota_{\mathrm{rr}}(1)=1$ for all $d\in\mathbb R_{>0}$, hence $1\in\mathcal X_{\mathrm{st}}$ and $p^*\in\mathcal P_{\mathrm{st}}$ for all $d\in\mathbb R_{>0}$.
\begin{lemma}\label{m_optimization_stationary}
The stationary points of $\Delta_d$ are given by $\mathcal P_{\mathrm{st}}$.
\end{lemma}
\begin{proof}
Using Lemma \ref{m_optimization_minimizers_local}, a fully supported distribution $p\in\mathcal P(\{0,1,2\})$ is a stationary point of $\Delta_d$ iff there exists $x\in\mathbb R_{>0}$ such that $p=p_x$ and the derivative of $\Delta_d$ at $p_x$ in the direction $b_1$ vanishes, i.e.~$p_x$ is a solution of
\begin{flalign*}
0=\left(\left(\ln\left(\frac{p_x(s)}{p^*(s)}\right)\right)_{s\in\{0,1,2\}}^{\mathrm{t}}-\frac{(d-1)k}{d}\left(\ln\left(\frac{p_{x,\mathrm{e}}(y)}{p^*_{\mathrm{e}}(y)}\right)\right)_{y\in\{0,1\}^2}^{\mathrm{t}}W\right)b_1\textrm{,}
\end{flalign*}
where we used the chain rule for multivariate calculus, that $W$ is column stochastic and that $b_1\in 1_{\{0,1,2\}}^{\perp}$.
Recall from Section \ref{p_second_moment_empirical_overlap_distributions}, e.g.~from the proof of Lemma \ref{p_second_moment_empirical_overlap_distributions_characterization}, that $Wb_1=\frac{d}{k}w$, hence computing the dot product with $b_1$ gives
\begin{flalign*}
0=d\ln(x)-(d-1)\ln\left(\frac{p_{x,\mathrm{e}}(1,1)p_{x,\mathrm{e}}(0,0)}{p_{x,\mathrm{e}}(1,0)p_{x,\mathrm{e}}(0,1)}\right)\textrm{.}
\end{flalign*}
Obviously, equality holds if and only if $x\in\mathcal X_{\mathrm{st}}$, hence $p$ is a stationary point of $\Delta_d$ iff $p\in\mathcal P_{\mathrm{st}}$.
\end{proof}
Lemma \ref{m_optimization_stationary} does not only allow to translate the stationary points of $\Delta_d$ into fixed points of $\iota_{\mathrm{rr}}$, it also allows to translate the types as follows.
\begin{lemma}\label{m_optimization_stationary_type}
Fix $x\in\mathbb R_{>0}$.
Then we have $\iota_{\mathrm{rr}}(x)<x$ if and only if $(\Delta_d\circ\iota_{\mathrm{rp}})'(x)>0$,
$\iota_{\mathrm{rr}}(x)>x$ if and only if $(\Delta_d\circ\iota_{\mathrm{rp}})'(x)<0$, and
$\iota_{\mathrm{rr}}(x)=x$ if and only if $(\Delta_d\circ\iota_{\mathrm{rp}})'(x)=0$.
\end{lemma}
\begin{proof}
Fix $x\in\mathbb R_{>0}$. The proof of Lemma \ref{m_optimization_stationary} directly suggests that the first derivative of $\Delta_d$ at $p_x$ in the direction $b_1$ is strictly positive if and only if
\begin{flalign*}
0<\ln(x)-\frac{d-1}{d}\ln(\iota^*_{\mathrm{rr}}(x))\textrm{,}
\end{flalign*}
which holds if and only if $\iota_{\mathrm{rr}}(x)<x$.
We're left to establish that the direction of $\iota_{\mathrm{rp}}$ is consistent with $b_1$.
Intuitively, using Lemma \ref{p_second_moment_empirical_overlap_distributions} and Lemma \ref{m_optimization_minimizers_local} we can argue that $x\mapsto p_{x,\mathrm{e}}(1,1)$ is a bijection and hence either increasing or decreasing. Taking the limits $x\rightarrow 0$ and $x\rightarrow\infty$ suggests that it is increasing, hence with $c\in\mathbb R^2$ denoting the coordinates of $\iota'_{\mathrm{rp}}$, i.e.~$\iota'_{\mathrm{rp}}(x)=(b_1,b_2)c$, we know that $c_1\ge 0$.

Formally, we quantify the direction of $\iota_{\mathrm{rp}}$. For this purpose we compute the derivative of $\iota_{\mathrm{rp}}$ at $x\in\mathbb R_{>0}$, given by
\begin{flalign*}
\iota'_{\mathrm{rp}}(x)=\left(\frac{sp^*(s)x^{s-1}\sum_{s'\in\{0,1,2\}}p^*(s')x^{s'}-p^*(s)x^{s}\sum_{s'\in\{0,1,2\}}s'p^*(s')x^{s'-1}}{\left(\sum_{s'\in\{0,1,2\}}p^*(s')x^{s'}\right)^2}\right)_{s\in\{0,1,2\}}\textrm{.}
\end{flalign*}
Notice that $v=\iota'_{\mathrm{rp}}(x)\in 1_{\{0,1,2\}}^{\perp}$, since $\iota_{\mathrm{rp}}(\mathbb R_{>0})\subseteq\mathcal P(\{0,1,2\})$ or by computing $\sum_sv_s=0$ directly.
Now, let $c\in\mathbb R^2$ be the coordinates of $v$, i.e.~$v=(b_1,b_2)c$.
This directly gives $c_2=v_2$ and hence $c_1=d^{-1}(v_1+2v_2)=d^{-1}\sum_ssv_s$.
Now, notice that
\begin{flalign*}
S&=dxc_1
=\sum_{s,s'\in\{0,1,2\}}p_x(s)p_x(s')s(s-s')\\
&=\sum_{s>s'}p_x(s)p_x(s')s(s-s')-\sum_{s>s'}p_x(s)p_x(s')s'(s-s')
=\sum_{s>s'}p_x(s)p_x(s')(s-s')^2>0\textrm{,}
\end{flalign*}
which directly gives $c_1=\frac{S}{dx}\in\mathbb R_{>0}$.
Now, with $\nabla=\left(\frac{\partial \Delta_d}{\partial p(s)}(p_x)\right)_{s\in\{0,1,2\}}\in\mathbb R^{\{0,1,2\}}$ denoting the partial derivatives of $\Delta_d$ at $p_x$ and using the chain rule we have
\begin{flalign*}
(\Delta_d\circ\iota_{\mathrm{rp}})'(x)
&=\nabla^{\mathrm{t}}\iota'_{\mathrm{rp}}(x)
=c_1\nabla^{\mathrm{t}}b_1
+c_2\nabla^{\mathrm{t}}b_2
=c_1\nabla^{\mathrm{t}}b_1\textrm{,}
\end{flalign*}
since the derivative $\nabla^{\mathrm{t}}b_2$ of $\Delta_d$ at $p_x$ in the direction $b_2$ is zero,
hence we have $(\Delta_d\circ\iota_{\mathrm{rp}})'(x)>0$ if and only if the derivative $\nabla^{\mathrm{t}}b_1$ of $\Delta_d$ at $p_x$ in the direction $b_1$ is strictly positive, which is the case if and only if $\iota_{\mathrm{rr}}(x)<x$.
\end{proof}
Lemma \ref{m_optimization_stationary_type} with Lemma \ref{m_optimization_stationary} shows that control over $\iota_{\mathrm{rr}}$ gives complete control over the location and characterization of the stationary points of $\Delta_d$.
However, instead of solving the fixed point equation given by $\iota_{\mathrm{rr}}$ directly, we use a slight modification inspired by the \emph{belief propagation} algorithm applied to the constraint satisfaction discussed in Section \ref{p_second_moment_CSP} and initialized with uniform messages.

For this purpose let $N\in\mathbb Z_{\ge 0}$, further $N_1$, $N_2\in[N]_0$ and the hypergeometric distribution $p_{N,N_1,N_2}\in\mathcal P(\mathbb Z)$ be given by
\begin{flalign*}
p_{N,N_1,N_2}(s)=\frac{\binom{N_1}{s}\binom{N-N_1}{N_2-s}}{\binom{N}{N_2}}=\frac{\binom{N}{N-N_1-N_2+s,N_1-s,N_2-s,s}}{\binom{N}{N_1}\binom{N}{N_2}}\textrm{ for }s\in\mathbb Z\textrm{.}
\end{flalign*}
The latter form directly shows that $p_{N,N_1,N_2}=p_{N,N_2,N_1}$ and further highlights the intuition behind $p^*=p_{k,2,2}$ in the context of the squared constraint satisfaction problem, i.e.~$p^*(s)$ gives the probability of seeing a certain overlap $s\in\{0,1,2\}$ when drawing two satisfying constraint assignments for the standard problem uniformly and independently.
Now, for $y\in\{0,1\}^2$ let $p^*_y\in\mathcal P(\{0,1,2\})$ be given by $p^*_{(1,1)}(s)=p_{k-1,1,1}(s-1)$, $p^*_{(1,0)}(s)=p_{k-1,1,2}(s)$, $p^*_{(0,1)}(s)=p_{k-1,2,1}(s)$, $p^*_{(0,0)}(s)=p_{k-1,2,2}(s)$ for $s\in\{0,1,2\}$.
Hence, $p^*_y(s)$ gives the probability of seeing a certain overlap $s\in\{0,1,2\}$ when drawing two satisfying constraint assignments for the standard problem uniformly and independently, but knowing the pair $y$ of values of one (fixed or random) coordinate.
In particular, this explains why $p^*_{(1,1)}(0)=0$.
Further, notice that $p^*_{(1,0)}=p^*_{(0,1)}$ and finally that the matrix $W'$ in the proof of Lemma \ref{m_laplace_integral} is given by $W'=(p^*_y)_{y\in\{0,1\}^2}$.

On the other hand, for $y\in\{0,1\}^2$ let $p'_y\in\mathcal P(\{0,1,2\})$ be given by $p'_{(1,1)}(s)=p_{k-1,1,1}(s)$ for $s\in\{0,1,2\}$ and further $p'_y=p^*_y$ for $y\in\{0,1\}^2\setminus\{(1,1)\}$.
Hence, knowing the value $y$ of one (fixed or random) coordinate, $p'_y(s)$ gives the probability of seeing a certain overlap \emph{on the remaining coordinates}.
This explains both why $p'_{(1,1)}(s)=p^*_{(1,1)}(s-1)$ and $p'_y(s)=p^*_y(s)$ for $y\neq(1,1)$.
Finally, for a distribution $p\in\mathcal P(\{0,1,2\})$ let $f_p:\mathbb R_{\ge 0}\rightarrow\mathbb R_{\ge 0}$, $x\mapsto\sum_{s\in\{0,1,2\}}p(s)x^s$, be its (factorial) moment generating function, and further $\iota_{\mathrm{BP}}:\mathbb R_{>0}\rightarrow\mathbb R_{>0}$ given by
\begin{flalign*}
\iota_{\mathrm{BP}}(x)=\iota^*_{\mathrm{BP}}(x)^{d-1}\textrm{, }
\iota^*_{\mathrm{BP}}(x)=\frac{f_{p'_{(1,1)}}(x)f_{p'_{(0,0)}}(x)}{f_{p'_{(1,0)}}(x)f_{p'_{(0,1)}}(x)}\textrm{, for }x\in\mathbb R_{>0}\textrm{.}
\end{flalign*}
\begin{lemma}\label{m_optimization_fixed_point_equations}
Fix $x\in\mathbb R_{>0}$.
Then we have $\iota_{\mathrm{rr}}(x)<x$ if and only if $\iota_{\mathrm{BP}}(x)<x$,
$\iota_{\mathrm{rr}}(x)>x$ if and only if $\iota_{\mathrm{BP}}(x)>x$, and
$\iota_{\mathrm{rr}}(x)=x$ if and only if $\iota_{\mathrm{BP}}(x)=x$.
\end{lemma}
\begin{proof}
First, notice that the normalization constant of $p_x$ cancels out in $\iota^*_{\mathrm{rr}}$, as does the normalization constant $\binom{k}{2}^2$ of $p^*$. Further, with $v=(k-4+s,2-s,2-s,s)^{\mathrm{t}}\in\mathbb R^{\{0,1\}^2}$ we have $W_{y,s}\binom{k}{v}=\frac{v_y}{k}\binom{k}{v}=\binom{k-1}{v-(\delta_{y,z})_z}$ for $y\in\{0,1\}^2$, $s\in\{0,1,2\}$, and thereby
\begin{flalign*}
\iota^*_{\mathrm{rr}}(x)=\frac{\left(\sum_s\binom{k-1}{k-4+s,2-s,2-s,s-1}x^s\right)\left(\sum_s\binom{k-1}{k-4+s-1,2-s,2-s,s}x^s\right)}{\left(\sum_s\binom{k-1}{k-4+s,2-s,2-s-1,s}x^s\right)\left(\sum_s\binom{k-1}{k-4+s,2-s-1,2-s,s}x^s\right)}
\textrm{ for }x\in\mathbb R_{>0}\textrm{.}
\end{flalign*}
Now, since the normalization constants cancel out in total, this directly gives
\begin{flalign*}
\iota^*_{\mathrm{rr}}(x)=\frac{f_{p^*_{(1,1)}}(x)f_{p^*_{(0,0)}}(x)}{f_{p^*_{(1,0)}}(x)f_{p^*_{(0,1)}}(x)}
=x\iota^*_{\mathrm{BP}}(x)
\textrm{ for }x\in\mathbb R_{>0}\textrm{,}
\end{flalign*}
using that $p^*_{(1,1)}(s)=p'_{(1,1)}(s-1)$ for $s\in\{0,1,2\}$, hence $f_{p^*_{(1,1)}}(x)=xf_{p'_{(1,1)}}(x)$, and $p^*_y(s)=p'_y(s)$ for $y\neq(1,1)$.
Now, we have $x=\iota_{\mathrm{rr}}(x)$ if and only if $x=x^{\frac{d-1}{d}}\iota_{\mathrm{BP}}(x)^{\frac{1}{d}}$, which holds if and only if $x^{\frac{1}{d}}=\iota_{\mathrm{BP}}(x)^{\frac{1}{d}}$, which then again is equivalent to
$x=\iota_{\mathrm{BP}}(x)$. Equivalence of the inequalities follows analogously.
\end{proof}
The following part is dedicated to the identification of the fixed points of $\iota_{\mathrm{BP}}$, and the only part where we actually require $r=2$ with the occupation number $r$ as defined in Section \ref{s_occupation_problems}.
We start with a discussion of $\iota^*_{\mathrm{BP}}$.
For this purpose let $g^*_1$, $g^*_k:\mathbb R_{\ge 1}\rightarrow\mathbb R_{\ge 1}$ be given by
\begin{flalign*}
g^*_1(x)=\frac{1}{k-1}(x-1)+1\textrm{ and }
g^*_k(x)=\frac{13k-12}{27(k-1)(k-2)}(x-k)+\frac{2(7k-12)}{9(k-2)}\textrm{, }x\in\mathbb R_{\ge 1}\textrm{.}
\end{flalign*}
\begin{lemma}\label{m_optimization_base_case}
For any $k\in\mathbb Z_{\ge 4}$ we have
\begin{flalign*}
\lim_{x\rightarrow 0}\iota^*_{\mathrm{BP}}(x)=\frac{k-4}{k-3}\textrm{, }
\iota^*_{\mathrm{BP}}(1)=1=g^*_1(1)\textrm{ and }
\iota^*_{\mathrm{BP}}(k)=g^*_k(k)\textrm{.}
\end{flalign*}
For the first derivative ${\iota^*}'_{\mathrm{BP}}$ we have 
\begin{flalign*}
{\iota^*}'_{\mathrm{BP}}(1)={g^*}'_1(1)\textrm{, }
\quad
{\iota^*}'_{\mathrm{BP}}(k)={g^*}'_k(k)\textrm{ and} \quad
\lim_{x\rightarrow\infty}{\iota^*}'_{\mathrm{BP}}(x)=\frac{1}{2(k-2)}\textrm{.}
\end{flalign*}
Moreover, for the second derivative we have that
\begin{flalign*}
{\iota^*}''_{\mathrm{BP}}(x)<0\textrm{ for }x\in(0,k)\textrm{, } \quad
{\iota^*}''_{\mathrm{BP}}(k)=0\textrm{,} \quad 
{\iota^*}''_{\mathrm{BP}}(x)>0\textrm{ for }x\in\mathbb R_{>k}\textrm{.}
\end{flalign*}
For $k=4$ and $x\in\mathbb R_{>0}$ we have $\iota^*_{\mathrm{BP}}(x^{-1})=\left(\iota^*_{\mathrm{BP}}(x)\right)^{-1}$.
\end{lemma}
\begin{proof}
Using $f_p(1)=1$ for the moment generating function of any finitely supported law $p$, we have $\iota^*_{\mathrm{BP}}(1)=1$. Further, using that the first moment of a hypergeometric distribution $p_{N,N_1,N_2}$ is $\frac{N_1N_2}{N}$ and that $f'_p(1)$ is the first moment of $p$, we have ${\iota^*}'_{\mathrm{BP}}(1)=\frac{1}{k-1}+\frac{4}{k-1}-2\cdot\frac{2}{k-1}=\frac{1}{k-1}$.
The symmetry of $\iota^*_{\mathrm{BP}}$ for the special case $k=4$ can be seen as follows.
First, recall that $p_{N,N_1,N_2}(s)=p_{N,N-N_1,N_2}(N_2-s)$ for any hypergeometric distribution $p_{N,N_1,N_2}$.
For $s\in\{0,1,2\}$ this gives 
\begin{flalign*}
p'_{(0,0)}(s)&=p_{3,2,2}(s)=p_{3,1,2}(2-s)=p'_{(1,0)}(2-s)\\
&=p'_{(0,1)}(2-s)=p_{3,2,1}(2-s)=p_{3,1,1}(s-1)=p'_{(1,1)}(s-1)\textrm{.}
\end{flalign*}
These relations can be directly translated to the moment generating functions, i.e.
\begin{flalign*}
f_{p'_{(0,0)}}(x)&=x^2f_{p'_{(1,0)}}(x^{-1})=x^2f_{p'_{(0,1)}}(x^{-1})=xf_{p'_{(1,1)}}(x)
\end{flalign*}
for $x\in\mathbb R_{>0}$. Using these transformations we have
\begin{flalign*}
\iota^*_{\mathrm{BP}}(x^{-1})
=\frac{f_{p'_{(1,1)}}(x^{-1})f_{p'_{(0,0)}}(x^{-1})}{f_{p'_{(1,0)}}(x^{-1})f_{p'_{(0,1)}}(x^{-1})}
=\frac{x^{-1}f_{p'_{(1,0)}}(x)x^{-2}f_{p'_{(0,1)}}(x)}{x^{-1}f_{p'_{(1,1)}}(x)x^{-2}f_{p'_{(0,0)}}(x)}
=\left(\iota^*_{\mathrm{BP}}(x)\right)^{-1}\textrm{.}
\end{flalign*}
For $k\in\mathbb Z_{\ge 4}$ and $x\in\mathbb R_{>0}$ direct computation gives
\begin{flalign*}
\iota^*_{\mathrm{BP}}(x)&=\frac{1}{2(k-2)}x+\frac{2k-5}{2(k-2)}+\frac{(k-1)(k-3)(x-1)}{2(k-2)(2x+k-3)^2}\textrm{,}\\
{\iota^*}'_{\mathrm{BP}}(x)&=\frac{1}{2(k-2)}+\frac{(k-1)(k-3)(-2x+k+1)}{2(k-2)(2x+k-3)^3}\textrm{,}\\
{\iota^*}''_{\mathrm{BP}}(x)&=\frac{4(k-1)(k-3)(x-k)}{(k-2)(2x+k-3)^4}\textrm{.}
\end{flalign*}
The remaining assertions follow immediately or with routine computations.
\end{proof}
Lemma \ref{m_optimization_base_case} has the following immediate consequences.
\begin{corollary}\label{m_optimization_d_le_2}
For any $d\in(0,2]$ we have $\iota^*_{\mathrm{BP}}\in(x,1)$ for $x\in(0,1)$ and $\iota^*_{\mathrm{BP}}(x)\in(1,x)$ for $x\in\mathbb R_{>1}$.
In particular, $p^*$ is the unique minimizer of $\Delta_d$.
\end{corollary}
\begin{proof}
Using Lemma \ref{m_optimization_base_case} we notice that ${\iota^*}'_{\mathrm{BP}}(x)\in[{\iota^*}'_{\mathrm{BP}}(k),\frac{1}{k-1}]\subset(0,1)$ for $x\in\mathbb R_{\ge 1}$ and $\iota^*_{\mathrm{BP}}(x)=x$ for $x=1$, hence we have $\iota^*_{\mathrm{BP}}(x)\in(1,x)$ for $x\in\mathbb R_{>1}$.
For $k=4$ this gives $\iota^*_{\mathrm{BP}}(x)\in(x,1)$ using the symmetry result.
For $k\in\mathbb Z_{>4}$ we have $\lim_{x\rightarrow 0}\iota^*_{\mathrm{BP}}(x)>0$, which gives $x^*=\inf\{x\in\mathbb R_{>0}:\iota^*_{\mathrm{BP}}(x)<x\}\in(0,1]$ using $\iota^*_{\mathrm{BP}}(1)=1$.
Assume that $x^*<1$, then using the continuity of $x-\iota^*_{\mathrm{BP}}(x)$ we directly get $\iota^*_{\mathrm{BP}}(x^*)=x^*$, and further ${\iota^*}'_{\mathrm{BP}}(x^*)\le 1$ since $\iota^*_{\mathrm{BP}}(x)>x$ for $x\in(0,x^*)$.
But then, since ${\iota^*}''_{\mathrm{BP}}(x)<0$ for $x\in(0,k)$, this implies that ${\iota^*}'_{\mathrm{BP}}(x)<1$ for $x\in(x^*,1]$ which yields that ${\iota^*}_{\mathrm{BP}}(1)<1$ and hence a contradiction.
This shows that $\iota^*_{\mathrm{BP}}(x)\in(x,1)$ for $x\in(0,1)$.
Now, for any $d\in(0,2]$ we have $\iota_{\mathrm{BP}}(x)>\iota^*_{\mathrm{BP}}(x)\in(x,1)$ for $x\in(0,1)$ and 
$\iota_{\mathrm{BP}}(x)<\iota^*_{\mathrm{BP}}(x)\in(1,x)$ for $x\in\mathbb R_{>1}$.
Hence, using Lemma \ref{m_optimization_stationary_type} and Lemma \ref{m_optimization_fixed_point_equations} we immediately get that $p^*=p_1$ is the unique minimizer of $\Delta_d$.
\end{proof}
Corollary \ref{m_optimization_d_le_2} covers the case of simple graphs that was discussed in \cite{robalewska1996}.
Lemma \ref{p_second_moment_laplace_conditions} shows that this result is also immediate once we solve the optimization for $d=d^*$. On the other hand, Corollary \ref{m_optimization_d_le_2} suggests that Proposition \ref{s_second_moment_optimization} can only hold if $d^*>2$.
\begin{corollary}\label{m_optimization_d_star_g_2}
For all $k\in\mathbb Z_{\ge 4}$ we have $d^*\in\mathbb R_{>2}$.
\end{corollary}
\begin{proof}
For $d\in\mathbb R_{>0}$ let $f(d)=(\Delta_d\circ\iota_{\mathrm{rp}})(\infty)$ and notice that $f(d)=\frac{k}{d}\phi_1$.
Corollary \ref{m_optimization_d_le_2} shows that $f(d)>0$ for all $d\in(0,2]$.
On the other hand, as derived in Section \ref{s_first_moment} we know that $d^*$ is the unique root of $f$, which shows that $d^*>2$.
\end{proof}
Based on Corollary \ref{m_optimization_d_le_2} we can restrict to $d\in\mathbb R_{>2}$, while Corollary \ref{m_optimization_d_star_g_2} motivates the discussion of this interval.
Further, the restriction $d\in\mathbb R_{>2}$ ensures that $x^{d-1}$ is increasing and convex on $\mathbb R_{>0}$.
In the following we will consider the intervals $\mathbb R_{\ge k}$, $[\bar x,k]$, $[1,\bar x]$ and $(0,1]$ independently, where $\bar x=\frac{1}{7}(k+6)\in(1,k)$ is the intersection of $g_1$ and $g_k$ with $g_1(\bar x)=g_k(\bar x)=\frac{8}{7}$.
\begin{lemma}\label{m_optimization_beyond_inflection}
For $d\in(2,d_k)$, with $d_k=\frac{\ln(k)}{\ln(\iota^*_{\mathrm{BP}}(k))}+1$, there exists $x_{\max}\in\mathbb R_{>k}$ such that $\iota_{\mathrm{BP}}(x)<x$ for $x\in[k,x_{\max})$, $\iota_{\mathrm{BP}}(x_{\max})=x_{\max}$ and $\iota_{\mathrm{BP}}(x)>x$ for $x\in\mathbb R_{>x_{\max}}$.
\end{lemma}
\begin{proof}
Let $f(d)=\left(\iota_{\mathrm{BP}}^*(k)\right)^{d-1}$ for $d\in\mathbb R_{\ge 2}$, i.e.~$f(d)=\iota_{\mathrm{BP}}(k)$ is the value of $\iota_{\mathrm{BP}}$ at $k$ under a variation of $d$.
We know from Lemma \ref{m_optimization_base_case} that $\iota_{\mathrm{BP}}^*(k)\in\mathbb R_{>1}$, hence $f(d)$ is strictly increasing, and further direct computation gives $f(d_k)=k$, so we have $\iota_{\mathrm{BP}}(k)<k$ for any $d\in(2,d_k)$.
Now, since $\iota^*_{\mathrm{BP}}$ is strictly increasing and convex on $\mathbb R_{>k}$ by Lemma \ref{m_optimization_base_case} and further the function $x^{d-1}$ is increasing and convex for $d\in(2,d_k)$, we know that $\iota_{\mathrm{BP}}$ is convex and increasing on $\mathbb R_{>k}$, or formally
\begin{flalign*}
\iota'_{\mathrm{BP}}(x)&=(d-1)\iota^{*}_{\mathrm{BP}}(x)^{d-2}{\iota^*}'_{\mathrm{BP}}(x)=(d-1)\iota_{\mathrm{BP}}(x)\frac{{\iota^*}'_{\mathrm{BP}}(x)}{\iota^*_{\mathrm{BP}}(x)}>0\textrm{, }\\
\iota''_{\mathrm{BP}}(x)&=(d-1)\iota_{\mathrm{BP}}(x)\left[(d-2)\left(\frac{{\iota^*}'_{\mathrm{BP}}(x)}{\iota^*_{\mathrm{BP}}(x)}\right)^2+\frac{{\iota^*}''_{\mathrm{BP}}(x)}{\iota^*_{\mathrm{BP}}(x)}\right]>0\textrm{.}
\end{flalign*}
Using Lemma \ref{m_optimization_base_case} and for $x\rightarrow\infty$ we have
$\iota^{*}_{\mathrm{BP}}(x)\rightarrow\infty$ since ${\iota^*}'_{\mathrm{BP}}(x)\rightarrow\frac{1}{2(k-2)}$ and hence
$\iota'_{\mathrm{BP}}(x)\rightarrow\infty$ by the above, i.e.~$\iota_{\mathrm{BP}}(x)-x\rightarrow\infty$, which suggests the existence of $x_{\max}\in\mathbb R_{>k}$ with $\iota_{\mathrm{BP}}(x_{\max})=x_{\max}$ since $\iota_{\mathrm{BP}}(k)<k$.
Now, let $x_+=\inf\{x\in\mathbb R_{\ge k}:\iota_{\mathrm{BP}}(x)\ge x\}$, then we have $x_+\in(k,x_{\max}]$.
Since $\iota_{\mathrm{BP}}(x)<x$ for $x\in[k,x_+)$ we need $\iota'_{\mathrm{BP}}(x_+)\ge 1$, which gives $\iota'_{\mathrm{BP}}(x)>1$ for $x>x_+$ since $\iota''_{\mathrm{BP}}(x)>0$, hence $\iota_{\mathrm{BP}}(x)>x$, thereby $x_+=x_{\max}$, and in summary $\iota_{\mathrm{BP}}(x)<x$ for $x\in[k,x_{\max})$, $\iota_{\mathrm{BP}}(x_{\max})=x_{\max}$ and $\iota_{\mathrm{BP}}(x)>x$ for $x\in\mathbb R_{>x_{\max}}$.
\end{proof}
The proof of Lemma \ref{m_optimization_beyond_inflection} serves as a blueprint for the next two cases, where we do not consider $\iota_{\mathrm{BP}}$ directly since ${\iota^*}''_{\mathrm{BP}}(x)<0$ on $(1,k)$, but work with $g_k(x)=g^*_k(x)^{d-1}$ and $g_1(x)=g^*_1(x)^{d-1}$ instead, which are convex, increasing and upper bounds for $\iota_{\mathrm{BP}}$ on $[1,k]$ since ${\iota^*}''_{\mathrm{BP}}(x)<0$ on $(1,k)$.
In the spirit of Lemma \ref{m_optimization_beyond_inflection} we continue to consider the maximal domain for $d\in\mathbb R_{>2}$. Let
\begin{flalign*}
d_{\bar x}=\frac{\ln(\bar x)}{\ln\left(g_1(\bar x)\right)}+1\textrm{ and  }
d_{\max}=\min\left(d_{\bar x},d_k\right)\textrm{.}
\end{flalign*}
We postpone the proof that $d^*\le d_{\max}$, instead we focus on the interval $(1,k)$.
\begin{lemma}\label{m_optimization_below_inflection}
For any $d\in(2,d_{\max})\subseteq(2,k)$ and all $x\in(1,k]$ we have $\iota_{\mathrm{BP}}(x)<x$.
\end{lemma}
\begin{proof}
Fix $d\in(2,d_{\max})$.
Since ${\iota^*}''_{\mathrm{BP}}(x)<0$ for $x\in[1,k)$, we know that
$\iota^*_{\mathrm{BP}}(x)\le g^*_k(x)$ for $x\in[\bar x,k]$ and 
$\iota^*_{\mathrm{BP}}(x)\le g^*_1(x)$ for $x\in[1,\bar x]$, so using that $x^{d-1}$ is increasing we have that
$\iota_{\mathrm{BP}}(x)\le g_k(x)$ for $x\in[\bar x,k]$ and
$\iota_{\mathrm{BP}}(x)\le g_1(x)$ for $x\in[1,\bar x]$.
Analogous to Lemma \ref{m_optimization_beyond_inflection} we notice that $g_k(k)=\iota_{\mathrm{BP}}(k)<k$ since $d<d_k$, that $g_k(\bar x)=g_1(\bar x)<\bar x$ since $d<d_{\bar x}$ and that $g_1(1)=1$.
Further, since $g^*_1$, $g^*_k$ are increasing and convex, using that $x^{d-1}$ is increasing and convex yields that $g_1$, $g_k$ are increasing and convex.
In particular, we can upper bound $g_k$ with the line $l_k:[\bar x,k]\rightarrow[g_k(\bar x),g_k(k)]$ connecting $(\bar x,g_k(\bar x))$ and $(k,g_k(k))$, which is entirely and strictly under the diagonal.
Analogously, we can upper bound $g_1$ with the line $l_1:[1,\bar x]\rightarrow[1,g_1(\bar x)]$ connecting $(1,1)$ and $(\bar x,g_1(\bar x))$, which is also entirely and strictly under the diagonal except for $(1,1)$ where the two lines intersect.
In total, $\iota_{\mathrm{BP}}(x)\le\min(g_1(x),g_k(x))\le\min(l_1(x),l_k(x))<x$ for all $x\in(1,k]$.
Finally, another implication is that $\frac{d-1}{k-1}=g'_1(1)\le l'_1(1)<1$ since $l_1$ is below the diagonal, which suggests that $d_{\max}\le k$.
\end{proof}
Combining Lemma \ref{m_optimization_beyond_inflection} and Lemma \ref{m_optimization_below_inflection} shows for any $d\in(2,d_{\max})$ that $\Delta_d\circ\iota_{\mathrm{rp}}$ has exactly one stationary point $x_{\max}$ on $\mathbb R_{>1}$ which is the unique maximizer of $\Delta_d\circ\iota_{\mathrm{rp}}$ on this interval.
Further, since $d_{\max}\le k$ and using $\iota'_{\mathrm{BP}}(1)=\frac{d-1}{k-1}$ we also know that $x=1$ is an isolated minimizer of $\Delta_d\circ\iota_{\mathrm{rp}}$.
Aside, notice that this argumentation can be used to show that $H_d$ as defined in Section \ref{s_second_moment_results} is positive semi-definite for all $d<k$, hence with the arguments from the proof of Lemma \ref{p_second_moment_laplace_conditions} we see that $H_d$ is positive definite for all $d<k$ and finally with Lemma \ref{m_laplace_integral} that $\eta_{\chi^2}=\frac{1}{k-1}$.

For the low overlap region $x\in(0,1)$ we need a significantly different approach, since $\iota^*_{\mathrm{BP}}$ is increasing and concave, but $\iota_{\mathrm{BP}}(x)<\iota^*_{\mathrm{BP}}(x)$ and we need to show that
$\iota_{\mathrm{BP}}(x)>x$.
This means that first order approximations as used for $(1,k)$ are useless since they are upper bounds to $\iota^*_{\mathrm{BP}}$ and there are no immediate implications for $\iota''_{\mathrm{BP}}$ as was the case for $\mathbb R_{>k}$.
However, the symmetric case $k=4$ can be discussed easily.
\begin{corollary}\label{m_optimization_low_overlaps_symmetric}
For $k=4$ and $d\in(2,d_{\max})$ we have
$\iota_{\mathrm{BP}}(x)<x$ for $x\in(0,x_{\max}^{-1})$,
$\iota_{\mathrm{BP}}(x_{\max}^{-1})=x_{\max}^{-1}$, and
$\iota_{\mathrm{BP}}(x)>x$ for $x\in(x_{\max}^{-1},1)$.
\end{corollary}
\begin{proof}
Combining Lemma \ref{m_optimization_below_inflection} and Lemma \ref{m_optimization_beyond_inflection} we have $\iota_{\mathrm{BP}}(x)<x$ for $x\in(1,x_{\max})$ and $\iota_{\mathrm{BP}}(x)>x$ for $x\in(x_{\max},\infty)$, hence using the symmetry from Lemma \ref{m_optimization_base_case} directly gives the result.
\end{proof}
Corollary \ref{m_optimization_low_overlaps_symmetric} allows to restrict to $k\in\mathbb Z_{>4}$ in the remainder.
Now, we basically reverse the method used for the interval $(1,k)$, i.e.~instead of using tangents $g^*_1$, $g^*_k$ to $\iota^*_{\mathrm{BP}}$ and scaling them with $(d-1)$, we scale $\iota^*_{\mathrm{BP}}$ such that the diagonal is a tangent, meaning we consider $\iota_k=\iota_{\mathrm{BP}}$ for $d=k$ since $\iota'_{\mathrm{BP}}(1)=\frac{d-1}{k-1}$, and show that $\iota_k$ is sufficiently convex to ensure $\iota_k(x)>x$ for $x\in(0,1)$.
The next lemma ensures that this approach is applicable for all $k\ge 5$.
\begin{lemma}\label{m_optimization_low_overlaps}
For any $k\in\mathbb Z_{\ge 5}$, $d\in(2,k]$ and all $x\in(0,1)$ we have $\iota_{\mathrm{BP}}(x)>x$.
\end{lemma}
\begin{proof}
Let $k\in\mathbb Z_{\ge 5}$. As derived in the proof of Lemma \ref{m_optimization_below_inflection} we have $\iota_k(1)=1$ and $\iota'_k(1)=1$, i.e.~the diagonal is a tangent to $\iota_k$ at $x=1$. Further, as discussed in the proof of Lemma \ref{m_optimization_beyond_inflection},
\begin{flalign*}
\iota''_k(x)
&=(k-1)\frac{\iota_k(x)}{\iota^*_{\mathrm{BP}}(x)^2}
\left[(k-2){\iota^*}'_{\mathrm{BP}}(x)^2+\iota^*_{\mathrm{BP}}(x){\iota^*}''_{\mathrm{BP}}(x)\right]\textrm{ for }x\in(0,1)\textrm{.}
\end{flalign*}
Since the leading factor is clearly strictly positive, we may focus on the term in the square brackets.
Further, using that moment generating functions are strictly positive for strictly positive real numbers we can extract the strictly positive denominator of the term and normalize to get
\begin{flalign*}
f(x)&=\frac{\iota^*_{\mathrm{BP}}(x)^2\left[(k-1)f_{p'_{(1,0)}}(x)\right]^6(k-2)^2\iota''_k(x)}{(k-1)\iota_k(x)}=\sum_{i\in[6]_0}a_ix^i\textrm{, where}\\
a_6&=16(k-2)\textrm{,}\\
a_5&=48(k-2)(k-3)\textrm{,}\\
a_4&=4(k-3)\left[9(k-2)(k-3)+4(k-2)(k-4)+2(k-1)\right]\textrm{,}\\
a_3&=8(k-3)\left[3(k-2)_3+(k-2)(k^2-3k+4)+2(k-1)(k-4)\right]\textrm{,}\\
a_2&=4(k-3)b_2\textrm{, }
b_2=(k-2)_2\left[(k-4)^2+3(k^2-3k+4)\right]-(k-1)(k^2+11k-36)\textrm{,}\\
a_1&=4(k-3)^2b_1\textrm{, }b_1=(k-2)(k-4)(k^2-3k+4)-2(k-1)(2k^2-3k-4)\textrm{,}\\
a_0&=(k-2)(k-3)^2b_0\textrm{, }
b_0=(k^2-3k+4)^2-4k(k-1)(k-4)\textrm{.}
\end{flalign*}
We can easily verify that $a_i>0$ for $i\in[6]\setminus[2]$ using that $x^2-3x+4>0$ for all $x\in\mathbb R$.
Viewing the coefficients $b_i$, $i\in\{0,1,2\}$, as polynomials $b_i(x)$, $x\in\mathbb R$, of degree $4$ and evaluated at $x=k$, we have
\begin{flalign*}
b''_2(x)=48x^2-228x+254\textrm{, }b_2(5)=82\textrm{, }b'_2(5)=225\textrm{,}
\end{flalign*}
hence $b''_2(x)>0$ for $x>x_2$ with $x_2=\frac{19}{8}+\frac{\sqrt{201}}{24}<3$ and by that $b_2(x)>0$ for all $x\in\mathbb R_{\ge 5}$, so in particular $b_2=b_2(k)>0$ since $k\in\mathbb Z_{\ge 5}$ and thereby $a_2>0$.
Using the same technique for the degree four polynomials $b_1(x)$ we obtain that $b''_1(x)>0$ for $x>x_1$ with $x_1=\frac{13}{4}+\frac{\sqrt{561}}{12}<6$, $b_1(10)=564$, $b'_1(10)=854$, and hence that $b_1>0$ if $k\ge 10$. For $b_0(x)$ we have $b''_0(x)>0$ for $x>x_0$ with $x_0=\frac{5}{2}+\frac{\sqrt{3}}{6}<3$, $b'_0(3)=20$, $b_0(3)=40$, so $b_0>0$ for all $k\in\mathbb Z_{\ge 5}$.

Hence, for $k\in\mathbb Z_{\ge 10}$ we know that $a_i>0$ for all $i\in[6]_0$, which directly implies that $f(x)>0$ for all $x\in\mathbb R_{>0}$, so $\iota''_k(x)>0$, further $\iota'_k(x)<1$ for $x\in(0,1)$, $\iota'_k(x)>1$ for $x\in\mathbb R_{>1}$ and thereby $\iota_k(x)>x$ for $x\in\mathbb R_{>0}\setminus\{1\}$.

For $5\le k\le 9$ we still have $a_i>0$ for $i\in[6]_0\setminus\{1\}$.
For $6\le k\le 9$ we consider the quadratic function $g_k(x)=\sum_{i\in\{0,1,2\}}a_ix^i$ explicitly, given by
\begin{flalign*}
g_6(x)=6120x^2-11664x+8784\textrm{, }
g_7(x)=24960x^2-25344x+41600\textrm{, }\\
g_8(x)=72560x^2-34400x+156000\textrm{, }
g_9(x)=172944x^2-9504x+484848\textrm{.}
\end{flalign*}
It turns out that $g_k(x)>0$ for all $x\in\mathbb R$ and $6\le k\le 9$, which in particular yields $f(x)>0$ for all $x\in\mathbb R_{>0}$. Using the same argumentation as for $k\in\mathbb Z_{\ge 10}$ shows that $\iota_k(x)>x$ for all $x\in\mathbb R_{>0}\setminus\{1\}$ and $k\in\mathbb Z_{\ge 6}$.

As opposed to the previous cases the function $\iota_k$ is not convex for $k=5$, and while this slightly complicates the computation, we will show that this does not affect the overall picture.
Now, we consider the complete sixth order polynomial
\begin{flalign*}
f(x)=48x^6 + 288x^5 + 592x^4 + 1088x^3 + 656x^2 - 3296x + 1392\textrm{.}
\end{flalign*}
We notice that $f''(x)>0$ for all $x\in\mathbb R_{\ge 0}$ since $a_i>0$ for $i\in[6]\setminus\{1\}$, hence $f'(x)$ is strictly increasing for $x\in\mathbb R_{\ge 0}$, which shows the existence of a unique root $x_{\min}\in(0,1)$, using $f'(0)<0$ and $f'(1)>0$, i.e.~$x_{\min}$ is the unique minimizer of $f$ on $\mathbb R_{\ge 0}$.
Computing $f(x_{1-})>0$, $f(x_{1+})<0$ and $f(1)>0$ with $x_{1-}=0.581$ and $x_{1+}=0.582$ ensures the existence of exactly two roots $x_1\in(x_{1-},x_{1+})$ and $x_2\in(x_{1+},1)$ of $f$, hence $\iota_k$ is convex on $(0,x_1)$, concave on $(x_1,x_2)$ and convex on $\mathbb R_{>x_2}$.
Let $g:\mathbb R\rightarrow\mathbb R$, $x\mapsto\iota'_k(x_{1-})(x-x_{1-})+\iota_k(x_{1-})$ denote the tangent of $\iota_k$ at $x_{1-}$.
Since we can write both $\iota_k$ and $\iota'_k$ as the ratio of polynomials with integer coefficients, we can compute $\iota_k(x_{1-})\in(0.584,0.585)$, $\iota'_k(x_{1-})\in(0.99,0.991)$ and $g(x_{1+})\in(0.585,0.586)$ exactly.
Using the convexity of $\iota_k$ on $(0,x_1)$, $g'=\iota'_k(x_{1-})<1$ and $g(x_{1+})>x_{1+}$ immediately gives that $\iota_k(x)\ge g(x)>x$ for $x\in(0,x_1]$.
The fact that $\iota_k(x)>x$ for $x\in[x_2,1)$ immediately follows from the convexity of $\iota_k$ on $(x_2,1)$ and that the diagonal is a tangent to $\iota_k$ at $x=1$.
But now, since $(x_1,\iota_k(x_1))$ and $(x_2,\iota_k(x_2))$ are above the diagonal, so is the line connecting the two, which is a lower bound to $\iota_k$ on $(x_1,x_2)$ since $\iota_k$ is concave on this interval.
By that we have finally showed that the overall picture is also the same for $k=5$, i.e.~$\iota_k(x)>x$ for $x\in\mathbb R_{>0}\setminus\{1\}$.

The fact that $\iota_k(x)>x$, i.e.~$\iota_{\mathrm{BP}}(x)>x$ with $d=k$, for $x\in\mathbb R_{>0}\setminus\{1\}$
shows that the only stationary point of $\Delta_k$ is a saddle at $p^*$ (respectively a maximum of $\Delta_k\circ\iota_{\mathrm{rp}}$ at $x=1$). But more importantly, since $x^{d-1}$ is decreasing in $d$ for $x\in(0,1)$ and $\iota^*_{\mathrm{BP}}(x)\in(0,1)$, we have $\iota_{\mathrm{BP}}(x)\ge\iota_k(x)>x$ for all $d\in(2,k]$.
\end{proof}
The combination of Lemma \ref{m_optimization_beyond_inflection}, Lemma \ref{m_optimization_below_inflection} and Lemma \ref{m_optimization_low_overlaps_symmetric} shows that for $k=4$ and all $d\in(2,d_{\max})$ there exist exactly three fixed points $x_-<x_0<x_+$ of $\iota_{\mathrm{BP}}$, with $x_0=1$, $x_+\in(k,\infty)$ and $x_-=x_+^{-1}$, hence Lemma \ref{m_optimization_fixed_point_equations} and Lemma \ref{m_optimization_stationary_type} suggest that $x_0$ is a minimizer of $\Delta_d\circ\iota_{\mathrm{rp}}$ while $x_-$ and $x_+$ are maximizers. In particular, we have the three minimizers $\{0,1,\infty\}$ of $\Delta_d\circ\iota_{\mathrm{rp}}$ in total.

The combination of Lemma \ref{m_optimization_beyond_inflection}, Lemma \ref{m_optimization_below_inflection} and Lemma \ref{m_optimization_low_overlaps} shows that for $k\in\mathbb Z_{\ge 5}$ and all $d\in(2,d_{\max})\subseteq(0,k)$ there exist exactly two fixed points $x_0<x_+$ of $\iota_{\mathrm{BP}}$, with $x_0=1$ and $x_+\in(k,\infty)$, hence Lemma \ref{m_optimization_fixed_point_equations} and Lemma \ref{m_optimization_stationary_type} suggest that $x_0$ is a minimizer of $\Delta_d\circ\iota_{\mathrm{rp}}$ while $x_+$ is a maximizer. In particular, we have the two minimizers $\{1,\infty\}$ of $\Delta_d\circ\iota_{\mathrm{rp}}$ in total, while $x=0$ is a maximizer in these cases.

The last step is to show that we have $d^*\in(2,d_{\max})$, which then directly establishes that the unique minimizers of $\Delta_{d^*}\circ\iota_{\mathrm{rp}}$ are given by
$\{0,1,\infty\}$ for $k=4$ and 
$\{1,\infty\}$ for $k\in\mathbb Z_{\ge 5}$ as required.
On the other hand, direct computation as in the proof of Corollay \ref{m_optimization_d_star_g_2} shows that all minimizers are roots of $\Delta_{d^*}\circ\iota_{\mathrm{rp}}$, i.e.~all minimizers are global minimizers and the global minimum of $\Delta_{d^*}\circ\iota_{\mathrm{rp}}$ is $0$.
Lemma \ref{m_optimization_minimizers_local} directly suggests that the global minimizers of $\Delta_{d^*}\circ\iota_{\mathrm{rp}}$ are in one to one correspondence with the global minimizers of $\Delta_{d^*}$ via $x\mapsto p_x$, which then completes the proof of Proposition \ref{p_second_moment_optimization}.
\begin{lemma}\label{m_optimization_d_star_le_d_max}
For all $k\in\mathbb Z_{\ge 4}$ we have $2<d^*<d_k<d_{\bar x}$.
\end{lemma}
\begin{proof}
Recall from Corollary \ref{m_optimization_d_star_g_2} that $d^*>2$.
For convenience, we consider the extensions of $d^*-1$, $d_{\bar x}-1$ and $d_k-1$ to the real line, i.e.~for $x\in\mathbb R_{\ge 3}$ let
\begin{flalign*}
f_0(x)=\frac{\ln\left(\frac{1}{2}x(x-1)\right)}{xH(2/x)-\ln\left(\frac{1}{2}x(x-1)\right)}\textrm{, }
f_1(x)=\frac{\ln\left(\frac{1}{7}(x+6)\right)}{\ln\left(8/7\right)}\textrm{, }
f_2(x)=\frac{\ln(x)}{\ln\left(\frac{2(7x-12)}{9(x-2)}\right)}\textrm{,}
\end{flalign*}
i.e.~$d^*=f_0(k)+1$, $d_{\bar x}=f_1(k)+1$ and $d_k=f_2(k)+1$ for all $k\in\mathbb Z_{\ge 4}$.
We start with the asymptotical comparison of $f_1$ and $f_2$. The corresponding rearrangement gives
\begin{flalign*}
f_1(x)&=m_1\ln(x)+t_1(x)\textrm{, }m_1=\frac{1}{\ln(8/7)}\textrm{, }t_1(x)=-\frac{\ln(7)}{\ln(8/7)}+\frac{\ln\left(1+\frac{6}{x}\right)}{\ln(8/7)}\textrm{,}\\
f_2(x)&=m_2(x)\ln(x)\textrm{, }m_2(x)=\frac{1}{\ln(14/9)+\ln\left(1+\frac{2}{7(x-2)}\right)}\textrm{.}
\end{flalign*}
Notice that $\ln(x)>0$ since $x\ge 3$, further $t_1(x)$ is decreasing while $m_2(x)$ is increasing,
and thereby we have $f_1(x)\ge f_{1\infty}(x)$ and $f_2(x)\le f_{2\infty}(x)$ with
\begin{flalign*}
f_{1\infty}(x)
&=m_{1\infty}\ln(x)+t_{1\infty}\textrm{, }m_{1\infty}=m_1\textrm{, }t_{1\infty}=-\frac{\ln(7)}{\ln(8/7)}\textrm{,}\\
f_{2\infty}(x)
&=m_{2\infty}\ln(x)\textrm{, }m_{2\infty}=\frac{1}{\ln(14/9)}\textrm{.}
\end{flalign*}
Notice that $m_{1\infty}>m_{2\infty}$, $t_{1\infty}<0$ and further
$f_{1\infty}(x)>f_{2\infty}(x)$ iff $x>x_{12}$ with $x_{12}=\exp\left(\frac{-t_{1\infty}}{m_{1\infty}-m_{2\infty}}\right)\in(16,17)$, so $f_1(x)>f_2(x)$ for all $x\in\mathbb R_{\ge 17}$ and hence $d_k<d_{\bar x}$ for $k\in\mathbb Z_{\ge 17}$.
We check by hand that $d_k<d_{\bar x}$ also holds for $4\le k\le 16$.

Hence, we're left to show that $1<f_0(x)<f_2(x)$ for $x\in\mathbb Z_{\ge 4}$.
Again, we start with the asymptotical comparison, where the corresponding rearrangement
$f_0(x)=m_0(x)\ln(x)+t_0(x)$ is given by
\begin{flalign*}
m_0(x)&=\frac{1+\ln\left(1-\frac{1}{x}\right)}{n_0(x)}\textrm{,}~~
t_0(x)=\frac{-\ln(2)}{n_0(x)}\textrm{,}~~
n_0(x) =-(x-2)\ln\left(1-\frac{2}{x}\right)-\ln\left(1-\frac{1}{x}\right)-\ln(2)\textrm{.}
\end{flalign*}
Recall that for given $c\in\mathbb R$ we have $\ln(1+\frac{c}{x})\sim\frac{c}{x}$ and $\ln(1+\frac{c}{x})\le\frac{c}{x}$ for all $x\in\mathbb R_{>|c|}$, since $\ln(x)$ is concave and the tangent at $1$ is $x-1$.
Hence, for all $x\in\mathbb R_{\ge 3}$ we have $n_0(x)\ge n_+(x)>0$ since $x>2$ and $x>x_1$, where
\begin{flalign*}
n_+(x)=\ln\left(\frac{e^2}{2}\right)-\frac{3}{x}\textrm{ and }x_1=\frac{3}{\ln\left(\frac{e^2}{2}\right)}\in(2,3)\textrm{.}
\end{flalign*}
Since $t_0(x)<0$ and $m_0(x)\le m_+(x)$ we have $f_0(x)\le f_+(x)=m_+(x)\ln(x)$ with
\begin{flalign*}
m_+(x)=\frac{1}{\ln\left(\frac{e^2}{2}\right)-\frac{3}{x}}\textrm{.}
\end{flalign*}
Now, since $m_+(x)$ is decreasing in $x$ and $m_2(x)$ is increasing in $x$, we numerically determine $x^*\in\mathbb R_{>0}$ such that $m_+(x^*)=m_2(x^*)$ and find that $x^*\in(4,5)$. In particular, we have $f_0(x)\le f_+(x)<f_2(x)$ for $x\in\mathbb R_{\ge 5}$, and check that $d^*<d_k$ for $k=4$ by hand.
\end{proof}
Lemma \ref{m_optimization_d_star_le_d_max} concludes the proof of Proposition \ref{p_second_moment_optimization} as discussed before.
\begin{figure}
\begin{subfigure}[t]{0.497\textwidth}
\centering
\begin{tikzpicture}
  [scale=0.8, every node/.style={transform shape},
    var/.style={circle,draw,thick,minimum size=4mm},
    fac/.style={rectangle,draw,thick,minimum size=4mm},
    varedg/.style={circle,draw,thick,minimum size=2mm,inner sep=0},
    facedg/.style={rectangle,draw,thick,minimum size=2mm,inner sep=0}]
\node[var] (v1) at (4,5) {};\node[inner sep=0, outer sep=0] at (4,5.4) {$i_2$};
\node[var] (v2) at (1,1) {};\node[inner sep=0, outer sep=0] at (1,0.55) {$i_1$};
\node[var] (v3) at (7,1) {};\node[inner sep=0, outer sep=0] at (7,0.55) {$i_3$};
\node[fac] (f1) at (1,5) {};\node[inner sep=0, outer sep=0] at (1,5.4) {$a_2$};
\node[fac] (f2) at (7,5) {};\node[inner sep=0, outer sep=0] at (7,5.4) {$a_3$};
\node[fac] (f3) at (4,1) {};\node[inner sep=0, outer sep=0] at (4,0.55) {$a_1$};
\node[varedg] (ve11) at (5,5) {};\node[inner sep=0, outer sep=0] at (5.1,4.67) {$h_{i_2,3}$};
\node[varedg] (ve12) at (3,5) {};\node[inner sep=0, outer sep=0] at (3.1,4.67) {$h_{i_2,2}$};
\node[varedg] (ve13) at (4,3.67) {};\node[inner sep=0, outer sep=0] at (4.5,3.62) {$h_{i_2,1}$};
\node[varedg] (ve21) at (2,1) {};\node[inner sep=0, outer sep=0] at (1.8,1.35) {$h_{i_1,1}$};
\node[varedg] (ve22) at (1,2.33) {};\node[inner sep=0, outer sep=0] at (1.5,2.28) {$h_{i_1,2}$};
\node[varedg] (ve31) at (7,2.33) {};\node[inner sep=0, outer sep=0] at (6.5,2.28) {$h_{i_3,1}$};
\node[varedg] (ve32) at (6,1) {};\node[inner sep=0, outer sep=0] at (6.1,1.35) {$h_{i_3,2}$};
\node[facedg] (fe11) at (2,5) {};\node[inner sep=0, outer sep=0] at (2.1,4.67) {$h_{a_2,1}$};
\node[facedg] (fe12) at (1,3.67) {};\node[inner sep=0, outer sep=0] at (1.5,3.62) {$h_{a_2,2}$};
\node[facedg] (fe21) at (6,5) {};\node[inner sep=0, outer sep=0] at (6.1,4.67) {$h_{a_3,2}$};
\node[facedg] (fe22) at (7,3.67) {};\node[inner sep=0, outer sep=0] at (6.5,3.62) {$h_{a_3,1}$};
\node[facedg] (fe31) at (5,1) {};\node[inner sep=0, outer sep=0] at (4.8,1.35) {$h_{a_1,3}$};
\node[facedg] (fe32) at (4,2.33) {};\node[inner sep=0, outer sep=0] at (3.5,2.28) {$h_{a_1,2}$};
\node[facedg] (fe33) at (3,1) {};\node[inner sep=0, outer sep=0] at (3.1,1.35) {$h_{a_1,1}$};
\draw[thick](f1) to (fe11);\draw[thick] (fe11) to (ve12);\draw[thick] (ve12) to (v1);
\draw[thick](f1) to (fe12);\draw[thick] (fe12) to (ve22);\draw[thick] (ve22) to (v2);
\draw[thick](f2) to (fe21);\draw[thick] (fe21) to (ve11);\draw[thick] (ve11) to (v1);
\draw[thick](f2) to (fe22);\draw[thick] (fe22) to (ve31);\draw[thick] (ve31) to (v3);
\draw[thick](f3) to (fe31);\draw[thick] (fe31) to (ve32);\draw[thick] (ve32) to (v3);
\draw[thick](f3) to (fe32);\draw[thick] (fe32) to (ve13);\draw[thick] (ve13) to (v1);
\draw[thick](f3) to (fe33);\draw[thick] (fe33) to (ve21);\draw[thick] (ve21) to (v2);
\draw[->](1.3,1.1) to (1.8,1.1);
\draw[->](6.9,1.41) to (6.9,1.91);
\node[inner sep=0, outer sep=0] at (2.5,3) {$\gamma_1$};
\node[inner sep=0, outer sep=0] at (5.5,3) {$\gamma_2$};
\draw[->](2,2.5) to (3,2.5);\draw[->](3,2.5) to (3,3.5);\draw[->](3,3.5) to (2,3.5);\draw[->](2,3.5) to (2,2.5);
\draw[->](5,2.5) to (6,2.5);\draw[->](6,2.5) to (6,3.5);\draw[->](6,3.5) to (5,3.5);\draw[->](5,3.5) to (5,2.5);
\begin{scope}
\path (0,0) to (8,0) to (8,6) to (0,6) to (0,0);
\end{scope}
\end{tikzpicture}
\subcaption{pair $(\gamma_1,\gamma_2)$ of intersecting $4$-cycles}\label{f_relative_position_representation_gamma}
\end{subfigure}
\begin{subfigure}[t]{0.497\textwidth}
\centering
\begin{tikzpicture}
  [scale=0.8, every node/.style={transform shape},
    var/.style={circle,draw,thick,minimum size=4mm},
    fac/.style={rectangle,draw,thick,minimum size=4mm},
    varedg/.style={circle,draw,thick,minimum size=2mm,inner sep=0},
    facedg/.style={rectangle,draw,thick,minimum size=2mm,inner sep=0}]
\node[var] (v1) at (4,5) {};\node[inner sep=0, outer sep=0] at (4,5.35) {$2$};
\node[var] (v2) at (1,1) {};\node[inner sep=0, outer sep=0] at (1,0.55) {$1$};
\node[var] (v3) at (7,1) {};\node[inner sep=0, outer sep=0] at (7,0.55) {$3$};
\node[fac] (f1) at (1,5) {};\node[inner sep=0, outer sep=0] at (1,5.35) {$2$};
\node[fac] (f2) at (7,5) {};\node[inner sep=0, outer sep=0] at (7,5.35) {$3$};
\node[fac] (f3) at (4,1) {};\node[inner sep=0, outer sep=0] at (4,0.55) {$1$};
\node[varedg] (ve11) at (5,5) {};\node[inner sep=0, outer sep=0] at (5,4.67) {$3$};
\node[varedg] (ve12) at (3,5) {};\node[inner sep=0, outer sep=0] at (3,4.67) {$2$};
\node[varedg] (ve13) at (4,3.67) {};\node[inner sep=0, outer sep=0] at (4.3,3.62) {$1$};
\node[varedg] (ve21) at (2,1) {};\node[inner sep=0, outer sep=0] at (2,1.3) {$1$};
\node[varedg] (ve22) at (1,2.33) {};\node[inner sep=0, outer sep=0] at (1.3,2.28) {$2$};
\node[varedg] (ve31) at (7,2.33) {};\node[inner sep=0, outer sep=0] at (6.7,2.28) {$1$};
\node[varedg] (ve32) at (6,1) {};\node[inner sep=0, outer sep=0] at (6,1.3) {$2$};
\node[facedg] (fe11) at (2,5) {};\node[inner sep=0, outer sep=0] at (2,4.67) {$1$};
\node[facedg] (fe12) at (1,3.67) {};\node[inner sep=0, outer sep=0] at (1.3,3.62) {$2$};
\node[facedg] (fe21) at (6,5) {};\node[inner sep=0, outer sep=0] at (6,4.67) {$2$};
\node[facedg] (fe22) at (7,3.67) {};\node[inner sep=0, outer sep=0] at (6.7,3.62) {$1$};
\node[facedg] (fe31) at (5,1) {};\node[inner sep=0, outer sep=0] at (5,1.3) {$3$};
\node[facedg] (fe32) at (4,2.33) {};\node[inner sep=0, outer sep=0] at (3.7,2.28) {$2$};
\node[facedg] (fe33) at (3,1) {};\node[inner sep=0, outer sep=0] at (3,1.3) {$1$};
\draw[thick](f1) to (fe11);\draw[thick] (fe11) to (ve12);\draw[thick] (ve12) to (v1);
\draw[thick](f1) to (fe12);\draw[thick] (fe12) to (ve22);\draw[thick] (ve22) to (v2);
\draw[thick](f2) to (fe21);\draw[thick] (fe21) to (ve11);\draw[thick] (ve11) to (v1);
\draw[thick](f2) to (fe22);\draw[thick] (fe22) to (ve31);\draw[thick] (ve31) to (v3);
\draw[thick](f3) to (fe31);\draw[thick] (fe31) to (ve32);\draw[thick] (ve32) to (v3);
\draw[thick](f3) to (fe32);\draw[thick] (fe32) to (ve13);\draw[thick] (ve13) to (v1);
\draw[thick](f3) to (fe33);\draw[thick] (fe33) to (ve21);\draw[thick] (ve21) to (v2);
\draw[->](1.3,1.1) to (1.8,1.1);
\draw[->](6.9,1.41) to (6.9,1.91);
\node[inner sep=0, outer sep=0] at (2.5,3) {$\rho_1$};
\node[inner sep=0, outer sep=0] at (5.5,3) {$\rho_2$};
\draw[->](2,2.5) to (3,2.5);\draw[->](3,2.5) to (3,3.5);\draw[->](3,3.5) to (2,3.5);\draw[->](2,3.5) to (2,2.5);
\draw[->](5,2.5) to (6,2.5);\draw[->](6,2.5) to (6,3.5);\draw[->](6,3.5) to (5,3.5);\draw[->](5,3.5) to (5,2.5);
\begin{scope}
\path (0,0) to (8,0) to (8,6) to (0,6) to (0,0);
\end{scope}
\end{tikzpicture}
\subcaption{relative positions $(\rho_1,\rho_2)$ for $(\gamma_1,\gamma_2)$}\label{f_relative_position_representation_relative_positions}
\end{subfigure}
\captionsetup{width=0.89\textwidth, font=small}
\caption{The left figure shows a sequence $\gamma=(\gamma_1,\gamma_2)$ of two directed (intersecting) four-cycles with base variables $i_1$ and $i_3$ and directions indicated by the arrows respectively. Analogously to Figure \ref{f_related_problems_2_configuration} we only denoted the $i$-edges and $a$-edges instead of the v-edges and f-edges.
The relative positions $\rho=(\rho_1,\rho_2)$ corresponding to $\gamma$ are depicted in the right figure. Here, the variables, constraints, $i$-edges and $a$-edges are labeled according to the order of first traversal (where $\gamma_1$ is traversed before $\gamma_2$).
The numbers $n(\rho)=3$, $m(\rho)=3$, $e(\rho)=7$ of variables, constraints and edges in $\rho$ are equal to the corresponding numbers in $\gamma$, further the degree $d_j(\rho)$ of the variable $j\in[3]$ equals the degree of $i_j$ in $\gamma$, and analogously for the degrees $k_b(\rho)$ of the constraints $b\in[3]$ in $\rho$.
The absolute values $\alpha=(\alpha_v,\alpha_f,(\alpha_{v,j})_{j\in[3]},(\alpha_{f,b})_{b\in[3]})$ are given by $\alpha_v=(i_j)_{j\in[3]}$,
$\alpha_f=(a_b)_{b\in[3]}$,
$\alpha_{v,j}=(h_{i_j,e})_{e\in[d_j(\rho)]}$, $j\in[3]$, and
$\alpha_{f,b}=(h_{a_b,e})_{e\in[k_b(\rho)]}$, $b\in[3]$, i.e.~they store the (initial) labels of $\gamma$ corresponding to the labels of $\rho$.}\label{f_relative_position_representation}
\end{figure}
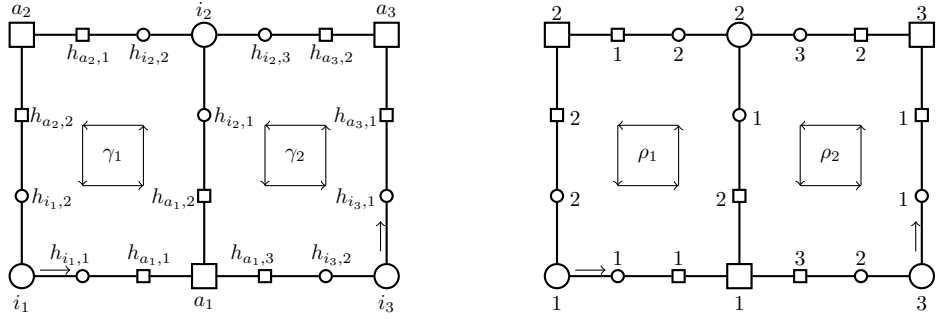
\section{Small Subgraph Conditioning}\label{s_small_subgraph_conditioning}
In this section we prove the remaining parts of Theorem \ref{small_subgraph_conditioning}, thereby establishing Theorem \ref{satisfiability_threshold_configuration}.
The first part of the proof heavily relies on Section \ref{a_number_of_cycles} and illustrates the correspondences.
We start with the derivation of  $\delta_\ell$ by computing $\mathbb E[ZX_\ell]$.
For this purpose we fix $\ell\in\mathbb Z_{>0}$, $n\in\mathcal N$ sufficiently large, and let $\bar c_\ell$ denote the canonical $2\ell$-cycle, i.e.~the cycle with variables $i$, constraints $a$ in $[\ell]$ and $i$-edges, $a$-edges in $\{1,2\}$ with labels ordered by first traversal, see e.g.~the left cycle in Figure~\ref{f_relative_position_representation_relative_positions}.
Analogous to the previous sections
we rewrite the expectation and count the number $|\mathcal E|$ of triplets $(g,c,x)\in\mathcal E$ such that $c$ is a $2\ell$-cycle and $x$ a solution in $g$, i.e.
\begin{flalign*}
\mathbb E[ZX_\ell]&=\frac{|\mathcal E|}{|\mathcal G|}
=\sum_{y\in\{0,1\}^\ell}\frac{e_1e_2e_3}{2\ell(dn)!}\textrm{, where}\\
e_1&=e_1(y)=\binom{n}{n_1}(n_1)_{r_1}(n-n_1)_{\ell-r_1}(d(d-1))^\ell\textrm{,}\\
e_2&=e_2(y)=\binom{k}{2}^m(m)_{\ell}2^{r_2}(2(k-2))^{2(r_1-r_2)}((k-2)(k-3))^{\ell-2r_1+r_2}\textrm{,}\\
e_3&=e_3(y)=(dn_1-2r_1)!(d(n-n_1)-2(\ell-r_1))!\textrm{,}
\end{flalign*}
and $r=r(y)=(r_1,r_2)$ is defined as follows.
For $y\in\{0,1\}^\ell$ we let $r_1=r_1(y)$ denote the number of ones in $y$.
Further, let $r_2=r_2(y)$ denote the number of constraints $b\in[\ell]$ in $\bar c_\ell$ such that both $b$-edges take the value one under the assignment $y$ of the variables $j\in[\ell]$ in $\bar c_\ell$. With $y$ fixed we can compute the number of suitable triplets $(g,c,x)$ as follows.
The denominator in the first line reflects $|\mathcal G|^{-1}$ and the compensation $2\ell$ as we will count directed cycles $\gamma$ in $g$.
The sum over $y\in\{0,1\}^\ell$ implements the choice of the assignment of the variables visited by $\gamma$ such that the variables $i_1,\dots,i_\ell$ traversed by $\gamma$ correspond to the variables $1,\dots,\ell$ in $\bar c_\ell$ in this order, i.e. $x_{i_1}=y_1,\dots,x_{i_\ell}=y_\ell$.
The first term in $e_1$ chooses the variables that take the value one under the solution $x$.
Then we choose the $r_1$ variables out of the $n_1$-variables that participate in the directed cycle $\gamma$ and take the value one consistent with $y$ (hence an ordered choice).
Analogously, we then choose the variables in $\gamma$ taking zero under $x$.
Finally, we choose the two $i$-edges traversed by $\gamma$ for each of the $\ell$ variables $i$ in the cycle.

The first term in $e_2$ is the usual choice of the two $a$-edges taking one under $x$ for each $a\in[m]$.
Then we choose the constraints visited by $\gamma$.
The remaining terms account for the ordered choice of the two $a$-edges that are traversed by $\gamma$ and that is consistent with the assignments $y$ and $x$ in the following sense.
The (already chosen) variables $i_1,\dots,i_\ell$ and constraints $a_1,\dots,a_\ell$ traversed by $\gamma$ correspond to the variables $1,\dots,\ell$ and constraints $1,\dots,\ell$ in $\bar c_\ell$ in this order respectively.
Further, the assignment of these variables is already fixed by $y$ and the $a$-edges taking the value one for each of these constraints are also fixed by our previous choice.
Hence, if $y_1=y_2=1$, then we have only two choices for the $a_1$-edge connecting to $i_1$, while the $a_1$-edge connecting to $i_2$ is fixed afterwards. For $y_1=1$ and $y_2=0$ we have two choices for the $a_1$-edge connecting to $i_1$ and $(k-2)$ choices for the $a_1$-edge connecting to $i_2$. The case $y_1=0$, $y_2=1$ is symmetric and we see that we have $(k-2)$ and $(k-3)$ choices for the remaining case $y_1=y_2=0$ analogously.
To derive the number of constraints for each of the cases above we recall that we have $r_1(y)$ ones in total and $r_2(y)$ ones whose successor is one (i.e.~the constraint $a$ between the two ones takes the value one on both $a$-edges, and where the successor of $y_\ell$ is $y_1$). But then $(r_1-r_2)$ ones in $y$ do not have the successor one, i.e.~they have the successor zero. Complementarily we see that since $r_2$ ones are succeeded by a one there are $r_2$ ones that are preceeded by a one, hence there are $(r_1-r_2)$ ones that are preceeded by zero.
Then again, this means that there are $(r_1-r_2)$ zeros that are succeeded by a one, hence the remaining $(\ell-2r_1+r_2)$ zeros out of the $(\ell-r_1)$ zeros are succeeded by a zero.
This fixes $\gamma$, so in particular $2r_1$ v-edges that take the value one and $2(\ell-r_1)$ v-edges that take the value zero.
The two terms in $e_3$ then wire the remaining edges.

We divide by $\mathbb E[Z]$ to match the left hand side of Theorem
\ref{small_subgraph_conditioning}\ref{small_subgraph_conditioning_joint_dist}, i.e.
\begin{flalign*}
\frac{\mathbb E[ZX_\ell]}{\mathbb E[Z]}
&=\sum_{y\in\{0,1\}^\ell}\frac{e_1e_2e_3}{2\ell(2m)!(dn-2m)!}\textrm{, where}\\
e_1&=e_1(y)=(n_1)_{r_1}(n-n_1)_{\ell-r_1}(d(d-1))^\ell\textrm{,}\\
e_2&=e_2(y)=(m)_{\ell}2^{r_2}(2(k-2))^{2(r_1-r_2)}((k-2)(k-3))^{\ell-2r_1+r_2}\textrm{,}\\
e_3&=e_3(y)=(dn_1-2r_1)!(d(n-n_1)-2(\ell-r_1))!\textrm{,}
\end{flalign*}
and using Stirling's formula we easily derive that
\begin{flalign*}
\frac{\mathbb E[ZX_\ell]}{\mathbb E[Z]}&\sim\lambda_\ell\sum_{y\in\{0,1\}^\ell}
M_{11}^{r_2}M_{01}^{r_1-r_2}M_{10}^{r_1-r_2}M_{00}^{\ell-2r_1+r_2}
=\lambda_\ell(1+\delta_\ell)\textrm{, }~
M=\begin{pmatrix}
1-\frac{2}{k-1} & 1-\frac{1}{k-1}\\
\frac{2}{k-1} & \frac{1}{k-1}
\end{pmatrix}\textrm{.}
\end{flalign*}
The matrix $M$ has a nice interpretation as a (column stochastic) transition probability matrix in a two state Markov process, with
\begin{flalign*}
1+\delta_\ell
=\sum_{\substack{y\in\{0,1\}^\ell\\y_1=0}}M_{11}^{r_2}M_{01}^{r_1-r_2}M_{10}^{r_1-r_2}M_{00}^{\ell-2r_1+r_2}
+\sum_{\substack{y\in\{0,1\}^\ell\\y_1=1}}M_{11}^{r_2}M_{01}^{r_1-r_2}M_{10}^{r_1-r_2}M_{00}^{\ell-2r_1+r_2}
\end{flalign*}
reflecting the probabilities that we return to the starting point given that the starting point is zero and one respectively.
Let us consider the first partial sum restricted to sequences $y$ (of Markov states) such that $y_1=0$, i.e. we start in the state zero.
Then $M_{0y_2}$ reflects the probability that we move from the initial state zero to the state $y_2$ given that we are in state zero (which is the case because we know that $y_1=0$).
As discussed above we will move from a one to a one in $y$ exactly $r_2$ times, from a one to a zero $(r_1-r_2)$ times, from a zero to a one $(r_1-r_2)$ times and from a zero to a zero $(\ell-2r_1+r_2)$ times. Hence the contribution to the first partial sum for given $y$ exactly reflects the probability that we start in the state zero and (with this given) return to the state zero after $\ell$ steps (since the successor of $y_\ell$ is $y_1=0$). Since we sum over all such sequences $y$ the first sum reflects the probability that we reach state zero after $\ell$ steps given that we start in the state zero. The discussion of the second sum is completely analogous. This directly yields
\begin{flalign*}
1+\delta_\ell=(M^\ell)_{00}+(M^\ell)_{11}=\Tr(M^\ell)=\lambda'_1+\lambda'_2=\lambda_1^\ell+\lambda_2^\ell\textrm{, }
\quad
\lambda_1=1\textrm{, }
\lambda_2=-\frac{1}{k-1}\textrm{,}
\end{flalign*}
where we used the Kolmogorov-Chapman equalities in the first step, i.e.~that the $\ell$-step transition probability matrix is the $\ell$-th power of the one step transition probability matrix, which allow to translate the first sum into the transition probability $(M^\ell)_{00}$ that we reach the state zero after $\ell$ steps given that we start in the state zero and analogously for the second sum.
In the second step we use the definition of the trace, while in the third step we use that the trace is the sum of the eigenvalues $\lambda'_1$, $\lambda'_2$ of $M^\ell$. In the next step we use that the eigenvalues $\lambda'_1$ $\lambda'_2$ of the $\ell$-th power $M^\ell$ of the matrix $M$ are the $\ell$-th powers of the eigenvalues $\lambda_1$, $\lambda_2$ of $M$.
In particular this also yields that $\delta_\ell>-1$ for all $k>3$ and establishes $\delta_\ell=(1-k)^{-\ell}$.

Following the strategy of Section \ref{a_number_of_cycles} we turn to the case of disjoint cycles. Similarly, the present case is a canonical extension of the single cycle case discussed above.
We fix $\bar \ell\in\mathbb Z_{>0}$, $r\in\mathbb Z_{\ge 0}^{\bar \ell}$ and $n\in\mathcal N$ sufficiently large.
Further, as in the previous sections we
rewrite the expectation and count the number $|\mathcal E|$ of triplets $(g,c,x)\in\mathcal E$ such that
$c=(c_s)_{s\in[\bar r]}$ is a sequence of $\bar r=\sum_{\ell\in[\bar \ell]}r_\ell$ distinct $2\ell_s$-cycles $c_s$ in the configuration $g$ sorted by their length $\ell_s$ in ascending order (as described in Section \ref{a_number_of_cycles}) and $x$ is a solution of $g$.
This yields
\begin{flalign*}
\mathbb E\left[Z\prod_{\ell\in[\bar \ell]}(X_\ell)_{r_\ell}\right]=\frac{|\mathcal E|}{|\mathcal G|}
=\frac{|\mathcal E_0|}{|\mathcal G|}+\frac{|\mathcal E_1|}{|\mathcal G|}\textrm{,}
\end{flalign*}
where $\mathcal E_0\subseteq\mathcal E$ is the set over all triplets $(g,c,x)\in\mathcal E$ involving sequences $c$ of disjoint cycles and $\mathcal E_1=\mathcal E\setminus\mathcal E_0$.
We begin with the first contribution, which can be regarded as a combination of the discussion of disjoint cycles in Section \ref{a_number_of_cycles} and the single cycle case above, i.e.
\begin{flalign*}
\frac{|\mathcal E_0|}{|\mathcal G|}&=\sum_{y\in\{0,1\}^{\mathfrak l}}\frac{e_1e_2e_3}{(dn)!\prod_{s\in[\bar r]}(2\ell_s)}\textrm{,}\\
e_1&=e_1(y)=\binom{n}{n_1}(n_1)_{\mathfrak r_1}(n-n_1)_{\mathfrak l-\mathfrak r_1}(d(d-1))^{\mathfrak l}\textrm{,}\\
e_2&=e_2(y)=\binom{k}{2}^m(m)_{\mathfrak l}2^{\mathfrak r_2}(2(k-2))^{2(\mathfrak r_1-\mathfrak r_2)}((k-2)(k-3))^{\mathfrak l-2\mathfrak r_1+\mathfrak r_2}\textrm{,}\\
e_3&=e_3(y)=(dn_1-2\mathfrak r_1)!(d(n-n_1)-2(\mathfrak l-\mathfrak r_1))!\textrm{,}\\
\mathfrak l&=\sum_{s\in[\bar r]}\ell_s\textrm{, }
\mathfrak r_i=\sum_{s\in[\bar r]}r_i(y_s)\textrm{, }i\in[2]\textrm{,}
\end{flalign*}
where $y=(y_s)_{s\in[\bar r]}$ is the subdivision of $y$ corresponding to the definition of $c$, and $r_1$, $r_2$ are the notions defined above.
The combinatorial arguments are now fairly self-explanatory, e.g.~we make an ordered choice of the $r_1(y_1)$ variables taking one for $\gamma_1$, then an ordered choice of $r_1(y_2)$ variables taking one for $\gamma_2$ out of the remaining $n_1-r_1(y_1)$ variables taking one and so on.

The asymptotics are also completely analogous to the single cycle case and Section \ref{a_number_of_cycles}.
First, we notice that the sum is still bounded, i.e. we can also use the asymptotic equivalences for the corresponding ratio here.
Then, the sum can be decomposed into the product of the $\bar r$ factors that correspond to the single cycle case above, analogously to Section \ref{a_number_of_cycles}, which yields
\begin{flalign*}
\frac{|\mathcal E_0|}{|\mathcal G|\mathbb E[Z]}\sim\prod_{\ell\in[\bar \ell]}\lambda_\ell^{r_\ell}(1+\delta_\ell)^{r_\ell}\textrm{.}
\end{flalign*}
Now we turn to the proof that the second contribution involving $\mathcal E_1$ is negligible, which is a combination of the above and the discussion of intersecting cycles in Section \ref{a_number_of_cycles}.
We let
\begin{flalign*}
\mathcal E_2&=\{(g,\gamma,x):(g,c(\gamma),x)\in\mathcal E_1\}\textrm{, }
\mathcal R=\{\rho(\gamma):(g,\gamma,x)\in\mathcal E_2\}\textrm{ and }\\
\mathcal E_{\rho}&=\{(g,\gamma,x)\in\mathcal E_2:\rho(\gamma)=\rho\}\textrm{ for }\rho\in\mathcal R
\end{flalign*}
denote the sets that match the corresponding sets in Section \ref{a_number_of_cycles}.
For relative positions $\rho\in\mathcal R$ we consider an assignment $y\in\{0,1\}^{n(\rho)}$ of the variables $V=[n(\rho)]$ in the corresponding union of cycles $c=c(\rho)$ and
let
\begin{flalign*}
\mathfrak r_1&=\mathfrak r_1(\rho,y)=|\{j\in V:y_j=1\}|\textrm{,}\\
\mathfrak o(b)&=\mathfrak o_{\rho,y}(b)=|\{h\in[k_b(\rho)]:y_{i_c(b,h)}=1\}|\textrm{ for }b\in[m(\rho)]\textrm{ and}\\
\mathfrak o&=\mathfrak o(\rho,y)=\sum_{b\in[m(\rho)]}\mathfrak o(b)
\end{flalign*}
denote the number of variables $j\in V$ in $c$ that take the value one under $y$,
the number of $b$-edges for a constraint $b\in[m(\rho)]$ in $c$ that take the value one under $y$ and
the number of f-edges in $c$ that take the value one under $y$ respectively.
Since $c$ is a configuration the number of v-edges in $c$ that take the value one under $y$ is also $\mathfrak o$.
We are particularly interested in the assignments
\begin{flalign*}
y\in\mathcal Y=\mathcal Y(\rho)=\{z\in\{0,1\}^{n(\rho)}:\forall b\in[m(\rho)]\mathfrak o(b)\in[2+k_b-k,2]\}
\end{flalign*}
that do not directly violate a constraint $b\in[m(\rho)]$ in $c(\rho)$ in the sense that $\mathfrak o(b)\le 2$ and also do not indirectly violate $b$ in that $2-\mathfrak o(b)\le k-k_b$, i.e.~there are sufficiently many $b$-edges left to take the remaining $(2-\mathfrak o(b))$ ones.
With this slight extension of our machinery we can derive
\begin{flalign*}
\frac{|\mathcal E_1|}{|\mathcal G|}
&=\sum_{\rho\in\mathcal R}\frac{|\mathcal E_{\rho}|}{(dn)!\prod_{s\in[\mathfrak r]}(2\ell_s)}\textrm{, }
|\mathcal E_{\rho}|
=\sum_{y\in\mathcal Y}e_1e_2e_3\textrm{, }\\
e_1&=e_1(\rho,y)
  =\binom{n}{n_1}(n_1)_{\mathfrak r_1}(n-n_1)_{n(\rho)-\mathfrak r_1}\prod_{j\in[n(\rho)]}(d)_{d_j(\rho)}\textrm{,}\\
e_2&=e_2(\rho,y)=\binom{k}{2}^m(m)_{m(\rho)}\prod_{b\in[m(\rho)]}(2_{\mathfrak o(b)}(k-2)_{k_b(\rho)-\mathfrak o(b)})
\textrm{,}\\
e_3&=e_3(\rho,y)=(dn_1-\mathfrak o)!(d(n-n_1)-(e(\rho)-\mathfrak o))!\textrm{,}
\end{flalign*}
for the following reasons.
With $\rho\in\mathcal R$ and $y\in\mathcal Y(\rho)$ fixed we choose the $n_1$ variables out of the $n$ variables in the configuration $g$ that should take the value one under $x$. Out of these $n_1$ variables we choose the $\mathfrak r_1$ variables (ordered by first traversal) that take the value one in the directed cycles $\gamma$ under $x$, corresponding to the $\mathfrak r_1$ variables in $\rho$ that take one under $y$ (more precisely we choose the values $i\in[n]$ of the absolute values $\alpha_v$ for the $\mathfrak r_1$ variables $j\in[n(\rho)]$ in $\rho$ that take the value one under $y$) and analogously for the variables that take zero.
Then, for each variable $j\in[n(\rho)]$ in $\rho$ and corresponding variable $i=\alpha_v(j)$ in $\gamma$ we choose the $i$-edges that participate in $\gamma$ (meaning that we choose $\alpha_{v,j}$).
On the constraint side we first choose the two $a$-edges that take the value one under $x$ in $g$ for each $a\in[m]$. Then we select the $m(\rho)$ constraints that participate in $\gamma$ (i.e.~we fix $\alpha_f$).
Further, for each constraint $b\in[m(\rho)]$ in $\rho$ and its corresponding constraint $a=\alpha_f(b)$ in $\gamma$ we choose the $\mathfrak o(b)$ $a$-edges that take the value one in $\gamma$ under $x$ consistent with $\rho$ and $y$ out of the two $a$-edges that take the value one in $g$ under $x$ and analogously for the $a$-edges that take the value zero (which means that we fix $\alpha_{f,b}$ for $b\in[m(\rho)]$ consistent with the choice of $y$ and the choice of the two $a$-edges that take the value one for each $a\in[m]$). This fixes the sequence of the directed cycles (i.e.~the isomorphism $\alpha$ and further $\gamma$). The remaining terms wire the $(dn_1-\mathfrak o)$ remaining v-edges that take the value one and the v-edges taking zero respectively.

As opposed to the rather demanding combinatorial part the asymptotics are still easy to derive since both sums are bounded, so the procedure analogous to Section \ref{a_number_of_cycles} yields
\begin{flalign*}
\frac{|\mathcal E_1|}{|\mathcal G|\mathbb E[Z]}
&\sim\sum_{\rho\in\mathcal R}\sum_{y\in\mathcal Y}c_1(\rho,y)n^{n(\rho)+m(\rho)-e(\rho)}\textrm{,}
\end{flalign*}
where $c_1(\rho,y)$ is a constant compensating the bounded terms.
The right hand side tends to zero by the argumentation in Section \ref{a_number_of_cycles}, so this contribution is indeed negligible.
This shows that $\frac{|\mathcal E|}{|\mathcal G|}\sim\frac{|\mathcal E_0|}{|\mathcal G|}$ and thereby establishes Theorem
\ref{small_subgraph_conditioning} (\ref{small_subgraph_conditioning_joint_dist}).

With $d\in[1,d^*)\subseteq[1,k)$ as discussed in Lemma \ref{m_optimization_below_inflection} and Lemma \ref{m_optimization_d_star_le_d_max}, $\lambda_\ell$ as derived in Lemma \ref{lem_number_of_cycles}, $\delta_\ell=(1-k)^{-\ell}$, the asymptotics of the second moment discussed in Lemma \ref{p_second_moment_results} and the Taylor series
$\ln(1-x)= - \sum_{\ell \ge 1} {x^\ell}/{\ell}, x\in(0,1)$,
we establish Theorem \ref{small_subgraph_conditioning} (\ref{small_subgraph_conditioning_variance_explained}) by applying our results to the sum
\begin{flalign*}
\sum_{\ell \ge 1}\lambda_\ell\delta_\ell^2=\sum_{\ell \ge 1}\frac{1}{2\ell}\left(\frac{d-1}{k-1}\right)^\ell
=-\frac{1}{2}\ln\left(1-\frac{d-1}{k-1}\right)
=\ln\left(\sqrt{\frac{k-1}{k-d}}\right)\textrm{.}
\end{flalign*}
This concludes the proof of Theorem \ref{small_subgraph_conditioning} and further the proof of Theorem \ref{theorem_satisfiability_threshold}.
\bibliographystyle{plain}
\bibliography{literature}

\begin{thebibliography}{10}

\bibitem{abbe2018}
E.~Abbe.
\newblock Community detection and stochastic block models: Recent developments.
\newblock {\em Journal of Machine Learning Research}, 18:1--86, 2018.

\bibitem{bahmanian2017}
M.~A. Bahmanian and M.~\v{S}ajna.
\newblock Quasi-{E}ulerian hypergraphs.
\newblock {\em Electron. J. Combin.}, 24(3):Paper 3.30, 12, 2017.

\bibitem{bapst2017}
V.~Bapst, A.~Coja-Oghlan, and C.~Efthymiou.
\newblock Planting colourings silently.
\newblock {\em Combin. Probab. Comput.}, 26(3):338--366, 2017.

\bibitem{bollobas1980}
B.~Bollob\'{a}s.
\newblock A probabilistic proof of an asymptotic formula for the number of
  labelled regular graphs.
\newblock {\em European J. Combin.}, 1(4):311--316, 1980.

\bibitem{braunstein2005}
A.~Braunstein, M.~M\'{e}zard, and R.~Zecchina.
\newblock Survey propagation: an algorithm for satisfiability.
\newblock {\em Random Structures Algorithms}, 27(2):201--226, 2005.

\bibitem{castellani2003}
T.~Castellani, V.~Napolano, F.~Ricci-Tersenghi, and R.~Zecchina.
\newblock Bicolouring random hypergraphs.
\newblock {\em J. Phys. A}, 36(43):11037--11053, 2003.

\bibitem{coja2016a}
A.~Coja-Oghlan.
\newblock Constraint satisfaction: random regular {$k$}-{SAT}.
\newblock In {\em Statistical physics, optimization, inference, and
  message-passing algorithms}, pages 231--251. Oxford Univ. Press, Oxford,
  2016.

\bibitem{coja2018}
A.~Coja-Oghlan, C.~Efthymiou, N.~Jaafari, M.~Kang, and T.~Kapetanopoulos.
\newblock Charting the replica symmetric phase.
\newblock {\em Comm. Math. Phys.}, 359(2):603--698, 2018.

\bibitem{coja2018a}
A.~Coja-Oghlan, F.~Krzakala, W.~Perkins, and L.~Zdeborova.
\newblock Information-theoretic thresholds from the cavity method.
\newblock {\em Adv. Math.}, 333:694--795, 2018.

\bibitem{coja2016}
A.~Coja-Oghlan and K.~Panagiotou.
\newblock The asymptotic {$k$}-{SAT} threshold.
\newblock {\em Adv. Math.}, 288:985--1068, 2016.

\bibitem{cooper1996}
C.~Cooper, A.~Frieze, M.~Molloy, and B.~Reed.
\newblock Perfect matchings in random {$r$}-regular, {$s$}-uniform hypergraphs.
\newblock {\em Combin. Probab. Comput.}, 5(1):1--14, 1996.

\bibitem{asta2008}
L.~Dall'Asta, A.~Ramezanpour, and R.~Zecchina.
\newblock Entropy landscape and non-{G}ibbs solutions in constraint
  satisfaction problems.
\newblock {\em Phys. Rev. E (3)}, 77(3):031118, 16, 2008.

\bibitem{ding2015}
J.~Ding, A.~Sly, and N.~Sun.
\newblock Proof of the satisfiability conjecture for large {$k$}.
\newblock In {\em S{TOC}'15---{P}roceedings of the 2015 {ACM} {S}ymposium on
  {T}heory of {C}omputing}, pages 59--68. ACM, New York, 2015.

\bibitem{ding2016}
J.~Ding, A.~Sly, and N.~Sun.
\newblock Maximum independent sets on random regular graphs.
\newblock {\em Acta Math.}, 217(2):263--340, 2016.

\bibitem{ding2016a}
J.~Ding, A.~Sly, and N.~Sun.
\newblock Satisfiability threshold for random regular {NAE}-{SAT}.
\newblock {\em Comm. Math. Phys.}, 341(2):435--489, 2016.

\bibitem{erdos1960}
P.~Erd\H{o}s and A.~R\'{e}nyi.
\newblock On the evolution of random graphs.
\newblock {\em Magyar Tud. Akad. Mat. Kutat\'{o} Int. K\"{o}zl.}, 5:17--61,
  1960.

\bibitem{erdos1966}
P.~Erd\H{o}s and A.~R\'{e}nyi.
\newblock On the existence of a factor of degree one of a connected random
  graph.
\newblock {\em Acta Math. Acad. Sci. Hungar.}, 17:359--368, 1966.

\bibitem{feige02}
U.~Feige.
\newblock Relations between average case complexity and approximation
  complexity.
\newblock In {\em Proceedings of the Thiry-fourth Annual ACM Symposium on
  Theory of Computing}, STOC '02, pages 534--543, New York, NY, USA, 2002. ACM.

\bibitem{frieze1996}
A.~Frieze, M.~Jerrum, M.~Molloy, R.~W. Robinson, and N.~C. Wormald.
\newblock Generating and counting {H}amilton cycles in random regular graphs.
\newblock {\em J. Algorithms}, 21(1):176--198, 1996.

\bibitem{gabrie2017}
M.~Gabri\'{e}, V.~Dani, G.~Semerjian, and L.~Zdeborov\'{a}.
\newblock Phase transitions in the {$q$}-coloring of random hypergraphs.
\newblock {\em J. Phys. A}, 50(50):505002, 44, 2017.

\bibitem{galanis14}
A.~Galanis, D.~\v{S}tefankovi\v{c}, and E.~Vigoda.
\newblock Inapproximability for antiferromagnetic spin systems in the tree
  non-uniqueness region.
\newblock In {\em Proceedings of the Forty-sixth Annual ACM Symposium on Theory
  of Computing}, STOC '14, pages 823--831, New York, NY, USA, 2014.

\bibitem{higuchi2010}
S.~Higuchi and M.~M\'ezard.
\newblock Correlation-based decimation in constraint satisfaction problems.
\newblock {\em Journal of Physics: Conference Series}, 233(1):012003, 2010.

\bibitem{janson1995}
S.~Janson.
\newblock Random regular graphs: asymptotic distributions and contiguity.
\newblock {\em Combin. Probab. Comput.}, 4(4):369--405, 1995.

\bibitem{janson2000}
S.~Janson, T.~{\L}uczak, and A.~Rucinski.
\newblock {\em Random graphs}.
\newblock Wiley-Interscience Series in Discrete Mathematics and Optimization.
  Wiley-Interscience, New York, 2000.

\bibitem{kahn2022}
J.~Kahn.
\newblock Hitting times for {S}hamir's problem.
\newblock {\em Trans. Amer. Math. Soc.}, 375(1):627--668, 2022.

\bibitem{moore2008}
V.~Kalapala and C.~Moore.
\newblock The phase transition in exact cover.
\newblock {\em Chic. J. Theoret. Comput. Sci.}, pages Article 5, 9, 2008.

\bibitem{kemkes2010}
G.~Kemkes, X.~P\'{e}rez-Gim\'{e}nez, and N.~C. Wormald.
\newblock On the chromatic number of random {$d$}-regular graphs.
\newblock {\em Adv. Math.}, 223(1):300--328, 2010.

\bibitem{krzakala2016}
F.~Krzakala, F.~Ricci-Tersenghi, L.~Zdeborov\'{a}, R.~Zecchina, E.~W. Tramel,
  and L.~F. Cugliandolo, editors.
\newblock {\em Statistical physics, optimization, inference, and
  message-passing algorithms}.
\newblock Oxford University Press, Oxford, 2016.

\bibitem{kudekar13}
S.~Kudekar, T.~Richardson, and R.~Urbanke.
\newblock Spatially coupled ensembles universally achieve capacity under belief
  propagation.
\newblock {\em IEEE Trans. Inform. Theory}, 59:7761--7813, 2013.

\bibitem{maneva2007}
E.~Maneva, E.~Mossel, and M.~J. Wainwright.
\newblock A new look at survey propagation and its generalizations.
\newblock {\em J. ACM}, 54(4):Art. 17, 41, 2007.

\bibitem{mezard2009}
M.~M\'{e}zard and A.~Montanari.
\newblock {\em Information, physics, and computation}.
\newblock Oxford Graduate Texts. Oxford University Press, Oxford, 2009.

\bibitem{mezard1987}
M.~M\'{e}zard, G.~Parisi, and M.~A. Virasoro.
\newblock {\em Spin glass theory and beyond}, volume~9 of {\em World Scientific
  Lecture Notes in Physics}.
\newblock World Scientific Publishing Co., Inc., Teaneck, NJ, 1987.

\bibitem{molloy1997}
M.~Molloy, H.~D. Robalewska, R.~W. Robinson, and N.~C. Wormald.
\newblock {$1$}-factorizations of random regular graphs.
\newblock {\em Random Structures Algorithms}, 10(3):305--321, 1997.

\bibitem{moore2016}
C.~Moore.
\newblock The phase transition in random regular exact cover.
\newblock {\em Ann. Inst. Henri Poincar\'{e} D}, 3(3):349--362, 2016.

\bibitem{mora2007}
T.~Mora.
\newblock {\em {Geometry and Inference in Optimization and in Information
  Theory}}.
\newblock Theses, {Universit{\'e} Paris Sud - Paris XI}, September 2007.

\bibitem{panagiotou2019}
K.~Panagiotou and M.~Pasch.
\newblock {Satisfiability Thresholds for Regular Occupation Problems}.
\newblock In {\em 46th Int.~Col.~on Automata, Languages, and Programming (ICALP
  '19)}, volume 132 of {\em Leibniz International Proceedings in Informatics
  (LIPIcs)}, pages 90:1--90:14, 2019.

\bibitem{robalewska1996}
H.~D. Robalewska.
\newblock {$2$}-factors in random regular graphs.
\newblock {\em J. Graph Theory}, 23(3):215--224, 1996.

\bibitem{robbins1955}
H.~Robbins.
\newblock A remark on {S}tirling's formula.
\newblock {\em Amer. Math. Monthly}, 62:26--29, 1955.

\bibitem{schmidt2016}
C.~Schmidt, N.~Guenther, and L.~Zdeborov\'{a}.
\newblock Circular coloring of random graphs: statistical physics
  investigation.
\newblock {\em J. Stat. Mech. Theory Exp.}, 2016(8):083303, 28, 2016.

\bibitem{talagrand2003}
M.~Talagrand.
\newblock {\em Spin glasses: a challenge for mathematicians}, volume~46 of {\em
  Results in Mathematics and Related Areas. 3rd Series. A Series of Modern
  Surveys in Mathematics}.
\newblock Springer-Verlag, Berlin, 2003.

\bibitem{yedidia2001}
J.~S. Yedidia, W.~T. Freeman, and Y.~Weiss.
\newblock Bethe free energy, kikuchi approximations, and belief propagation
  algorithms.
\newblock Technical Report TR2001-16, MERL - Mitsubishi Electric Research
  Laboratories, Cambridge, MA 02139, May 2001.

\bibitem{zdeborova2011}
L.~Zdeborov\'{a} and F.~Krzakala.
\newblock Quiet planting in the locked constraint satisfaction problems.
\newblock {\em SIAM J. Discrete Math.}, 25(2):750--770, 2011.

\bibitem{zdeborova2008a}
L.~Zdeborov\'a and M.~M\'ezard.
\newblock Constraint satisfaction problems with isolated solutions are hard.
\newblock {\em J. Stat. Mech. Theory Exp.}, 2008(12):P12004, 2008.

\end{thebibliography}
\appendix

%
%
\section{Proof of Lemma \ref{lem_number_of_cycles}}\label{a_number_of_cycles}
We present the proof of Lemma \ref{lem_number_of_cycles} in detail so as to facilitate the presentation of the small subgraph conditioning method in Section \ref{s_small_subgraph_conditioning}.
Lemma \ref{lem_number_of_cycles} can be shown by a direct application of the method of moments, which is discussed, for example, in \cite{janson2000} (Theorem 6.10).
\begin{theorem}[Method of Moments]
\label{method_of_moments}
Let $\bar\ell\in\mathbb Z_{>0}$ and $((X_{\ell,i})_{\ell\in[\bar\ell]})_{i\in\mathbb Z_{>0}}$ be a sequence of a vector of random variables.
If $\lambda\in\mathbb R_{\ge 0}^{\bar\ell}$ is such that, as $i\rightarrow\infty$,
\begin{flalign*}
\mathbb E\left[\prod_{\ell=1}^{\bar\ell}(X_{\ell,i})_{r_\ell}\right]\rightarrow\prod_{\ell=1}^{\bar\ell}\lambda_\ell^{r_\ell}
\end{flalign*}
for every $r\in\mathbb Z_{\ge 0}^{\bar\ell}$, then $(X_{\ell,i})_{\ell\in[\bar\ell]}$ converges in distribution to $(Z_\ell)_{\ell\in[\bar\ell]}$, where
the $Z_\ell\sim\Po(\lambda_\ell)$ are independent Poisson distributed random variables.
\end{theorem}
First, we notice that $G=G_{k,d,n,m}$ and further $X_\ell=X_{k,d,n,\ell}$ is only defined for $m=dn/k\in\mathbb Z$
as stated in Lemma \ref{lem_satthresh_basics}, hence Lemma \ref{lem_number_of_cycles} only applies to such sequences of configurations.

Fix $k$, $d\in\mathbb Z_{>1}$.
Before we turn to the general case we consider the $\mathbb E[X_\ell]$ for $\ell\in\mathbb Z_{>0}$.
For this purpose let $n$ and $m(n)$ be sufficiently large. Let $\mathcal C_{\ell,g}$ be the set of all $2\ell$-cycles in $g\in\mathcal G$. Then  
\begin{flalign*}
\mathbb E[X_\ell]&=\sum_{g\in\mathcal G}\frac{X_\ell(g)}{|\mathcal G|}
=|\mathcal G|^{-1}\sum_{g\in\mathcal G}|\mathcal C_{\ell,g}|
=\frac{|\mathcal E|}{|\mathcal G|}~\textrm{, where }~ \mathcal E=\{(g,c):g\in\mathcal G,c\in\mathcal C_{\ell,g}\}.
\end{flalign*}
With this at hand we obtain that 
\begin{flalign*}
\mathbb E[X_\ell]=\frac{1}{2\ell(dn)!}(n)_\ell(m)_\ell(d(d-1))^\ell(k(k-1))^\ell(dn-2\ell)!
\end{flalign*}
using the following combinatorial arguments.
Instead of counting pairs $(g,c)$ of configurations $g$ and $2\ell$-cycles $c\in\mathcal C_{\ell,g}$ we count pairs $(g,\gamma)$ of configurations $g$ and \emph{directed} $2\ell$-cycles $\gamma$ (based at a variable node) in $g$.
There are exactly $2\ell$ directed cycles $\gamma$ corresponding to each (undirected) cycle $c$ of length $2\ell$ since we can choose the base from the $\ell$ variables in $c$ and $\gamma$ is then determined by one of the two possible directions.
The denominator reflects the compensation for this counting next to the probability $|\mathcal G|^{-1}$.
Further, the term $(n)_\ell$ reflects the ordered choice of the variables for the directed cycle, as does $(m)_\ell$ for the constraints.
The next two terms account for the choice of the two $i$-edges and $a$-edges traversed by the cycle for each of the $\ell$ variables $i$ and constraints $a$. This fixes the directed cycle $\gamma$ and further the corresponding undirected cycle $c(\gamma)$.
In particular, the $2\ell$ edges of the cycle $c$ in $g$ are fixed, i.e.~the corresponding restriction of $g$ to $c$.
This leaves us with $(dn-2\ell)$ half-edges in $\dom(g)$ and $(km-2\ell)$ half-edges in $\im(g)$ that have not been wired yet.
The last term gives the number of such wirings.

Next, we turn to asymptotics. Extracting $\lambda_\ell$ and expanding the falling factorials yields
\begin{flalign*}
\mathbb E[X_\ell]=\lambda_\ell d^\ell k^\ell\frac{n!m!(dn-2\ell)!}{(dn)!(n-\ell)!(m-\ell)!}\textrm{.}
\end{flalign*}
Using Stirling's formula we readily obtain that
\begin{flalign*}
\mathbb E[X_\ell]
&=\lambda_\ell d^\ell k^\ell\sqrt{\frac{nm(dn-2\ell)}{dn(n-\ell)(m-\ell)}}
\frac{n^nm^m(dn-2\ell)^{dn-2\ell}}{(dn)^{dn}(n-\ell)^{n-\ell}(m-\ell)^{m-\ell}},
\end{flalign*}
and so
\begin{flalign*}
\mathbb E[X_\ell]&\sim\lambda_\ell d^\ell k^\ell\sqrt{\frac{(1-\frac{2\ell}{dn})}{(1-\frac{\ell}{n})(1-\frac{\ell}{m})}}
\frac{n^{\ell}m^{\ell}(1-\frac{2\ell}{dn})^{dn-2\ell}}{(dn)^{2\ell}(1-\frac{\ell}{n})^{n-\ell}(1-\frac{\ell}{m})^{m-\ell}}
\sim\lambda_\ell d^\ell k^\ell\frac{n^{\ell}m^{\ell}}{(dn)^{2\ell}}.
\end{flalign*}
Using that $dn = km$ leads to
\begin{flalign*}
\mathbb E[X_\ell]&\sim\lambda_\ell d^\ell k^\ell\frac{n^{\ell}(dk^{-1}n)^{\ell}}{(dn)^{2\ell}}=\lambda_\ell\textrm{,}
\end{flalign*}
as claimed.
We turn to the general case.
For this purpose let $\bar\ell\in\mathbb Z_{>0}$, $r\in\mathbb Z_{\ge 0}^{\bar\ell}$ and let $n$ and $m$ be sufficiently large.
Then
\begin{flalign*}
(X_\ell(g))_{r_\ell}&=\prod_{s=0}^{r_\ell-1}(|\mathcal C_{\ell,g}|-s)=|\mathcal C_{\ell,r_\ell,g}|~\textrm{, where }~
\mathcal C_{\ell,r_\ell,g} =\{c\in\mathcal C_{\ell,g}^{r_\ell}:\forall s\in[r_\ell]\forall s'\in[s-1]\,c_s\neq c_{s'}\}
\end{flalign*}
for $g\in\mathcal G$, since this corresponds to an ordered choice of $2\ell$-cycles in $g$ without repetition.
The product can then be directly written as
\begin{flalign*}
\prod_{\ell=1}^{\bar\ell}(X_\ell(g))_{r_\ell}=|\mathcal C_{r,g}|~~\textrm{, where }~~
\mathcal C_{r,g}=\prod_{\ell=1}^{\bar\ell}\mathcal C_{\ell,r_\ell,g}\textrm{.}
\end{flalign*}
To avoid double indexed sequences we use the equivalent representation $c=(c_s)_{s\in[\mathfrak r]}\in\mathcal C_{r,g}$ where
$\mathfrak r=\sum_{1\le \ell \le \bar\ell} r_\ell$.
From the above we see that the cycles $c_s$ are ordered by their length $\ell_s$ in ascending order and are pairwise distinct. We obtain that
\begin{flalign*}
\mathbb E\left[\prod_{\ell=1}^{\bar\ell}(X_\ell)_{r_\ell}\right]=\frac{|\mathcal E|}{|\mathcal G|}~~\textrm{, where }~~
\mathcal E=\{(g,c):g\in\mathcal G,c\in\mathcal C_{r,g}\}.
\end{flalign*}
Since we have $\ell_s$ distinct variables and constraints in each cycle $c_s$ respectively, we can have at most $\mathfrak l=\sum_{s\in[\mathfrak r]}\ell_s$ distinct variables and constraints in $c$. Specifically, we only have $|V(c)|=\mathfrak l$ variables and $|F(c)|=\mathfrak l$ constraints iff all cycles $c_s$ are disjoint.
So, let
\begin{flalign*}
\mathcal E_0=\{(g,c)\in\mathcal E:|V(c)|=|F(c)|=\mathfrak l\}
\end{flalign*}
denote the set of pairs $(g,c)\in\mathcal E$ with disjoint cycles and further $\mathcal E_1=\mathcal E\setminus\mathcal E_0$ the remaining pairs.
Then we have
\begin{flalign*}
\frac{|\mathcal E_0|}{|\mathcal G|}
=\frac{1}{(dn)!\prod_{s=1}^{\mathfrak r}(2\ell_s)}(n)_{\mathfrak l}(m)_{\mathfrak l}(d(d-1)^{\mathfrak l}(k(k-1))^{\mathfrak l}(dn-2\mathfrak l)!
\end{flalign*}
for the following reasons.
For each cycle $c_s$ in $c$ counting the $2\ell_s$ directed cycles facilitates the computation, hence we find the corresponding product in the denominator.
Since the variables within each directed cycle and the cycles in the sequence are ordered we have an ordered choice of all variables.
Further, since the $\ell_s$ variables within each cycle are distinct and the cycles are pairwise disjoint we choose all variables without repetition. This explains the first falling factorial. The next term for the constraints follows analogously.
But since variables and constraints are disjoint the edges are too, hence we choose two edges for each of the $\mathfrak l$ variables and constraints respectively. Then we wire the remaining edges.\\
The asymptotics are derived analogously to the base case, i.e.
\begin{flalign*}
\frac{|\mathcal E_0|}{|\mathcal G|}
\sim\frac{(d-1)^{\mathfrak l}(k-1)^{\mathfrak l}}{\prod_{s=1}^{\mathfrak r}(2\ell_s)}
=\prod_{s=1}^{\mathfrak r}\lambda_{\ell_s}=\prod_{\ell=1}^{\bar \ell}\lambda_\ell^{r_\ell}\textrm{,}
\end{flalign*}
using the definition of $c=(c_s)_{s\in[\mathfrak r]}$ in the last step.
Since the contribution of the disjoint cycles already yields the desired result, we want to show that the contribution of intersecting cycles is negligible.
As before, we count directed cycles $\gamma_s$ and adjust the result accordingly, so let
\begin{flalign*}
\mathcal E_2=\{(g,\gamma):(g,c(\gamma))\in\mathcal E_1\}\textrm{, i.e. }|\mathcal E_2|=|\mathcal E_1|\prod_{s\in[\mathfrak r]}(2\ell_s)\textrm{.}
\end{flalign*}
In the next step we consider the \emph{relative position representations} $(\alpha,\rho)$ of sequences $\gamma$ of directed cycles.
Instead of a formal introduction we illustrate this concept in Figure \ref{f_relative_position_representation}.
The corresponding decomposition of the contributions to the expectation according to $\rho$ is
\begin{flalign*}
\frac{|\mathcal E_1|}{|\mathcal G|}
&=\sum_{\rho\in\mathcal R}
  \frac{|\mathcal E_{\rho}|}{|\mathcal G|\prod_{s\in[\mathfrak r]}(2\ell_s)}\textrm{, }
\mathcal E_{\rho}=\{(g,\gamma)\in\mathcal E_2:\rho(\gamma)=\rho\}\textrm{, }
\mathcal R=\{\rho(\gamma):(g,\gamma)\in\mathcal E_2\}\textrm{.}
\end{flalign*}
For the following reasons we can then derive
\begin{flalign*}
|\mathcal E_{\rho}|
=(n)_{n(\rho)}(m)_{m(\rho)}
\prod_{j\in[n(\rho)]}(d)_{d_j(\rho)}\prod_{b\in[m(\rho)]}(k)_{k_b(\rho)}(dn-e(\rho))!\textrm{.}
\end{flalign*}
Since $\rho$ is fixed, we have to fix the absolute values $\alpha$, thereby the directed cycle $\gamma$, and wire the remaining edges.
But the first four terms exactly correspond to the number of choices for the index vectors in $\alpha$.
This fixes $\gamma$, further the union $c(\gamma)$ of cycles and in particular $e(\rho)$ edges.
The remaining term counts the number of choices to wire the remaining edges.

For the asymptotics we notice that $n(\rho)$, $m(\rho)\le\mathfrak l$ and that also the two products are bounded in both the multiplication region and values. But this further implies that $|\mathcal R|$ is bounded, i.e.~the summation region is also finite in the limit and hence we can consider the asymptotics of each term separately, which yields
\begin{flalign*}
\frac{|\mathcal E_1|}{|\mathcal G|}
&=\sum_{\rho\in\mathcal R}
\frac{\prod_{i\in[n(\rho)]}(d)_{d_i(\rho)}\prod_{a\in[m(\rho)]}(k)_{k_a(\rho)}}{\prod_{s\in[\mathfrak r]}(2\ell_s)}
  \frac{(n)_{n(\rho)}(m)_{m(\rho)}(dn-e(\rho))!}{(dn)!}\\
&=\sum_{\rho\in\mathcal R}c_1(\rho)
  \frac{(n)_{n(\rho)}(m)_{m(\rho)}(dn-e(\rho))!}{(dn)!}\\
&\sim\sum_{\rho\in\mathcal R}c_1(\rho)
\left(\frac{1}{e}\right)^{n(\rho)}
\left(\frac{d}{ke}\right)^{m(\rho)}
\left(\frac{e}{d}\right)^{e(\rho)}
n^{n(\rho)+m(\rho)-e(\rho)}\\
&=\sum_{\rho\in\mathcal R}c_2(\rho)
n^{n(\rho)+m(\rho)-e(\rho)},
\end{flalign*}
where we summarized the terms that only depend on $\rho$ into constants.
Now, let $\rho\in\mathcal R$ and let $c=c(\rho)$ be the graph of $\rho$ as introduced in Section \ref{p_small_subgraph_conditioning}.
Since $\rho$ is a sequence of directed cycles that are not all disjoint,
its graph $c$ is the union of the corresponding (undirected) cycles that are not all disjoint.
But then $c$ has more edges than vertices, i.e. $3e(\rho)>n(\rho)+m(\rho)+2e(\rho)$, and hence
\begin{flalign*}
\frac{|\mathcal E_1|}{|\mathcal G|}
&\sim\sum_{\rho\in\mathcal R}c_2(\rho)
n^{n(\rho)+m(\rho)-e(\rho)}
\le n^{-1}\sum_{\rho\in\mathcal R}c_2(\rho)
=c_3n^{-1}\textrm{,}
\end{flalign*}
which shows that this contribution is negligible. This establishes the asymptotic equivalence
\begin{flalign*}
\mathbb E\left[\prod_{\ell\in[\bar \ell]}(X_\ell)_{r_\ell}\right]\sim\prod_{\ell\in[\bar \ell]}\lambda_\ell^{r_\ell}
\end{flalign*}
and allows to apply the method of moments, which directly yields Lemma \ref{lem_number_of_cycles}.
\end{document}